\documentclass[11pt,a4paper]{amsart}

\title{Overconvergent Hilbert modular forms via perfectoid modular varieties}

\usepackage{etex}
\usepackage[english]{babel}
\usepackage[all,cmtip]{xy}
\usepackage{enumerate}
\usepackage{dsfont}
\usepackage{ marvosym }
\usepackage{mathrsfs}
\usepackage{setspace}
\usepackage{longtable}
\usepackage{tabularx, tabu}
\usepackage{tikz}
\usepackage{tikz-cd}
\usetikzlibrary{matrix,arrows}
\usepackage{framed, extarrows}
\usepackage{changes,scalerel}
\usepackage{amsfonts,amsthm,amsmath,amssymb,comment,stmaryrd,mathalfa,url,mathrsfs, yfonts}
\usepackage{graphicx}
\graphicspath{ {images/} }
\usepackage[section]{placeins}
\usepackage{microtype}
\usepackage{booktabs, calc, multirow}
\usepackage{array}
\usepackage[ocgcolorlinks,backref=page]{hyperref}
\hypersetup{citecolor=blue,
	linkcolor=red}
\usepackage{amscd, fix-cm}
\usepackage{amsmath}
\usepackage{makecell}
\usepackage{chngcntr,todonotes}

\usetikzlibrary{%
	matrix,%
	calc,%
	arrows%
}

\usepackage[margin=1 in]{geometry}
\tikzset{
	labl/.style={anchor=south, rotate=90, inner sep=.5mm}
}

\author{Christopher Birkbeck}	
\address{Department of Mathematics, University College London, Gower street, London WC1E 6BT, UK}
\email{c.birkbeck@ucl.ac.uk} 

\author{Ben Heuer}
\address{Mathematical Institute, University of Bonn, Endenicher Allee 60, 53012 Bonn, Germany}
\email{heuer@math.uni-bonn.de}

\author{Chris Williams}
\address{Mathematics Institute, University of Warwick, Coventry CV4 7AL, UK}
\email{christopher.d.williams@warwick.ac.uk}

\usepackage{etoolbox}
\usepackage{comment}
\usepackage{verbatim}
\newtoggle{comments}
\toggletrue{comments}

\iftoggle{comments}{
	
	\definecolor{mygreen}{RGB}{253,106,2}
	\newcommand{\chris}[1]{{\color{black} \sf
			$\spadesuit$ Chris B: [#1]}}   
	\newcommand{\benwithoutthepicture}[1]{{\color{blue} \sf $\diamondsuit$ Ben: [#1]}} 
	\newcommand{\ben}[1]{\benwithoutthepicture{#1}}
	\newcommand{\CWnote}[1]{{\color{mygreen} \sf
			$\clubsuit$ Chris W: [#1]}}
	
}{
	
	\newcommand{\chris}[1]{}
	\newcommand{\ben}[1]{}
	\newcommand{\benwithoutthepicture}[1]{} 
	\newcommand{\CWnote}[1]{}
	
}

\theoremstyle{plain}
\newtheorem{thm}{Theorem}[section]
\newtheorem*{thm*}{Theorem}
\newtheorem{lem}[thm]{Lemma}
\newtheorem{cor}[thm]{Corollary}
\newtheorem{corollary}[thm]{Corollary}

\newtheorem*{con*}{Conjecture}
\newtheorem{prop}[thm]{Proposition}
\newtheorem*{prop*}{Proposition}
\newtheorem{proposition}[thm]{Proposition}
\newtheorem{theorem}[thm]{Theorem}
\newtheorem{lemma}[thm]{Lemma}
\newtheorem*{claim}{Claim}

\theoremstyle{definition}
\newtheorem{defn}[thm]{Definition}

\newtheorem{definition}[thm]{Definition}
\newtheorem{rmrk}[thm]{Remark}
\newtheorem*{rmrk*}{Remark}
\newtheorem{remark}[thm]{Remark}
\newtheorem{warn}[thm]{Warning}
 
\newtheorem*{exmp*}{Example}

\newtheorem*{obs*}{Observation}
\newtheorem{nota}[thm]{Notation}

\numberwithin{equation}{section}

\newcommand{\Q}{\mathbb{Q}}
\renewcommand{\O}{\mathcal{O}}
\newcommand{\Tr}{\operatorname{Tr}}

\renewcommand{\c}{\mathfrak c}

\newcommand{\Z}{\mathbb{Z}}
\newcommand{\ZZ}{\mathbb{Z}}
\newcommand{\Zp}{\mathbb{Z}_p}
\newcommand{\Zpthick}{\ZZ_p^\times}

\newcommand{\RR}{\mathbb{R}}

\newcommand{\A}{\mathbb{A}}

\newcommand{\C}{\mathbb{C}}

\newcommand{\N}{\mathbb{N}}
\newcommand{\F}{\mathbb{F}}

\newcommand{\TT}{\mathbb{T}}

\def\AA{\mathbb{A}}
\def\CC{\mathbb{C}}

\def\GG{\mathbb{G}}
\def\NN{\mathbb{N}}
\renewcommand{\P}{\mathbb{P}}
\renewcommand{\A}{\mathbb{A}}
\newcommand{\G}{\mathbb{G}}
\def\PP{\mathbb{P}}
\def\QQ{\mathbb{Q}}
\def\TT{\mathbb{T}}

\newcommand{\T}{\mathcal{T}}

\newcommand{\roi}{\mathcal{O}}

\newcommand{\cA}{\mathcal{A}}
\newcommand{\cU}{\mathcal{U}}

\newcommand{\cX}{\mathcal{X}}
\newcommand{\cO}{\mathcal{O}}
\newcommand{\cG}{\mathcal{G}}
\newcommand{\cY}{\mathcal{Y}}

\newcommand{\cZ}{\mathcal{Z}}

\newcommand{\W}{\mathcal{W}}

\def \Y {\mathcal{Y}}
\def \XX {\mathcal{X}}
\def \XXt {\mathcal{X}^{\ast}}

\def \ed {\epsilon_{\k}^{\mathrm{def}}}

\newcommand{\n}{N}
\newcommand{\fp}{\mathfrak{p}}

\def\gothc{\mathfrak{c}}

\def \nn {N}

\def\gothp{\mathfrak{p}}
\def\ps{\mathfrak{p}}

\def \proet {\mathrm{pro\acute{e}t}}

\newcommand{\pol}{\mathrm{pol}}

\def \AL {\mathrm{AL}}

\def \a{\alpha}
\def \e {\epsilon}

\def \g {\gamma}
\def \GA {\Gamma}
\def \w {\omega}

\def \W {\mathcal{W}}

\def \YY {\Y}

\def \GX {\mathfrak{X}}

\def \v {v}
\def \lam {\lambda}

\def \d {\mathfrak{d}}
\def \OO {\mathcal{O}}
\def \OF {\mathcal{O}_F}
\def \OFU {\mathcal{O}_F^{\times,+}}

\def \Z {\ZZ}

\def \PGI {\mathrm{P}\Gamma_{\!\scaleto{0}{3pt}}(p^n)}
\def \GI {\Gamma_{\!\scaleto{0}{3pt}}(p^n)}
\def \GIst {\Gamma_{\!\scaleto{0}{3pt}}^*(p^n)}

\def \wh{\widehat}

\def \HHH {\mathcal{H}}
\def \U {\mathcal{U}}

\def \eU {\e_\k}
\def \kU {\k}
\def \wU {w_{\scaleto{\k}{4 pt}}}
\def \rU {t_{\scaleto{\k}{4 pt}}}

\def \k {\kappa}
\def \Norm {N_{F/\QQ}}
\def \Ha {\mathrm{Ha}}
\def \tHa {\widetilde{\Ha}}

\DeclareMathOperator{\coker}{coker}
\DeclareMathOperator{\Sch}{Sch}
\DeclareMathOperator{\Set}{Set}
\DeclareMathOperator{\hH}{H}
\DeclareMathOperator{\tor}{tor}

\DeclareMathOperator{\End}{End}
\DeclareMathOperator{\Lie}{Lie}

\DeclareMathOperator{\Spa}{Spa}
\DeclareMathOperator{\Spf}{Spf}
\DeclareMathOperator{\Spec}{Spec}

\newcommand{\an}{\mathrm{an}}
\newcommand{\ad}{\mathrm{ad}}
\newcommand{\Ig}{\mathrm{Ig}}

\newcommand{\id}{\operatorname{id}}

\DeclareMathOperator{\ttr}{Tr}

\DeclareMathOperator{\Tor}{Tor}
\newcommand{\HT}{\mathrm{HT}}
\newcommand{\AIP}{\mathrm{AIP}}

\DeclareMathOperator{\Hdg}{Hdg}

\newcommand{\cyc}{{\text{cyc}}}
\newcommand{\cts}{{\text{cts}}}

\DeclareMathOperator{\SL}{SL}
\DeclareMathOperator{\Sh}{Sh}

\DeclareMathOperator{\GL}{GL}
\DeclareMathOperator{\GSp}{GSp}

\DeclareMathOperator{\Hom}{Hom}

\DeclareMathOperator{\Res}{Res}

\DeclareMathOperator{\Cl}{Cl}

\DeclareMathOperator{\Aut}{Aut}
\DeclareMathOperator{\Mapc}{\operatorname{Map}_{\cts}}

\def \lra{\longrightarrow}

\newcommand\isorightarrow{\xrightarrow{
		\,\smash{\raisebox{-0.65ex}{\ensuremath{\scriptstyle\sim}}}\,}}

\newcommand{\smallmatrd}[4]{\left(\begin{smallmatrix}#1&#2\\#3&#4\end{smallmatrix}\right)}
\newcommand{\smallvector}[2]{\left(\begin{smallmatrix}#1\\#2\end{smallmatrix}\right)}

\newcommand{\PG}{\overline{\mathrm{P}\Gamma}}
\newcommand{\Gammabar}{\overline{\Gamma}}

\DeclareMathSizes{10.95}{10.95}{7}{6}

\begin{document}
	\maketitle
	
	%
	%
	\begin{abstract}
		We give a new construction of $p$-adic overconvergent Hilbert modular forms by using Scholze's perfectoid Shimura varieties at infinite level and the Hodge--Tate period map. The definition is analytic, closely resembling that of complex Hilbert modular forms as holomorphic functions satisfying a transformation property under congruence subgroups.  As a special case, we first revisit the case of elliptic modular forms, extending recent work of Chojecki, Hansen and Johansson. We then construct sheaves of geometric Hilbert modular forms, as well as subsheaves of integral modular forms, and vary our definitions in $p$-adic families. We show that the resulting spaces are isomorphic as Hecke modules to earlier constructions of Andreatta, Iovita and Pilloni. Finally, we give a new direct construction of sheaves of arithmetic Hilbert modular forms, and compare this to the construction via descent from the geometric case.

	\end{abstract}
	\setcounter{tocdepth}{2}

	%
	%
	
	\section{Introduction}
	
	  In a first introduction, modular forms are usually defined as certain holomorphic functions $f:\HHH\to \CC$ on the complex upper half-plane $\HHH$ satisfying a transformation property
	of the form
	\[\gamma^\ast f = (cz+ d)^kf\quad \forall \gamma = \big(\begin{smallmatrix}a & b\\ c & d\end{smallmatrix}\big)\in \Gamma,\]
	where $\Gamma \subset \SL_2(\Z)$ is a congruence subgroup.
	More algebraically, one can consider modular forms as sections of an automorphic line bundle $\omega$ on the complex modular curve $\Gamma\backslash\HHH$. 
	This algebraic definition admits a $p$-adic interpretation, giving rise to the theory of overconvergent modular forms varying in $p$-adic families, which has proved extremely important with wide-ranging applications in algebraic number theory and arithmetic geometry. 
	
	An analytic definition of $p$-adic overconvergent modular forms has, however, proved elusive, until such an approach was recently introduced in the case of rational quaternionic modular forms by Chojecki, Hansen and Johansson \cite{CHJ}.
	
	 In this article, we give an analytic definitions of both arithmetic and geometric $p$-adic Hilbert modular forms over any totally real field $F$, for any prime $p$, and show that they agree with earlier algebraic definitions of Andreatta--Iovita--Pilloni in \cite{AIP3}.\\

	Following \cite{CHJ}, the key idea of the construction is to use Scholze's perfectoid Shimura varieties at infinite level over a perfectoid base field $L$, and the associated Hodge--Tate period map $\pi_{\HT}$, all introduced in \cite{torsion}. 
These spaces can be viewed as $p$-adic analogues of $\HHH$. In the complex situation, the pullback of the automorphic bundle $\omega$ along the covering map 
\[\HHH\to \Gamma\backslash\HHH\]
can be canonically trivialised, and the descent to $\Gamma\backslash\HHH$  via the action of $\Gamma$ gives rise to the usual definition of complex modular forms, at least after dealing with compactifications.

 Similarly, in the $p$-adic situation, there is an adic analytic moduli space $\mathcal X$, which in our case is a Hilbert modular variety, carrying an automorphic bundle $\omega$. It has a cover
 \[\mathcal X_{\Gamma(p^\infty)}\to \mathcal X \]
  by a perfectoid Hilbert moduli space. Using $\pi_{\HT}$, the pullback of $\omega$ along this projection can be canonically trivialised over open subspaces. Via the action of the associated covering group -- a $p$-adic level subgroup -- one obtains a definition of overconvergent Hilbert modular forms.

More precisely, there are two different kinds of Hilbert modular forms: there are those associated to the group $G := \mathrm{Res}_{F/\Q}\GL_2$, which are called \emph{arithmetic}; and those associated to $G^* :=G \times_{\Res_{F/\QQ} \mathbb{G}_m} \mathbb{G}_m$, which are called \emph{geometric}. 
Shimura varieties for $G^*$ have a moduli interpretation in terms of abelian varieties with PEL structure. For any (narrow) ideal class $\mathfrak c\in \Cl^+(F)$, we consider the $\mathfrak c$-polarised finite level Shimura variety $\mathcal X_{\mathfrak c}$. Since the Shimura variety for $G^{\ast}$ is in particular of Hodge type, one gets an associated infinite level  Hilbert moduli space $\mathcal{X}_{\mathfrak{c},\Gamma^*(p^\infty)}$ in the limit over level structures at $p$.
Inside of this we have for any small enough $\epsilon \geq 0$ an open subspace $\mathcal{X}_{\mathfrak{c},\Gamma^*(p^\infty)}(\epsilon)_a$, the $\epsilon$-overconvergent anticanonical locus. 

\begin{definition}\label{def:geometric}
	 For any weight character $\kappa:\Z_p^\times\to L^{\times}$, a \emph{geometric overconvergent Hilbert modular form of weight $\kappa$} is a function
	$f \in \mathcal{O}(\mathcal{X}_{\mathfrak{c},\Gamma^*(p^\infty)}(\epsilon)_a)$ satisfying
	\begin{equation}\label{eq:intro equation}
	\gamma^*f = \kappa^{-1}(c\mathfrak{z}+d)f\quad \forall \gamma = \big(\begin{smallmatrix}a & b\\ c & d\end{smallmatrix}\big)\in \Gamma_0^*(p),
	\end{equation}
	where $\Gamma_0^*(p)$ is a $p$-adic level subgroup, and $\kappa(c\mathfrak{z}+ d)$ is a factor of automorphy to be defined. 
\end{definition}
	
	From modular forms for $G^{\ast}$, one can obtain an indirect definition of modular forms for $G$ by descent.	
	As we shall show, one of the advantages of the analytic approach is that instead, one can also work with perfectoid Shimura varieties attached to $G$, and give a completely intrinsic definition. Let $\XX_{G,\mathfrak c}$ be the $\mathfrak c$-polarised Shimura variety for $G$. This is now not a fine moduli space of abelian varieties, but one can still construct a perfectoid cover
	\[ \mathcal{X}_{G,\mathfrak{c},\Gamma(p^\infty)}\to \mathcal X_{G,\mathfrak c}.\]
	
	\begin{definition}\label{def:arithmetic}
	An \emph{arithmetic overconvergent Hilbert modular form of weight $\kappa$} is a function $f\in \O(\mathcal{X}_{G,\mathfrak{c},\Gamma(p^\infty)}(\epsilon)_a)$ satisfying 
	\[\gamma^*f = \kU^{-1}(c\mathfrak z+d)\wU(\det \gamma)f\quad \forall \gamma = \big(\begin{smallmatrix}a & b\\ c & d\end{smallmatrix}\big)\in	\mathrm P\Gamma_0(p)\]
	where $\mathrm P\Gamma_0(p)$ is a $p$-adic level subgroup, and $\wU$ is a character to be defined. 
	\end{definition}
	
	More generally, one can similarly define line bundles $\omega^\kappa$ whose global sections are the modular forms of Definitions \ref{def:geometric} and \ref{def:arithmetic}. These bundles, and hence the modular forms, vary naturally over $p$-adic families $\cU$ in the respective weight spaces, by considering analytic functions on the sousperfectoid space $\mathcal{X}_{\mathfrak{c},\cU,\Gamma^*(p^\infty)}(\epsilon)_a := \mathcal{X}_{\mathfrak{c},\Gamma^*(p^\infty)}(\epsilon)_a \times_L \cU$ (and analogously for $G$).

	\subsection{What is new}

	Several constructions of both geometric and arithmetic $p$-adic overconvergent Hilbert modular forms have already appeared in the literature  (e.g.\ amongst others \cite{Kisin-Lai}, \cite{AIP3} and most generally \cite{AIP}),  so let us explain how our constructions differ and what, in our opinion, are some of their advantages.

	\begin{itemize}
		\item Our main goal is to give a new intrinsic definition of the sheaf of arithmetic Hilbert modular forms for $G$, which is arguably cleaner and easier to work with than the one via descent from $G^{\ast}$.  
		\item We also show how to define subspaces of integral geometric and arithmetic Hilbert modular forms in the analytic setting, which match up with the ones constructed in \cite{AIP}. An advantage of the perfectoid construction is that this does not require formal models. Rather, the subspace of integral forms is given by simply replacing $\OO$ with the integral subsheaf $\OO^+$ in the construction.
		\item As in \cite{CHJ}, the resulting framework is well-adapted to constructing overconvergent Eichler--Shimura maps from overconvergent cohomology, namely maps of the form
		\[H^g_{c}(\mathcal X_{\mathfrak c},D_\kappa)\to H^0(\mathcal X^{\tor}_{\mathfrak c}(\epsilon),\omega^{\kappa}\otimes \Omega_{\mathcal X^{\tor}_{\mathfrak c}(\epsilon)}^g(-\partial))(-g)\]
		where $D_\kappa$ is an \'etale sheaf of distribution modules, $g=[F:\QQ]$, and $\partial\subseteq \mathcal X^{\tor}_{\mathfrak c}(\epsilon)$ is the boundary of a chosen toroidal compactification. A proof of this will be included in upcoming work.
	\end{itemize}
	
	As a secondary goal, we modify the strategy of \cite{CHJ} in several ways: 
	\begin{itemize}
		\item We work with the anticanonical locus rather than the canonical one, which makes it easier to deal with the boundary of the Shimura varieties, an issue which is is not present in \textit{op.cit.} as there the construction is carried out for quaternionic modular forms. 
		
		As a minor but pleasant side effect, this results in the automorphic factor $\kappa(c\mathfrak{z}+d)$ appearing in \eqref{eq:intro equation}, like in the complex case,  rather than the $\kappa(b\mathfrak{z}+d)$ from \emph{op.cit}.
		\item \textcolor{black}{	We give a conceptually new proof that the resulting sheaves are line bundles.}
		\item Throughout we  work with sousperfectoid spaces, a language that was not available at the time that \cite{CHJ} was written. This allows one to define automorphic sheaves uniformly for arbitrary bounded weights $\mathcal U$ in a geometric way, by working over the fibre product of the infinite level modular variety with $\cU$. In particular, one does not have to impose restrictions on the shape of weights as in the construction using formal models.
		\item We also explain how the ``perfect" modular forms of \cite{AIP} appear in this anticanonical setting. In the elliptic case, this point of view has been used in \cite{heuer-thesis} for a perfectoid approach to Coleman's Spectral Halo, and one should be able to use our constructions to obtain similar results for Hilbert modular forms.
	\end{itemize}	
	
	We shall now explain our constructions and the organisation of the paper in some more detail.
	\subsection{Elliptic modular forms via the anticanonical locus}
	While the main focus of this paper is to construct families of Hilbert modular forms, we start in \S\ref{sec:HT for MC}, \ref{sec:mod forms via ac locus} and \ref{sec:AIP elliptic}, by treating the elliptic case.  One reason to consider this seperately is that while the boundary in the higher dimensional case can be dealt with via Koecher's principle, in the elliptic case it requires an explicit analysis. Our second reason to treat the elliptic case separately is to illustrate the ways in which we deviate from the construction in \cite{CHJ}.

	To explain this, we first summarise their construction. Let $L$ be any perfectoid field over $\QQ_p^\cyc$, let $\XX^{\ast}$ be the (compact) modular curve over $L$ of some tame level considered as an adic space, and let $q:\mathcal{X}^*_{\Gamma(p^\infty)}\to \XX^{\ast}$ be Scholze's infinite level perfectoid modular curve (denoted $\cX_{\infty}$ \emph{op.\ cit}.). It admits a Hodge--Tate period map $\pi_{\HT} : \XX^{\ast}_{\Gamma(p^\infty)} \rightarrow \PP^1$ with the key property \[\pi_{\HT}^{\ast}\O(1)=q^{\ast}\omega.\]
	To study this sheaf, they consider a family of open subspaces of $\PP^1$, parametrised by $w \in \QQ_{>0}$, on which $\O(1)$ admits a non-vanishing section. Pulling back under $\pi_{\HT}$ gives a family of neighbourhoods $\XX^{\ast}_{\Gamma(p^\infty),w} \subset \XX^{\ast}_{\Gamma(p^\infty)}$ of the (canonical) ordinary locus. There are then subspaces $\XX^{\ast}_w\subseteq \XX^{\ast}$ such that $\XX^{\ast}_{\Gamma(p^\infty),w} \to \XX^{\ast}_{w}$ is a pro-\'etale $\Gamma_{\!0}(p)$-torsor, at least away from the cusps. Here $\Gamma_{\!0}(p) \subset \GL_2(\Zp)$ is the subgroup of matrices that are upper-triangular modulo $p$. 
	
	Pulling back the natural parameter at $\infty\in \PP^1$, they obtain a parameter $\mathfrak{z}\in\O(\XX^{\ast}_{\Gamma(p^\infty),w})$.
	For a certain class of $p$-adic weights $\kappa$, and $w\geq 0$ sufficiently small, they then define the space of ``$w$-overconvergent'' modular forms of weight $\kappa$ to be the set of $f \in \mathcal{O}(\XX^{\ast}_{\Gamma(p^\infty),w})$ satisfying
	\[
	\gamma^*f = \kappa(b\mathfrak{z}+ d)^{-1}f\quad \forall \gamma = \smallmatrd{a}{b}{c}{d} \in \Gamma_{\!0}(p).
	\]

	\subsubsection{The case of elliptic modular forms}
	The results of \cite{CHJ} are only explicitly proved in the quaternionic case where the Shimura curve is compact, though they do mention that their methods can be extended to the elliptic case, where one can use ``soft'' techniques to deal with ramification at the boundary. This is also noted in \cite{howe-thesis}. We expand on these remarks and explain how to extend to the cusps by using perfectoid Tate curve parameter spaces. 
	
	
	Instead of considering the canonical locus, we choose to work with the anticanonical locus $\mathcal{X}^*_{\Gamma(p^\infty)}(\epsilon)_a$ everywhere.
	This definition is equivalent, since the two loci can be interchanged via the action of the Atkin--Lehner matrix $\smallmatrd{0}{1}{p}{0}$.
	We define a sheaf  
	\[
	\omega^{\kappa}_1 := \big\{f \in q_{\ast}\mathcal{O}_{\mathcal{X}^*_{\Gamma(p^\infty)}(\epsilon)_a} \big| \gamma^*f = \kappa^{-1}(c\mathfrak{z}+ d)f \quad\forall \gamma = \smallmatrd{a}{b}{c}{d} \in \Gamma_{\!0}(p)\big\}
	\]
	on $\cX^*_{\Gamma_{\!0}(p)}(\e)_a$, where now $\mathfrak{z}$ is the parameter on $\mathcal{X}^*_{\Gamma(p^\infty)}(\epsilon)_a$ defined by pulling back the canonical parameter on $\AA^{1,\mathrm{an}}\subseteq \PP^1$ at $0$, 
	and $q : \cX_{\Gamma(p^\infty)}^* \to \cX_{\Gamma_{\!\scaleto{0}{3pt}}(p)}^*$ is the projection. The space of $\epsilon$-overconvergent modular forms is then the space of global sections of this sheaf. We note that this is very similar to the complex definition. We then use the Atkin--Lehner isomorphism to obtain a sheaf $\omega^{\k} = \omega^{\k}_0 := \AL^*\omega_1^\kappa$ on the tame level space $\cX^*(p^{-1}\e)$.
	One could now prove, as in \cite{CHJ}, that the sheaf $\omega^{\kappa}$ is a line bundle, but we instead deduce this from our later comparison results.
	
	\subsubsection{Variation in families}
	Reinterpreting \cite{CHJ} in the context of sousperfectoid spaces, we show how to extend the definition to also work for $p$-adic families. The weights considered above can be considered as points $\Spa(L,\O_L) \xrightarrow{\kappa} \W$ in the rigid analytic \emph{weight space} 
	\[\W:= \Spf(\Z_p[[\Z_p^\times]])^{\ad}_\eta \times L,\]
	where throughout we consider all rigid spaces as adic spaces, and where $L$ is our perfectoid field.
	 We can then consider more general (families of) weights $\kappa : \U\to \W$, where $\U$ is a smooth rigid space and $\kappa$ has bounded image.
	This gives rise to a sheaf $\omega^{\kappa}$ on the fibre product $\mathcal{X}^\ast(p^{-1}\epsilon)\times_L \U$ whose fibre over any point $\kappa_0\in\U$ is the sheaf $\omega^{\kappa_0}$ defined above.
	
	By comparing the anticanonical locus with the Pilloni-torsor as described in \cite{AIP2}, in Thm.\ \ref{thm:comparison 1} we construct an isomorphism between $\omega^{\kappa}$ and the bundle $\w^{\kappa}_{\mathrm{AIP}}$ of forms \emph{op.\ cit}.:
	\begin{thm}
		\leavevmode Let $\kappa: \U \to \W$ be a bounded smooth weight (Defn.~\ref{def:smooth weight}; e.g.\ a point or an affinoid open in $\W$). Then there is a natural Hecke-equivariant isomorphism $\omega^{\kappa}\cong \omega^{\kappa}_{\mathrm{AIP}}$.
	\end{thm}
	
	One reason we prefer to work with the anticanonical locus over the canonical one is that it simplifies the proof of the above comparison. A second reason is that it makes it easier to study the boundary: the technical complication for defining elliptic modular forms rather than quaternionic ones is that the cover $ \mathcal{X}^*_{\Gamma(p^\infty)}(\epsilon)_a\to  \mathcal{X}^*_{\Gamma_{\!0}(p^n)}(\epsilon)_a$ is pro-\'etale over $\mathcal{X}_{\Gamma_{\!\scaleto{0}{3pt}}(p^n)}(\epsilon)_a$, but is ramified at the cusps. However, the situation at the cusps is easy to deal with in the anticanonical tower, because here the cusps are totally ramified and give rise to perfectoid versions of Tate curve parameter discs. This allows one to extend the arguments from \cite{CHJ} to the boundary.

	\subsection{Generalisation to the Hilbert case}\label{sec:intro hilbert}
	The main result of the present paper is a generalisation of this approach to the setting of Hilbert modular forms, that is, modular forms for $\GL_2$ over any totally real field $F$ of degree $g$. 
	
	Having treated the elliptic case separately, we will assume $g \geq 2$, and, by the Koecher principle, largely ignore the boundary in this case.
	Whilst conceptually the constructions follow the same lines as in the elliptic case, there are additional subtleties in the Hilbert case that do not arise when the base field is $\QQ$. 
	The most immediate is in the choice of \emph{classical} definition. The Shimura varieties arising from $G^*$ are of PEL (hence Hodge) type. They are fine moduli spaces parametrising Hilbert--Blumenthal abelian varieties (HBAVs), namely abelian varieties equppied with an $\OO_F$-action and a polarisation, plus some fixed tame level structure.
	The Shimura varieties for $G$, in contrast, are only of abelian type, and are only coarse moduli spaces, parametrising instead only equivalence classes of polarisations.
	These distinctions make it technically easier to work with $G^{\ast}$, although ultimately the case in which we are most interested is the arithmetic case of $G$, which has a better theory of Hecke operators.
	
	In both cases these Shimura varieties are called Hilbert modular varieties, for $G^{\ast}$ and $G$ respectively.
	In practice, we will  work with the $\mathfrak{c}$-polarised part of the Hilbert modular variety. We sometimes emphasize this with a subscript $\mathfrak{c}$, but usually drop this from the notation.

	\subsubsection{ Hilbert modular varieties for $G^*$ at infinite level}\label{sec:intro G*}
	
	As in the elliptic case, the key object in the construction of overconvergent forms is an infinite level Hilbert modular variety for $G^*$, which is a $p$-adic analogue of the classical complex Hilbert modular variety. It it is a special case of Scholze's perfectoid Shimura varieties of Hodge type \cite[\S III and IV]{torsion}.
	As we shall recall in \S\ref{sec:hbv at infinite level}, it arises from the tower of  ($\mathfrak{c}$-polarised) Hilbert modular varieties $\XX_{\gothc,\Gamma^*(p^n)}$ as the wild level $\Gamma^*(p^n) \subset G^*(\ZZ_p)$ varies.
	Once again, one can restrict to the anticanonical locus of an $\epsilon$-neighbourhood of the ordinary locus and  obtain a  Hilbert modular variety at infinite level
	\[
	\mathcal{X}_{\gothc,\Gamma^*(p^\infty)}(\epsilon)_a \sim \varprojlim_n \mathcal{X}_{\gothc,\Gamma^*(p^n)}(\epsilon)_a,
	\]
	which is a pro-\'etale $\Gamma_{\!0}^*(p)$-torsor over $\cX_{\gothc,\Gamma_{\!\scaleto{0}{3pt}}^*(p)}(\e)_a$, where $\Gamma_{\!0}^*(p)\subset G^*(\Zp)$ is the subgroup of matrices that are upper-triangular modulo $p$.
	
	We also need a version of the Hodge--Tate period map, as defined in \cite[\S IV]{torsion} and refined in \cite[\S2]{CarScho}. If $C$ is a perfectoid field extension of $L$, then a $(C,C^+)$-point of $	\mathcal{X}_{\gothc,\Gamma^*(p^\infty)}$ corresponds to a HBAV $A$ equipped with a trivialisation $\alpha:\O_p^2\isorightarrow T_pA^\vee$ and extra data, where  $\OO_p := \OO_F\otimes_\ZZ \Zp$. Here we note that the appearance of $A^{\vee}$ differs somewhat from \cite{torsion}, but in the presence of a polarisation $\lambda$, the two are always isomorphic after a choice of $p$-adic generator of $\mathfrak c$.	The reason we wish to parametrise $T_pA^\vee$ rather than $T_pA$ is that together with the Hodge--Tate morphism
	\[\O_p^2\xrightarrow{\alpha}T_pA^\vee \xrightarrow{\HT_A} \omega_{A},\]
	the trivialisation $\alpha$ gives rise to canonical differentials $\alpha(1,0)$ and $\alpha(0,1)$ of $\omega_{A}$ (instead of $\omega_{A^\vee}$). 
	
	This construction can be made more conceptual by way of the Hodge--Tate period morphism
	\[
	\pi_{\HT} :\XX_{\gothc,\Gamma^*(p^\infty)}(\e)_a \longrightarrow \Res_{\OF/\ZZ} \PP^1.
	\] 
	If $F$ splits in $L$ (which we do not assume in the main text), this decomposes canonically into maps $\pi_{\HT}=\prod_{v\in \Sigma} \pi_{\HT,v} :\XX_{\gothc,\Gamma^*(p^\infty)}(\e)_a \to (\PP^1)^{\Sigma}$, where $\Sigma$ is the set of embeddings $v: F \hookrightarrow L$. 
	On $C$-points, a point corresponding to an isomorphism  $\alpha$ is then sent to the point in $(\PP^1(C))^\Sigma$ defined by the Hodge filtration
	\[
	0 \rightarrow \mathrm{Lie} (A^\vee)(1) \longrightarrow T_pA^\vee \otimes_{\ZZ_p} C \xrightarrow{\HT_A} \omega_{A}\rightarrow 0.
	\]
	Crucial here is that $\pi_{\HT}$ allows one to extend this pointwise consideration to the universal situation: If $\omega_{\mathcal A}$ denotes the conormal sheaf to the universal abelian variety $\mathcal A\to\mathcal{X}_{\gothc,\Gamma^*_0(p^n)}(\epsilon)_a$ and $q:\mathcal{X}_{\gothc,\Gamma^*(p^\infty)}(\epsilon)_a \to \mathcal{X}_{\gothc,\Gamma^*_0(p^n)}(\epsilon)_a$ is the forgetful map, then there is a canonical isomorphism
	\[ q^*\omega_A = \pi_{\HT}^{\ast}\Res_{\OF/\ZZ} \O(1). \]
	If $F$ splits in $L$, then $\Res_{\OF/\ZZ}\O(1)$ is identified with the direct sum $\oplus_{\Sigma}\O(1)$ on $(\P^{1})^{\Sigma}$, and using canonical sections of $\O(1)$ near $(0:1)\in \PP^1$, we obtain a canonical section $\mathfrak s$ of $q^*\omega_{\mathcal A}$ which is a geometric incarnation of the section $\alpha(1,0)$ considered above.
	In general we work with a canonical section of $\Res_{\OF/\ZZ}\O(1)$, an instance of Scholze's ``fake Hasse invariants'' from \cite{torsion}.
	
	In the elliptic case, we had a canonical parameter $\mathfrak z$. In the Hilbert case, $\mathfrak{z}$ is now simply the restriction of $\pi_{\HT}$ to a function $\mathfrak z:\mathcal{X}_{\gothc,\Gamma^*(p^\infty)}(\epsilon)_a\to \Res_{\OF/\ZZ}\hat{\G}_a$ where $\hat{\G}_a\subseteq \P^1$ is the closed unit ball around $(0:1)\in \P^1$.
	When $F$ splits in $L$, via the canonical decomposition $\Res_{\OF/\ZZ}\hat{\G}_a=\G_a^{\Sigma}$ this can be interpreted as a collection of functions  $\mathfrak{z} = (\mathfrak{z}_v)_{v\in\Sigma}$ in  $\O^+(\mathcal{X}_{\gothc,\Gamma^*(p^\infty)}(\epsilon)_a)$.
	
	In order to define $p$-adic families of Hilbert modular forms, let $\W^*$ denote the weight space for $G^*$ (cf.  Defn.~\ref{def:Hilbert weight space}) and let $\kappa: \U \to \W^*$ be a bounded smooth weight. In \S\ref{sec:G* forms}, we use the sousperfectoid adic space 
	\begin{align*}
	&\mathcal{X}_{\gothc,\U,\GA^*(p^\infty)}(\epsilon)_a := \mathcal{X}_{\gothc,\Gamma^*(p^\infty)}(\epsilon)_a \times_L \mathcal{U}
	\end{align*}  
	to define the sheaf of $\mathfrak{c}$-polarised geometric Hilbert modular forms of weight $\kappa$ on $\XX_{\gothc,\cU,\Gamma^{\ast}_0(p)}(\epsilon)_a$ as
	\[
	\omega^{\kappa}_{1,\gothc} := \big\{ f \in \roi_{\mathcal{X}_{\mathfrak{c},\U,\Gamma^*(p^\infty)}(\epsilon)_a} \big| \gamma^*f = \kappa^{-1}(c\mathfrak{z}+d)f\quad \forall \gamma = \smallmatrd{a}{b}{c}{d} \in \Gamma_{\!0}^*(p)\big\}.
	\]
	
	Here $\epsilon>0$ is such that for any element of $\g= \smallmatrd{a}{b}{c}{d} \in \Gamma_{\!0}^*(p)$, we can make sense of $\kappa(c\mathfrak{z}+ d)$ as an invertible function on $\mathcal{X}_{\mathfrak{c},\U,\Gamma^*(p^\infty)}(\epsilon)_a$ (Defn.~\ref{d:cocycle-kappa(cz+d)}).	We describe the variation of this in families, and a local version giving an overconvergent automorphic bundle $\omega_{G^*}^\kappa$on $\cX(\e)$. 
	We also have integral versions of these spaces given by simply replacing $\roi_{\mathcal{X}}$ with $\roi_{\cX}^{+}$ in the above definition. In Thm.~\ref{thm:comparison hilbert} we use the canonical sections of $ \omega_{\mathcal A}$ at infinite level to define a comparison isomorphism to the sheaf of Hilbert modular forms defined in \cite{AIP3}.
	
	\subsubsection{Hilbert modular varieties for $G$ at infinite level}
	For arithmetic applications, it is desirable to have a version of this theory for arithmetic Hilbert modular forms, that is for the group $G$. 
	For example, these are objects that arise in modularity of elliptic curves over totally real fields. 
	
	Towards this goal, we pass from $G^{\ast}$ to $G$ and  discuss in \S\ref{sec:G vars} the perfectoid Hilbert modular variety at infinite level for $G$ and the corresponding Hodge--Tate period map. In the case of $L=\C_p$, this is a special case of the construction of Shen of perfectoid Shimura varieties of abelian type~\cite{Xu}. For Hilbert modular varieties, this is easy to extend to general perfectoid fields following the methods of \cite{torsion}, as we shall describe.
	
	In contrast to the definition in \cite{AIP3} by descent from $G^{\ast}$, this allows us to define the sheaf of $p$-adic overconvergent arithmetic Hilbert modular forms without reference to $G^{\ast}$.
	
	While this definition is ultimately quite simple, in order to explain why this is the correct definition, it is important for us to work out the geometric relation between the perfectoid Hilbert modular varieties for $G$ and $G^{\ast}$ rather explicitly, in particular keeping close track of all of the relevant Galois actions. At tame level, this is easy: let $\XX$ and $\XX_G$ be the adic analytifications of the tame $\mathfrak c$-polarised Hilbert modular varieties for $G^{\ast}$ and $G$ respectively. The natural map $\XX\to \XX_G$
	can then be described as the quotient under the action of $\OO_F^{\times,+}$ on the polarisation, which factors through the action of a finite group $\Delta(N)$.
	
	For level at $p$, however, the condition defining level structures for $G^*$ is \emph{not} preserved by the action of $\OO_F^{\times,+}$, so there is no longer a full polarisation action on the spaces $\cX_{\Gamma^*(p^n)}$. We therefore work with an auxiliary `mixed' moduli problem, and work with the spaces $\cX_{\Gamma(p^n)}$ relatively parametrising $G$-level structures over the space $\cX$ arising from $G^*$.  In the limit, combined with the Weil pairing these give rise to a perfectoid space $\cX_{\Gamma(p^\infty)}$ with a canonical $\Zp^\times$-torsor 
	\begin{equation}\label{eqn:intro torsor}
	\cX_{\Gamma^*(p^\infty)} \times {\cO_p^\times} \longrightarrow \cX_{\Gamma(p^\infty)}.
	\end{equation}

	Now, on $\cX_{\Gamma(p^\infty)}$, we \emph{do} get an $\OO_F^{\times,+}$ action on polarisations. At infinite level, this extends to an action of a profinite group $\Delta(p^{\infty}N)$, which makes the morphism $\cX_{\Gamma(p^\infty)}\to \cX_{G,\Gamma(p^\infty)}$ into a pro-\'etale $\Delta(p^\infty N)$-torsor. 
	We thus obtain a morphism of pro-\'etale torsors (cf. \eqref{eq:diag comparing G*,Gamma*(p^infty) to G,Gamma(p^infty) over tame})
	\begin{equation}\label{eq:intro diagram}
	\begin{tikzcd}[column sep = {4.5cm,between origins},{row sep = 0.55cm}]
	\mathcal X_{\Gamma(p^\infty)} \arrow[d,"\GI"'] \arrow[r, "\Delta(p^\infty \nn)"] & \XX_{G,\Gamma(p^\infty)} \arrow[d, "\PGI"]   \\
	\mathcal{X}_{\Gamma_{\!\scaleto{0}{3pt}}(p^n)} \arrow[r,"\Delta(N)"] & \mathcal{X}_{G,\Gamma_{\!\scaleto{0}{3pt}}(p^n)}.
	\end{tikzcd}
	\end{equation}
	The diagonal map is also a pro-\'etale torsor for some group $E(p^n)$, as we shall discuss in \S8.

	\subsubsection{Hilbert modular forms for $G$}
	
	In \S\ref{sec:G forms}, we use the above to define arithmetic Hilbert modular forms. Let $\W$ be the weight space for $G$ and  let $\k : \cU \to \W$ be a bounded smooth weight.
	\begin{definition}\label{def:intro forms}
		The sheaf of \emph{arithmetic Hilbert modular forms of weight $\kappa$} on $\mathcal{X}_{G,\gothc,\mathcal{U}}(\epsilon)$  is
		\[
		\w_G^{\k}=\left\{ f \in \O(\mathcal{X}_{G,\mathfrak{c},\U,\Gamma(p^\infty)}(\epsilon)_a) \middle| \gamma^*f = \kU^{-1}(c\mathfrak z+d)\wU(\det \gamma)f\quad \forall \gamma = \smallmatrd{a}{b}{c}{d} \in \mathrm{P}\Gamma_{\!0}( p)\right\}.
		\]
		
		Its global sections $ M_{\k}^G(\Gamma_{\!0}(p), \epsilon,\gothc)$ form the space of $\mathfrak{c}$-polarised $\epsilon$-overconvergent arithmetic Hilbert modular forms.
		We also have an integral subsheaf $\w_{G}^{\kU,+}$ by instead using the $\O^+$-sheaf. 
	\end{definition}
	Our approach yields various natural alternative definitions; for example, we could instead use the torsor $\mathcal{X}_{\Gamma(p^\infty)}(\epsilon)_a\to \XX_{G,\Gamma_{\!\scaleto{0}{3pt}}(p)}(\epsilon)_a$ to define forms for $G$. In \S\ref{sec:G forms}, we show that these alternatives (see Defn.~\ref{defn: 4 sheaves for the price of 2}) are all equivalent to the one given above.
	
	\subsubsection{Comparison to other definitions}
	Using the canonical section $\mathfrak s$,
	we obtain a comparison isomorphism to the Hilbert modular forms of Andreatta--Iovita--Pilloni in Thms.~\ref{thm:comparison hilbert} and~\ref{thm:comparison-hilbert-for-G}.
	
	\begin{theorem}
		There is a natural Hecke-equivariant isomorphism between $\w_{G^*}^{\kU,+}$ (resp.\ $\omega_{G}^{\k,+}$) and the sheaf of integral Hilbert modular forms for $G^*$ (resp.\ $G$) defined in \cite{AIP3,AIP}.
	\end{theorem} 
	We establish this for $G^*$, and prove that our modular forms for $G$ are the descent of those for $G^*$ under the action of $\Delta(N)$ (see Lem.~\ref{lem:sheaf descent}); from this we we obtain the analogous result for $G$, as the forms of \cite{AIP3} are \emph{defined} via this descent.

	\subsection{Acknowledgements}
	We would like to thank Christian Johansson for answering all our questions and for many helpful conversations. We also thank David Hansen, Vincent Pilloni and Peter Scholze for useful discussions. We would also like to thank the referee for very useful comments.
	
	The work was supported by the following grants: EPSRC Doctoral Prize Fellowship at University College London, [EP/N509577/1] (Birkbeck);  The EPSRC Centre for Doctoral Training in Geometry and Number Theory (The London School of Geometry and Number Theory), University College London, EPSRC [EP/L015234/1] (Heuer); and a Heilbronn Research Fellowship (Williams).
	
	%
	%
	
	\subsection{Notation}
	We fix a rational prime $p$ and a perfectoid field extension $L$ of $\QQ_p^\cyc$. For instance, we could take $L=\QQ_p^\cyc$, or any complete algebraically closed extension of $\Q_p$.
	
\textcolor{black}{	We use adic spaces in the sense of Huber \cite{huber2013etale}, and in particular the notion of  smooth adic spaces from \cite[Definition~1.6.5]{huber2013etale}. By a rigid space over $L$, we mean an adic space of topologically finite type over $\Spa(L,\O_L)$. We use the pro-\'etale site of a smooth adic space in the sense of \cite{Scholzepaidchodge}.}

	The letter $X$ typically refers to modular curves and Hilbert modular varieties. We typically use latin letters $X$ to refer to schemes, gothic letters $\mathfrak{X}$ to refer to formal schemes, and calligraphic letters $\mathcal{X}$ to refer to analytic adic spaces (typically over $L$), like rigid and perfectoid spaces. If $X$ is a modular variety of some tame level, we specify the level at $p$ of a modular variety by a subscript on the appropriate $X$. We will use a superscript $^*$ to denote the minimal compactification $X^{\ast}$.

	Lastly, if $\Gamma$ is any profinite set we also write $\Gamma$ for the associated profinite perfectoid space $\Spa(\mathrm{Maps}_{\cts}(\Gamma,L),\mathrm{Maps}_{\cts}(\Gamma,\O_L))$ when this is clear from the context. If $\Gamma$ is a profinite group, this will be a group object in perfectoid spaces over $L$.
	
	\section{Perfectoid modular curves and the Hodge--Tate period map}\label{sec:HT for MC}
	In this section we review the modular curve at infinite level and the Hodge--Tate period map, and discuss the open subspaces which we are going to use to define $p$-adic elliptic modular forms.

	
	\subsection{Modular curves and their canonical and anticanonical loci}\label{modular curves}
	Let $N$ be an integer coprime to $p$.
	Let $X$ be the modular curve over $L$ of some tame level $\Gamma^p$ at $N$ such that the corresponding moduli problem is representable by a scheme, e.g.\ $\Gamma(N)$ or $\Gamma_1(N)$ for $N\geq 3$.  \textcolor{black}{Similarly, we let $\Gamma_0(p) \subset \GL_2(\ZZ_p)$ denote the usual upper triangular Iwahori subgroup, corresponding to the choice of an order $p$ sub-group-scheme of our elliptic curves.}
	
	We denote by $\XX$ the rigid analytification, considered as an adic space.\footnote{We note that this is the only way in which our notation deviates from that in \cite[\S III]{torsion},  where $\mathcal X$ denotes the good reduction locus.} The space  $\XX$ represents the moduli functor that sends any adic space $S\to\Spa(L,\OO_L)$ to isomorphism classes of pairs $(E,\mu)$ where $E$ is an elliptic curve over $\OO_S(S)$ with $\Gamma^p$-structure $\mu$ \textcolor{black}{(see \cite[Lemma 2.3]{heuer-cusps}).}  
	
	Let $X^{\ast}$ be the compactification of $X$, with analytification $\XX^{\ast}$. The divisor of cusps $\XX^{\ast}\backslash \XX$ becomes a finite set of closed points after adding a primitive $N$-th root to $L$.

	For any $0\leq \epsilon< 1$ with $|p|^{\epsilon}\in |L|$, we denote by $\XX^{\ast}(\epsilon)\subseteq \XX^{\ast}$ the open subspace of the modular curve where $|\widetilde{\Ha}|\geq |p|^{\epsilon}$, where $\widetilde{\Ha}$ is any local lift of the Hasse invariant. For any analytic adic space $\YY\to \XX^{\ast}$, we write $\YY(\epsilon)\subseteq \YY$ for the preimage of $\XX^{\ast}(\epsilon)\subseteq \XX^{\ast}$. We call the elliptic curves parametrised by this open subspace \emph{$\epsilon$-nearly ordinary}.
	
	Let $\mathcal X^{\ast}_{\Gamma(p^\infty)}\sim \varprojlim \mathcal X^{\ast}_{\Gamma(p^n)}$ be the perfectoid modular ``curve" at infinite level as defined in \cite{torsion}. We in particular have the open subspace $\mathcal X^{\ast}_{\Gamma(p^\infty)}(\epsilon)\sim \varprojlim \mathcal X^{\ast}_{\Gamma(p^n)}(\epsilon)$.
	
	Recall that for any $n\in\Z_{\geq 1}$, the modular curve $X_{\Gamma_{\!\scaleto{0}{3pt}}(p^n)}\to X$ of level $\Gamma_{\!0}(p^n)$ relatively represents the choice of a cyclic rank $p^n$ subgroup scheme $D_n\subseteq E[p^n]$. If $0\leq\epsilon<1/(p+1)p^{n-2}$ then by Lubin's theory of the canonical subgroup, any elliptic curve corresponding to a morphism $S\to \XX(\epsilon)$ admits a canonical cyclic subgroup scheme $H_n\subseteq E[p^n]$ of rank $p^n$, which in the case of good reduction reduces to the kernel of the $n$-th iterate of Frobenius on $E$ modulo $p^{1-\epsilon}$. This defines a canonical section $\mathcal X_{\Gamma_{\!\scaleto{0}{3pt}}(p^n)}(\epsilon)\leftarrow \mathcal X(\epsilon)$ which in fact extends to the cusps.
	As a consequence, for $n=1$, the space $\mathcal X^{\ast}_{\Gamma_{\!\scaleto{0}{3pt}}(p)}(\epsilon)$ decomposes into two open and closed components
	\begin{center}
		\begin{tikzcd}[column sep ={between origins,1.2cm}]
		\mathcal X^{\ast}_{\Gamma_{\!\scaleto{0}{3pt}}(p)}(\epsilon)  &=& \mathcal X^{\ast}_{\Gamma_{\!\scaleto{0}{3pt}}(p)}(\epsilon)_c & \dot{\sqcup} & \mathcal X^{\ast}_{\Gamma_{\!\scaleto{0}{3pt}}(p)}(\epsilon)_a,
		\end{tikzcd}
	\end{center}
	the first of which (away from the cusps) parametrises triples $(E,\alpha,H_1)$ with $\alpha$ a tame level and $H_1$ the canonical subgroup, while the second parametrises $(E,\alpha,D_n)$ with $D_n\subseteq E[p]$ a cyclic rank $p$ subgroup such that $D_n\cap H_1=0$. The two components are called the canonical and the anticanonical locus, respectively. At infinite level, these two components, via pullback, give rise to canonical and anticanonical loci of $\mathcal X^{\ast}_{\Gamma(p^\infty)}(\epsilon)$ respectively: 
	\begin{center}
		\begin{tikzcd}[column sep ={between origins,1.2cm},row sep = 0.55cm]
		\mathcal X^{\ast}_{\Gamma(p^{\infty})}(\epsilon) \arrow[d] &=& \mathcal X^{\ast}_{\Gamma(p^{\infty})}(\epsilon)_c \arrow[d] & \dot{\sqcup} & \mathcal X^{\ast}_{\Gamma(p^{\infty})}(\epsilon)_a \arrow[d] \\
		\mathcal X^{\ast}_{\Gamma_{\!\scaleto{0}{3pt}}(p)}(\epsilon) &=& \mathcal X^{\ast}_{\Gamma_{\!\scaleto{0}{3pt}}(p)}(\epsilon)_c & \dot{\sqcup} & \mathcal X^{\ast}_{\Gamma_{\!\scaleto{0}{3pt}}(p)}(\epsilon)_a.
		\end{tikzcd}
	\end{center}
	
	For any perfectoid $(\QQ_p^{\cyc},\ZZ_p^{\cyc})$-algebra $(R,R^{+})$, the $(R,R^{+})$-points of $\mathcal X_{\Gamma(p^\infty)}$ are in functorial 1-1 correspondence with isomorphism classes of triples $(E,\mu,\alpha)$, where $E$ is an elliptic curve over $R$, $\mu$ is a $\Gamma^p$-structure, and $\alpha:\ZZ_p^2\isorightarrow T_pE$ is a trivialisation of the Tate module
	(see \cite[Cor.\ 3.2]{heuer-cusps}).
	We have an action of $\GL_2(\Zp)$ on $\cX_{\Gamma(p^\infty)}$ given by 
	\begin{equation}\label{eqn:action of GL2}
	\gamma \cdot (E,\mu,\alpha) = (E,\mu,\alpha \circ \gamma^\vee), \hspace{12pt} \gamma^\vee = \mathrm{det}(\gamma)\gamma^{-1} = \smallmatrd{d}{-b}{-c}{a} \text{  for   }\gamma = \smallmatrd{a}{b}{c}{d} \in \GL_2(\Zp).
	\end{equation}
	For notational convenience we also fix the following.

	\begin{defn}
		Let $E$ be an ordinary elliptic curve. Then $E$ has canonical cyclic subgroups $H_n$ of rank $p^n$ for all $n$. The \emph{canonical $p$-divisible subgroup of $E$} is $H=(H_n)_{n\in\NN}\subseteq E[p^\infty]$.
	\end{defn}

	
	\subsection{The Hodge--Tate period map around $0\in\PP^1$}

	We recall how the canonical and anticanonical loci behave under the Hodge--Tate period map
	\[
	\pi_{\HT}:\mathcal X^{\ast}_{\Gamma(p^\infty)}\rightarrow \mathbb P^{1}.
	\]
	By \cite[Lem.~III.3.6]{torsion}, the preimage of\/ $\PP^1(\ZZ_p)$ under $\pi_{\HT}$ is given by the closure of the ordinary locus $\mathcal X^{\ast}_{\Gamma(p^\infty)}(0)$. After removing the cusps, this parametrises isomorphism classes of triples $(E,\mu,\alpha)$ as above where $E$ has potentially semistable or good ordinary reduction in every fibre.
	Write $e_1, e_2$ for the standard basis of $\mathbb{Z}_p^2$.  Away from cusps, the preimage $\pi_{\HT}^{-1}(\infty)$ of $\infty = (1:0)$ parametrises triples where moreover $\alpha(e_1)$ generates the canonical $p$-divisible subgroup. 
	
	Instead of using the canonical locus, we shall work with the anticanonical locus $\mathcal X^{\ast}_{\Gamma(p^\infty)}(\epsilon)_a$, which by contrast is sent by $\pi_{\HT}$ to neighbourhoods of points of the form $(b:1)$ for $b\in \Z_p$.
	
	In order to define overconvergent modular forms on $\mathcal X^{\ast}_{\Gamma(p^\infty)}(\epsilon)_a$ of weight $\kappa$ for $\epsilon>0$, we need to make sense of the expression $\kappa(c\pi_{\HT}(x)+d)$ for $x\in \mathcal X^{\ast}_{\Gamma(p^\infty)}(\epsilon)_a$.  To account for the overconvergence, we therefore need to consider open neighboorhods of these points:
	\begin{defn}\label{definition:Zp(w)}
		Let $B_0(\Z_p:1)\subseteq \PP^1(\ZZ_p) \subseteq \PP^1$ be the subspace of points of the form $(a:1)$ for $a\in \Z_p$ considered as a profinite adic space.
		\textcolor{black}{For any $0<r\leq 1$ and any compact open subspace $U\subseteq \Z_p$, let $B_r(U:1)\subseteq \PP^1$ be the subspace defined as the union of all closed balls of radius $r$ around points $(a:1)\in \PP^1(\ZZ_p)$ with $a\in U$}	 
	\end{defn}
	
	\begin{defn}\label{df:parameter-z}		Let $z$ be the parameter on $\PP^1$ at $0$ arising from the canonical isomorphism of schemes $\AA^1 \isorightarrow \PP^1\backslash\{\infty\}$, $z\mapsto (z:1)$ and let $\mathfrak z:=\pi_
		{\HT}^*z$. It is easy to see that $B_r(\ZZ_p:1)\subseteq \AA^1$ is closed both under the additive group structure as well as the multiplicative monoid structure.
	\end{defn}

	\begin{lem}\label{l: action of Gamma on B_1(0)}
		The action of\/ $\Gamma_{\!0}(p)$ fixes $B_r(\ZZ_p:1)\subseteq \PP^1$. In terms of the parameter $z$, we have
		\begin{equation}\smallmatrd{a}{b}{c}{d}\cdot z = \frac{az+b}{cz+d}.
		\end{equation}
	\end{lem}
	\begin{proof}
		Let $\gamma=\smallmatrd{a}{b}{c}{d}\in \Gamma_{\!0}(p)$. Then inside $\PP^1$ we have 
		$\smallmatrd{a}{b}{c}{d} (z:1)=(az+b:cz+d)=(\gamma z:1)$. Moreover, $|cz+d|=1$ on $B_r(\ZZ_p:1)$ since $|z|\leq 1$ (as $r \leq 1$) , $c\in p\ZZ_p$, $d\in \ZZ_p^\times$.
		Consequently, since $(cz+d)^{-1}=d^{-1}\sum_{n\geq 0} (cd^{-1}z)^n\in B_r(\ZZ_p:1)$, the fact that $B_r(\ZZ_p:1)$ is closed under multiplication and addition implies that also $\gamma z = (az+b)(cz+d)^{-1}\in B_r(\ZZ_p:1)$ as desired.
	\end{proof}
	\begin{rmrk}\label{rem:canonical}
		In the definition of modular forms in \cite{CHJ}, the automorphic factor features the term $(b\mathfrak z+d)$, since in their notation -- where $z$ is a parameter for a neighbourhood of $\infty\in\PP^1$ -- the action of $\Gamma_{\!0}(p)$ is given by $z \mapsto (b+dz)/(a+cz)$. In switching from the canonical to the anticanonical locus,  we instead get $(c\mathfrak z+d)$ as in the complex case (see also Rem.~\ref{r:sign-difference-autom-factor-padic-vs-cpx}).
	\end{rmrk}
	
	\
	The following proposition implies that for any weight $\kappa$, there is an $\epsilon >0$ such that for any $x\in \mathcal X^{\ast}_{\Gamma(p^\infty)}(\epsilon)_a$ and any $\smallmatrd{a}{b}{c}{d}\in \Gamma_{\!0}(p)$, the factor of automorphy $\kappa(c\mathfrak z(x)+d)$ converges.
	\begin{prop}\label{proposition: X_w(0) and anticanonical locus}
		
		Let $0\leq r<1$. Then for $0\leq \epsilon\leq r/2$ if $p\geq 5$ or $\epsilon\leq r/3$ if $p=3$, or $\epsilon\leq r/4$ if $p=2$, we have $\pi_{\HT}(\mathcal X^{\ast}_{\Gamma(p^\infty)}(\epsilon)_a)\subseteq B_r(\ZZ_p:1)$.
	\end{prop}
	\begin{proof}
		Away from the cusps, this is a special case of Prop.~\ref{prop: epsilon w for HMFS} below.
		For the cusps, the statement is clear since these are contained in the ordinary locus and are thus sent to $\PP^1(\Q_p)$.
	\end{proof}

	%
	%
	
	\section{Overconvergent elliptic modular forms}
	\label{sec:mod forms via ac locus}

	

	

	In this section we define line bundles of $p$-adic modular forms of weight $\kappa$, where $\kappa$ is a smooth bounded weight. Following \cite{CHJ} with our slightly modified setup, these bundles are defined using the structure sheaf of $\XX^{\ast}_{\Gamma(p^\infty)}(\epsilon)_a$ by taking invariants under a group action with a factor of automorphy to descend to finite level,  mirroring the definition  of complex modular forms.  
	
	We first explain what we mean by a smooth bounded weight:
	\begin{definition}\label{def:smooth weight}
		The \emph{weight space} for $\GL_2$ is $\W:= \Spf(\Z_p[[\Z_p^\times]])^{\ad}_\eta \times_{\Q_p} L$. A \emph{smooth weight} over $L$ is a smooth adic space $\U$ over a perfectoid extension of $L$ together with a map $\U\to \W$. A smooth weight is \emph{bounded} if its image in $\W$  is contained in some affinoid subspace of $\W$.
	\end{definition}

	
	\subsection{Sousperfectoid spaces}\label{sec:sousperfectoid}
	We would like to define sheaves of families of modular forms of weight $\U$ to be functions on $ \XX_{\Gamma(p^\infty)} (\e)_a \times \U$. In order to obtain a sheaf, we need to make sure that the latter fibre product exists as an adic space. For this we use the language of sousperfectoid spaces, which we briefly recall \textcolor{black}{from \cite[Seection 7]{kedhan} and} \cite[Section 6.3]{berkeley17}. Their technical importance stems from Prop.~\ref{prop:sheafy}.
	
	\begin{definition}
		\begin{enumerate}
			\item A complete Tate $\ZZ_p$-algebra $R$ is called \emph{sousperfectoid} if there is a perfectoid $R$-algebra $\widetilde{R}$ such that $R\hookrightarrow \widetilde{R}$ splits in the category of topological $R$-modules.
			\item  A Huber pair $(R,R^{+})$ is called \emph{sousperfectoid} if $R$ is sousperfectoid.
			\item  An adic space is called \emph{sousperfectoid} if it can be covered by affinoid open subspaces of the form $\Spa(R,R^{+})$ where $R$ is sousperfectoid.
		\end{enumerate}
	\end{definition}	
	
	\begin{prop}\cite[Prop.~6.3.4]{berkeley17}\label{prop:sheafy}
		Any sousperfectoid Huber pair $(R,R^{+})$ is stably uniform. In particular, $\Spa(R,R^+)$ is a sheafy adic space.
	\end{prop}

	\begin{cor}\label{cor: fiber prod sous} 
		Let $\cX$ be a perfectoid space over $L$ and let $\cY$ be a  rigid space  smooth over a perfectoid extension of $L'/L$. Then the fibre product $\cX\times_L \cY$ exists as a sousperfectoid adic space.
	\end{cor}
	\begin{proof}
		By \cite[Cor.~1.6.10]{huber2013etale}, the smooth rigid space $\cY$ can be covered by open subspaces which are \'etale over some disc $B=\Spa(L'\langle X_1,\dots,X_n\rangle)$. Since the fibre product of perfectoid spaces is perfectoid,  we may without loss of generality assume that $L=L'$, and that $\cX=\Spa(S,S^{+})$ is affinoid perfectoid. The fibre product $\cX\times_L  B$ then exists and is sousperfectoid because the algebra $S\langle X_1,\dots,X_n\rangle$ is sousperfectoid by \cite[Prop.~6.3.3.(i) and (iii)]{berkeley17}. The fibre product $\cX\times_L \cY=(\cX\times_L B)\times_{ B}\cY$ now exists and is sousperfectoid because algebras \'etale over a sousperfectoid algebra are again sousperfectoid (Prop.~6.3.3.(ii) \emph{op.\ cit}.).
	\end{proof}
	
	\begin{cor}\label{cor: fibre prod sous}
		
		If\/ $\U$ is a smooth adic space over $L$, then $\XX_{\U,\GA(p^\infty)}^*(\e)_a:=\XX_{\Gamma(p^\infty)}^* (\e)_a\times_L \U$ exists as a sousperfectoid adic space. Moreover, if we define $\XX_{\U,\GA(p^n)}^*(\e)_a:=\XX_{\Gamma(p^n)}^* (\e)_a\times_L\U$, then
		\[
		\XX_{\U,\GA(p^\infty)}^*(\e)_a\sim \varprojlim_{n\in\NN} \XX_{\U,\GA(p^n)}^*(\e)_a.
		\]

	\end{cor}
	\begin{proof}
		The first part is immediate from the last corollary. The second part follows from the observation that when $(A_n)_{n\in\NN}$ is a direct system of Tate algebras \textcolor{black}{(by which we mean a Huber pair with a topologically nilpotent unit)}, and $A_\infty$ is a Tate algebra with compatible morphisms $A_n\to A_\infty$ such that $\varinjlim A_n\subseteq A_\infty$ has dense image, and $B$ is a Tate algebra over $A_1$, then $\varinjlim (A_n\hat{\otimes}_{A_1} B)\subseteq A_\infty\hat{\otimes}_{A_1} B$ has dense image by pointwise approximation.
	\end{proof}

	\begin{lem}[{\cite[Thm.~8.2.3]{KedLiu}]}]\label{l:thm-8.2.3-KL-II}
		Let $\cY$ be a seminormal adic space (see \cite[Defn.~3.7.1]{KedLiu}), for example a smooth rigid space. Let $v:\cY_{\proet}\to \cY_{\an}$ be the natural map. Then $v_{\ast}\widehat{\O}^+_{\cY_{\proet}}=\O^+_{\cY}$.
	\end{lem}
	\begin{proof}
		By \cite[Thm.~8.2.3]{KedLiu}, we have $v_{\ast}\widehat{\O}_{\cY_{\proet}}=\O_{\cY_{\an}}$. Using the adjunction morphism of $v$, we thus have inclusions $\O_\cY^+\subseteq v_{\ast}\widehat{\O}^+_{\cY_{\proet}}\subseteq \O_\cY$. On the other hand, for any affinoid $V\subseteq \cY$, we clearly have 
		$v_*\widehat{\O}_{\cY_{\proet}}^+(V)\subseteq v_*\widehat{\O}_{\cY_{\proet}}(V)^{\circ}=\O_{\cY}(V)^{\circ}$. Since $\cY$ is a rigid space, we have $\O_{\cY}(V)^{\circ}=\O^+_{\cY}(V)$, which shows $v_{\ast}\widehat{\O}^+_{\cY_{\proet}}\subseteq \O^+_\cY$.
	\end{proof}

	\begin{lemma}\label{gen lemma for sous}
		Let $\cY$ be an affinoid adic space over $L$ that is either a smooth rigid space or a perfectoid space. Let $\Gamma$ be a profinite group. Let $\cX \in \cY_{\proet}$ be an affinoid perfectoid pro-\'etale $\Gamma$-torsor. Let $\U$ be a smooth adic space over $L$, set $\cX_\U:=\cX \times_L \U$ and $\cY_\U:=\cY \times_L \U$, and denote the induced map by $h: \cX_\U \to \cY_\U$. Then 
		\[
		(h_{*}\OO^+_{\cX_\U})^\Gamma=\OO^+_{\cY_\U}
		\quad \text{and} \quad (h_{*}\OO_{\cX_\U})^\Gamma=\OO_{\cY_\U}.\]
	\end{lemma}
	
	\begin{proof}
		As the statement is local on $\cY_\U$, it suffices to check that for an affinoid open $V\subseteq \cY$ with affinoid perfectoid preimage $W=h^{-1}(V)$ we have $\OO_{\cX_\U}^+(V\times \U)^\Gamma=\OO_{\cY_\U}^+(W\times \U)$.  Since by our assumptions $\cY$ is stably uniform, \cite[Lem.~2.23 (2)]{CHJ} reduces this to checking that $(h_{*}\OO_{\cX}^+)^\Gamma=\OO_{\cY}^+$. 
		To see this, we first treat the case that $\cY$ is a smooth rigid space. Then in the pro-\'etale site $\cY_{\proet}$ in the sense of \cite{Scholzepaidchodge},  we have the structure sheaf \small$\OO^+_{\cY_{\proet}}$\normalsize as well as the completed structure sheaf \small$\wh{\OO}^+_{\cY_{\proet}}$\normalsize. For the affinoid perfectoid space $\cX$, we have \small$\OO^+_\cX(\cX)=\wh\OO^+_{\cY_{\proet}}(\cX)$\normalsize. 
		The Cartesian diagram expressing $\cX \to \cY$ as a pro-\'etale $\Gamma$-torsor then shows that we have
		\[
		\OO^+_\cX(\cX)^\Gamma=\wh{\OO}^+_{\cY_{\proet}}(\cX)^\Gamma=\wh{\OO}^+_{\cY_{\proet}}(\cY).
		\]
		The first part of the Lemma now follows from Lem.~\ref{l:thm-8.2.3-KL-II}. The second follows by inverting $p$.
		
		If $\cY$ is a perfectoid space, the same argument works in the pro-\'etale site $\cY_{\proet}$ of \cite{diamonds}.
	\end{proof}

	\begin{prop}\label{prop: G-invs}
		Let $\U$ be a smooth adic space over a perfectoid field extension $L'$ of $L$. For any $n\geq 1$ denote by $h:\XX_{\U,\GA(p^\infty)}^*(\e)_a \to \XX_{\U,\Gamma_{\!\scaleto{0}{3pt}}(p^n)}^*(\e)_a$ the natural map. Then \[ (h_{\ast}\OO^+_{\XX_{\U,\GA(p^\infty)}^*(\e)_a})^{\Gamma_{\!\scaleto{0}{3pt}}(p^n)}=\OO^+_{\XX_{\U,\Gamma_{\!\scaleto{0}{3pt}}(p^n)}^*(\e)_a}\quad \text{ and } \quad (h_{\ast}\OO_{\XX_{\U,\GA(p^\infty)}^*(\e)_a})^{\Gamma_{\!\scaleto{0}{3pt}}(p^n)}=\OO_{\XX_{\U,\Gamma_{\!\scaleto{0}{3pt}}(p^n)}^*(\e)_a}.\]
	\end{prop}
	For the proof, we explain how to deal with the boundary, which was not treated in \cite{CHJ}.
	\begin{proof}
		After base-change to $L'$, we may without loss of generality assume that $L=L'$.
		
		Over the open subspace away from the cusps, the map $h:\XX_{\U,\GA(p^\infty)}(\e)_a \to \XX_{\U,\Gamma_{\!\scaleto{0}{3pt}}(p^n)}(\e)_a$ is a pro-\'etale $\Gamma_{\!0}(p^n)$-torsor for the action defined in \eqref{eqn:action of GL2}. By Lem.~\ref{gen lemma for sous}, we thus have
		\[ (h_{\ast}\OO^+_{\XX_{\U,\GA(p^\infty)}(\e)_a})^{\Gamma_{\!\scaleto{0}{3pt}}(p^n)}=\OO^+_{\XX_{\U,\Gamma_{\!\scaleto{0}{3pt}}(p^n)}(\e)_a}.\]
		We are left to extend this to the cusps. Let us first look at the case that $\U$ is a single point. For this we can use Tate curve parameter discs as discussed in \cite{heuer-cusps}:
		For any geometric point $c$ in the boundary of $\XX^{\ast}$, there is an integer $d|N$ (depending on the tame level structure and the presence of unit roots in $L$) such that there is an open immersion $D\times \mu_d\hookrightarrow \XX^{\ast}$ where $D\subseteq L\langle q\rangle$ is the open disc defined by $|q|<1$, such that the image of the origin contains $c$. For $\XX^\ast_{\Gamma_{\!0}(p^n)}(\epsilon)_a$ there is then also a Tate curve parameter disc $D\times \mu_d\hookrightarrow\XX_{\Gamma_{\!\scaleto{0}{3pt}}(p^n)}^{\ast}(\epsilon)_a$. The induced map over $\XX^\ast_{\Gamma_{\!0}(p^n)}(\epsilon)_a\to \XX^\ast$ is $D\to D, q\mapsto q^{p^n}$  by Prop.~2.10 \textit{op.\ cit}. Equivalently, we may rewrite this as the open disc $D_n\subseteq \Spa(L\langle q^{1/p^{n}}\rangle)$. By taking tilde-limits, we obtain a perfectoid disc $D_\infty \sim \varprojlim D_n \subseteq \Spa(L\langle q^{1/p^\infty}\rangle).$ By \cite[Thm.~3.8]{heuer-cusps}, there is then a Cartesian diagram
		\begin{center}
			\begin{tikzcd}[row sep = 0.55cm]
			{\Gamma}_0(p^\infty)\times D_\infty\times \mu_d \arrow[d,hook] \arrow[r]& D_n\times \mu_d\arrow[d,hook]\\
			\mathcal X^{\ast}_{\Gamma(p^\infty)}(\epsilon)_a \arrow[r] & \mathcal X^{\ast}_{\Gamma_{\!0}(p^n)}(\epsilon)
			\end{tikzcd}
		\end{center}
		where ${\Gamma}_0(p^\infty)$ is the profinite perfectoid group of upper triangular matrices in $\GL_2(\Z_p)$. By Thm.~3.21 \textit{op.\ cit.}, the $\Gamma_{\!0}(p^n)$-invariance of a function $f$ on ${\Gamma}_0(p^\infty)\times D_\infty\times \mu_d$ now means precisely the following: first, the $\Gamma_{\!0}(p^\infty)$-invariance means that $f$ comes from a function on $D_\infty\times\mu_d$ via pullback along the projection ${\Gamma}_0(p^\infty)\times D_\infty\times\mu_d\to D_\infty\times \mu_d$. It is thus of the form $f\in \OO^+(D_\infty\times \mu_d)= \OO_L[\zeta_d][[q^{1/p^\infty}]]$. 
		Second, the remaining $\Gamma_{\!0}(p^n)/\Gamma_{\!0}(p^\infty)=p^n\Z_p$-action is the one which sends $q^{1/p^m}\to \zeta^h_{p^m}q^{1/p^m}$ for all $m\in\NN$ and $h\in p^n\ZZ_p$. 
		For $f$ to be invariant under this action means that $f\in\OO^+(D_n\times \mu_d)= \OO_L[\zeta_d][[q^{1/p^n}]]$, as desired. 
		
		For general weight $\U$, we may without loss of generality assume that $\U$ is affinoid. The same argument then still works, adding a fibre product with $\U$ everywhere in the above and working with $q$-expansions in $\OO^+(\U)[\zeta_d][[q^{1/p^n}]]$ instead.
	\end{proof}

	\begin{remark}
		\textcolor{black}{   As pointed out to us by the referee, one can also use the rigid analytic Riemann Hebbarkeitssatz for normal rigid spaces due to Bartenwerfer \cite{barten} and L\"utkebohmert \cite[Satz~1.6]{lutke} to extend over the boundary. One advantage of the perspective taken above is that it explicitly describes the $q$-expansions associated to $p$-adic modular forms.}
	\end{remark}

	\subsection{Overconvergent modular forms}
	\begin{definition}
		\textcolor{black}{For any rational $0 < r \leq 1$,} let $B_r(\ZZ_p^\times:1) \subset B_r(\ZZ_p:1) \subset \mathbb{P}^1$ be the union of all balls $B_r(a)$ of radius $r$ around points $(a:1)$ with $a \in \ZZ_p^\times$. For $r=0$, we instead let $ B_0(\ZZ^\times_p:1):={\ZZ}_p^\times$, where we recall that we mean by this the perfectoid space associated to the profinite group ${\ZZ}_p^\times$.
	\end{definition}
	The adic space $B_r(\ZZ_p^\times:1)$ inherits the structure of an adic group from $\mathbb G_m=\PP^1\backslash\{0,\infty\}\subseteq \PP^1$. Note that for $r < 1$, on $\ZZ_p$-points we have $B_r(\ZZ_p^\times:1)(\ZZ_p)=\ZZ_p^\times$ compatible with the group structure. We therefore regard $B_r(\ZZ_p^\times:1)$ as an adic analytic thickening of $\ZZ_p^\times$.
	
	\begin{definition}\label{d:G_a,G_m,and hats}
		\begin{enumerate}
			\item  In the following, we will also denote by $\G_a$ the adic analytification of the corresponding scheme over $L$. Its underlying adic space is $\A^{1,\an}$, and it represents the functor sending an adic space $\cZ$ over $(L,\O_L)$ to $\O(\cZ)$.
			\item We analogously define the adic group $\G_m$ as the adic analytification of the corresponding scheme over $L$. It represents the functor sending an adic space $\cZ$ over $(L,\O_L)$ to $\O(\cZ)^{\times}$.
			\item Denote by $\hat{\G}_a$ the adic generic fibre of the formal completion of the $\mathcal O_L$-scheme $\G_{a,\O_L}$; then $\hat{\G}_a\subseteq \G_a$ is an open subgroup, given by the closed ball of radius $1$ around the origin $0\in \G_a$. It represents the functor sending an adic space $\cZ$ over $(L,\O_L)$ to $\O^+(\cZ)$.
			\item Similarly, let $\hat{\G}_m$ be the adic generic fibre of the formal completion of $\G_{a,\O_L}$; then $\hat{\G}_m\subseteq \G_m$ is an open subgroup, given by the closed ball of radius $1$ around the origin $1\in \G_m$. It represents the functor sending an adic space $\cZ$ over $(L,\O_L)$ to $\O^+(\cZ)^\times$.
		\end{enumerate}	
	\end{definition}

	{\color{black} Any continuous character $\kappa : \Zp^\times \to L^\times$ has} a geometric incarnation as a morphism of adic spaces $\kappa:\ZZ_p^\times\to \hat{\G}_m.$
	Indeed, by the universal property of $\GG_m$, any such morphism corresponds to an element of $\mathcal \OO(\Zpthick)^\times = \Mapc(\ZZ_p^\times,L)^{\times}=\Mapc(\ZZ_p^\times,L^{\times})$. 
	Any such $\kappa$ has an analytic continuation to $B_r(\ZZ_p^\times:1)$ for small enough $r$. In fact, this holds more generally: let $\U$ be any bounded smooth weight (see Defn.~\ref{def:smooth weight}). This corresponds to a morphism
	$\kU: \Zpthick\times \U\to \hat{\G}_m$
	or, equivalently, a continuous morphism $\ZZ_p^{\times}\to \mathcal O_{\U}(\U)^\times$, {\color{black}called the \emph{character of} $\U$}.
	\begin{prop}\label{prop: an cont fam} \label{l: kappa extends to adic group homomorphism on P^1_w(0)}
		If $\k:\U \to \W$ is a bounded smooth weight, then there exists $r_\k$ such that for $r_\k \geq r >0$ there is a unique morphism
		\[\kU^{\an}: B_r(\ZZ_p^\times:1)\times \U\to \hat{\G}_m\]
		such that the restriction of $\kU^{\an}$ to $\ZZ_p^\times \times \U$ via $\ZZ_p^\times\hookrightarrow B_r(\ZZ_p^\times:1)$, $a\mapsto (a:1)$, is equal to $\kU$.
	\end{prop}
	\begin{proof}
		If $\U$ is affinoid this is a special case of \cite[Prop.~8.3]{buzeig}. In general, the assumption that $\U$ is bounded ensures that $\U$ is contained in some affinoid open subspace of $\mathcal W$. For a precise value of $r_\k$ see Prop.~\ref{an cnt hmfs}.
	\end{proof}
	
	\begin{defn}\label{contiuation notation}
		For any smooth bounded weight $\k:\mathcal U \to \W$, let $r_{\scaleto{\k}{4pt}}$ be the supremum of all $r$ such that the proposition holds. Similarly, let $\ed$ be the maximum $\e$ satisfying the conditions of Prop.~\ref{proposition: X_w(0) and anticanonical locus} with respect to $r_{\scaleto{\k}{4pt}}$, then $\pi_{\HT}(\mathcal X^{\ast}_{\Gamma(p^\infty)}(\epsilon)_a)\subseteq B_r(\ZZ_p:1)$.
	\end{defn}

	Recall $\mathfrak z:=\pi_{\HT}^{\ast}z$ is the function on $\mathcal X^{\ast}_{\Gamma(p^\infty)}(\e)_a$ defined by pullback of the function $z$ on $B_r(\ZZ_p:1)$ from Defn.~\ref{df:parameter-z}. 	Since $\pi_{\HT}$ is $\GL_2(\ZZ_p)$-equivariant, Lem.~\ref{l: action of Gamma on B_1(0)} implies that for any $\gamma\in\Gamma_{\!0}(p)$, we have
	\begin{equation}\label{equation: the action on mathfrak z}
	\gamma^{\ast}\mathfrak z = \frac{a\mathfrak z+b}{c\mathfrak z+d}.
	\end{equation}
	Using this and Prop.~\ref{proposition: X_w(0) and anticanonical locus}, we can then make the following definition:
	\begin{defn}
		Let $\k:\U \to \W$ be a bounded smooth weight and $\ed > \epsilon\geq 0$. For any $c\in p\ZZ_p$, $d\in \ZZ_p^{\times}$, we then let $\kU(c\mathfrak z+d)$ be the invertible function on $\XX_{\U,\Gamma(p^\infty)}(\e)_a$ defined by
		\[\kU(c\mathfrak z+d):\XX_{\U,\Gamma(p^\infty)}(\e)_a \xrightarrow{\pi_{\HT}\times \id} B_r(\ZZ_p:1) \times \U\xrightarrow{(z\mapsto cz+d)\times \id} B_r(\ZZ_p^\times:1)\times \U \xrightarrow{\k^{\an}} \hat{\G}_m. \]
		
	\end{defn}
	
	We can now give the definition of sheaves of overconvergent modular forms of weight $\kappa$.
	\begin{defn}\label{Definition:overconvergent modular forms via perfectoid varieties}
		
		For $\k: \U \to \W$ a bounded smooth weight, $n \in \ZZ_{\geq 1} \cup \{\infty\}$ and $0\leq \epsilon\leq \ed$, we define a sheaf ${\omega}^{\kU}_{n}$ on $\mathcal X^{\ast}_{\U,\GA_{\!\scaleto{0}{3pt}}(p^n)}(\e)_a$ by
		\[{\omega}^{\kU\phantom{,+}}_{n}:=\left\{f \in 	q_{\ast}\mathcal O^{\phantom{,+}}_{\mathcal X^{\ast}_{\U,\Gamma(p^\infty)}(\e)_a}\middle| \gamma^{\ast}f = \kU^{-1}(c\mathfrak z+d)f  \text{ for all }\gamma=\smallmatrd{a}{b}{c}{d}\in \Gamma_0(p^n) \right\}, \]
		where we recall that $q:\mathcal X^{\ast}_{\U,\Gamma(p^\infty)}(\e)_a\rightarrow \mathcal X^{\ast}_{\U,\Gamma_{\!\scaleto{0}{3pt}}(p^n)}(\e)_a$ denotes the projection. We also have an integral subsheaf ${\omega}^{\kU,+}_{n}$ on $\mathcal X^{\ast}_{\U,\GA_{\!\scaleto{0}{3pt}}(p^n)}(\e)_a$ defined by using instead the $\O^+$-sheaf:
		\[{\omega}^{\kU,+}_{n}:=\left\{f \in 	q_{\ast}\mathcal O^+_{\mathcal X^{\ast}_{\U,\Gamma(p^\infty)}(\e)_a}\middle| \gamma^{\ast}f = \kU^{-1}(c\mathfrak z+d)f  \text{ for all }\gamma=\smallmatrd{a}{b}{c}{d}\in \Gamma_0(p^n) \right\}. \]
		
		For $n=0$ and $\eU/p>\epsilon\geq 0$, we use the Atkin--Lehner isomorphism $\AL:\mathcal X^{\ast}_{\Gamma_{\!\scaleto{0}{3pt}}(p)}(p\e)_a\isorightarrow\mathcal X^{\ast}(\e)  $ to define {\color{black}the sheaf of overconvergent  $p$-adic modular forms 
		$\omega^{\k}$} to be the sheaf $\AL_{\ast}\omega_{1}^{\k}$ on $\mathcal X_{\U}^{\ast}(\e)$.
		We also have the $\O^+$-submodule of integral $p$-adic modular forms $\omega^{\k,+}$ given by $\AL_{\ast}\omega_{1}^{\k,+}$.
	\end{defn}
	
	We will later see that these are all invertible sheaves.  Note that for $n= \infty$, this defines sheaves of modular forms on $\mathcal X^{\ast}_{\U,\Gamma_{\!\scaleto{0}{3pt}}(p^\infty)}(\e)_a$ which, by analogy to \cite{AIP3}, we call \emph{perfect modular forms}.

	\begin{defn}
		Let $\k: \U \to \W$ be a bounded smooth weight and let $n \in \ZZ_{\geq 0}\cup\{\infty\}$.
		We define the space of \emph{overconvergent modular forms of weight $\U$, wild level $\Gamma_{\!0}(p^n)$, tame level $\Gamma^p$ and radius of overconvergence $0 \leq \e< \ed$} to be the $L$-vector space
		\begin{align*}
		M_{\k}(\Gamma_{\!0}(p^n),\epsilon):=&\hH^0(\XX^*_{\U,\Gamma_{\!\scaleto{0}{3pt}}(p^n)}(\e)_a, \w_n^{\k}  )\\ =&\left\{ f\in \mathcal O(\XX^*_{\U,\Gamma_{\!\scaleto{0}{3pt}}(p^\infty)}(\e)_a)\middle| \gamma^{\ast}f = \kU^{-1}(c \mathfrak z+d)f \text{ for all }\gamma \in \Gamma_{\!0}(p^n) \right\},
		\end{align*}
		and the analogous space of \emph{integral} forms to be
		\[
		M^+_{\k}(\Gamma_{\!0}(p^n),\epsilon):=\hH^0(\XX^*_{\U,\Gamma_{\!\scaleto{0}{3pt}}(p^n)}(\e)_a, \w_n^{\k,+}).
		\]
		Finally, we set $M_{\k}(\epsilon):=	M_{\k}(\Gamma_{\!0}(p),p\epsilon)$, and analogously for the integral subspaces.
		We note that for any $n<\infty$, there is then a natural Atkin--Lehner isomorphism 
		\[
		M^+_{\k}(\Gamma_{\!0}(p^n),\epsilon) \cong M^+_{\k}(p^{-n}\epsilon).
		\] 
	\end{defn}
	One can similarly define cusp forms by working instead with $\w_n^{\k}(-\partial)$ where $\partial$ denotes the boundary divisor in $\XX^*_{\U,\Gamma_{\!\scaleto{0}{3pt}}(p^n)}(\e)_a$. As usual, one now defines an action of Hecke operators $T_\ell$ for $\ell \nmid Np$ and $U_p$ via correspondences. We shall discuss this in detail in the Hilbert case in \S\ref{sec:hecke-operators}.

	\subsection{Comparison to overconvergent modular forms of classical weights}
	Recall that on $\XX^{\ast}$ we have the conormal sheaf $\omega_{\mathcal E}:=\pi_{\ast}\Omega_{\mathcal E|\mathcal X}^{1}$ of the universal semi-abelian scheme $\pi:\mathcal E\to \mathcal X^{\ast}$. For a $p$-level structure $\Gamma_p$ of the form $\Gamma_0(p^n)$ or $\Gamma(p^n)$ for $n\in \ZZ_{\geq 0}\cup\{\infty\}$, {\color{black}we write $\omega_{\Gamma_p}$ for} the pullback of $\omega_{\mathcal E}$ to $\XX^{\ast}_{\Gamma_p}(\epsilon)_a$.
	In this section, we show that for characters $\kappa$ of the form $x\mapsto x^k$, for $k \in \ZZ_{\geq 1}$, the sheaf $\omega_n^{\kappa}$ can be identified with $\omega_{\Gamma_0(p^n)}^{\otimes k}$. This shows that for classical weights, our definition agrees with the usual spaces of overconvergent modular forms, and contains the spaces of classical modular forms.
	The key to the comparison is the isomorphism of line bundles
	\begin{equation}\label{equation: differential bundle is pullback of O(1) via HT period map}
	\pi_{\HT}^{\ast}\mathcal O(1)=\omega_{\Gamma(p^\infty)}
	\end{equation}
	from \cite[ Thm.~III.3.18]{torsion}. We recall that on $(C,C+)$-points, this has the following moduli interpretation:
	The $C$-points of the total space 
	$\mathcal T(1)\rightarrow \PP^1$ of the line bundle $\mathcal O(1)$ parametrise pairs $(L,y)$ of a line $L\subseteq C^2$ together with a point $y\in C^2/L$ on the quotient. 
	Equivalently, this is the data $(\varphi,y)$ of a linear projection $\varphi:C^2\twoheadrightarrow Q$ to a 1-dimensional $C$-vector space $Q$ and a point $y\in Q$. We sometimes just write this as the point $y$ if $Q$ is clear from context. Using this description of $\mathcal O(1)$, one can now illustrate equation~(\ref{equation: differential bundle is pullback of O(1) via HT period map}) as follows.
	
	\begin{lem}\label{l: moduli interpretation of the isomorphism q*w=pi_HT^*O(1)}
		Let $x$ be a $(C,C^+)$-point of\/ $\mathcal X_{\Gamma(p^\infty)}$, corresponding to data $(E,\mu,\alpha:\ZZ_p^2\isorightarrow T_pE)$. Let $\HT_{E}:T_pE\rightarrow \omega_{E}$ be the Hodge--Tate map of $E$. 
		In terms of the total spaces, the isomorphism $\omega_{\Gamma(p^\infty)}=\pi_{\HT}^{\ast}\mathcal O(1)$ is given in the fibre of $x$ by the morphism 
		\begin{alignat*}{2}
		(E,\ \alpha,\ \eta\in \omega_E)\mapsto (\HT:T_pE\otimes_{\ZZ_p} C\rightarrow \omega_E,\ \eta\in \omega_E ).
		\end{alignat*}
	\end{lem}
	
	Following \cite{CHJ}, we now compare our sheaf of modular forms $\omega^\kappa$ for $\kappa=\id$ to the bundle of differentials $\omega_{\mathcal E}$ by using an explicit trivialisation of $\mathcal O(1)$ over $B_r(\Z_p:1)$:

	\begin{definition}\label{l: moduli interpretation of section s on PP^1}
		Let $s$ be the global section $\PP^1\rightarrow \mathcal T(1)$ given by
		\[ (x:y)\mapsto \big(C^2\to C^2\big/\big\langle \smallvector{x}{y}\big\rangle ,\ \ \smallvector{1}{0}+\big\langle \smallvector{x}{y}\big\rangle\big), \]
		where the second component lies in $C^2/\langle \smallvector{x}{y}\rangle.$
		This section is non-vanishing away from the point $(1:0)=\infty\in \PP^1$, and in particular it is invertible over $B_r(\ZZ_p:1)\subseteq \PP^1$.
	\end{definition}	
	
	We need to compute the action of  $\Gamma_{\!0}(p)$ on $s$ over $B_r(\ZZ_p:1)\subseteq \G_a\subseteq \P^1$ in terms of the parameter $(z:1)$. For later reference in the Hilbert case, we record this in diagrammatic fashion:
	\begin{lem}\label{l: action of Gamma_0(p) on the section}
		Let $\gamma=\smallmatrd{a}{b}{c}{d}\in \Gamma_{\!0}(p)$,
		then $\gamma^{\ast}s = (cz+d)\gamma$,  i.e.\ the following diagram commutes:
		\begin{center}
			\begin{tikzcd}[row sep = 0.55cm]
			\hat{\G}_m\times 	\T(1) \arrow[r, "\mathrm m"] & 	\T(1) & 	\T(1) \arrow[l, "\gamma"'] \\
			& 	\hat{\G}_a \arrow[r, "\gamma"] \arrow[lu, "{(cz+d)\times s}"] \arrow[u, "{\gamma^{\ast}s}"'] & \hat{\G}_a \arrow[u, " {s}"'].
			\end{tikzcd}
		\end{center}
	\end{lem}
	\begin{proof}
		We first note that the equivariant action of $\Gamma_{\!0}(p)$ on $\OO(1)$ which is compatible with \eqref{eqn:action of GL2} is by letting $\gamma$ act via $\det(\gamma)^{-1}\gamma$. In particular, $\gamma^{-1}$ acts as $\gamma^{\vee}$.  The natural fibre action of $\gamma$ on $s$ is by $\gamma^* s = \gamma^{-1}\circ s \circ \gamma$, i.e.\ the right square commutes by definition. We therefore have
		\begin{equation}\label{eq:proof-of-gamma*s = (cz+d)gamma}
		\gamma^* s(z)=\gamma^{\vee}\cdot \smallvector{1}{0}=\smallvector{d}{-c}\equiv \smallvector{d}{-c}+c\smallvector{z}{1} =\smallvector{cz+d}{0} = (cz+d)\smallvector{1}{0}\bmod  \langle \smallvector{z}{1}\rangle,
		\end{equation}
		which shows that 
		$\gamma^* s = (cz+d)s$ as desired.
	\end{proof}
	
	\begin{definition}\label{d:section-mathfrak-s-elliptic}
		Let $\mathfrak s:=\pi_{\HT}^{\ast}(s)$ be the pullback of $s$ to a section of $\pi_{\HT}^*\mathcal{O}(1) = q^*\omega$.
	\end{definition}
	Since the isomorphism $\pi_{\HT}^*\mathcal{O}(1)$ is equivariant for the $\Gamma_{\!0}(p)$-action, the action of $\Gamma_{\!0}(p)$ on $\mathfrak s$ is
	\begin{equation}\label{equation: 	gamma* mf s = (c mf z +d) mf s}
	\gamma^* \mathfrak s = (c\mathfrak z+d)\mathfrak s,
	\end{equation}
	where we recall $\mathfrak{z}$ is the pullback of the parameter $z$ to a function on $\XX^{\ast}_{\Gamma(p^\infty)}(\epsilon)_a$. We have the following  consequence of Lem.~\ref{l: moduli interpretation of the isomorphism q*w=pi_HT^*O(1)}.

	\begin{prop}\label{proposition: moduli interpretation of s in terms of HT}
		Let $x$ be a $(C,C^+)$-point of $\mathcal X^{}_{\Gamma(p^\infty)}(\epsilon)_a$ corresponding to a triple $(E,\mu,\alpha:\ZZ_p^2\isorightarrow T_pE)$. Then via the isomorphism $\pi_{\HT}^{\ast}\mathcal O(1)=\omega_{\Gamma(p^\infty)}$, we have
		\[ \mathfrak s(x)=\HT(\alpha(e_1))\in \omega_E.\] 
	\end{prop}
	
	\begin{prop}\label{prop: comparison to classical modular forms}
		Let $\kappa : x\mapsto x^k$. Then there is a natural isomorphism
		$\omega^\kappa\cong \omega^{\otimes k}.$
	\end{prop}
	\begin{proof}
		As $\mathfrak s$ is a non-vanishing section of $q^{\ast}\omega$ over $\XX^{\ast}_{\Gamma(p^\infty)}(\epsilon)_a$, the sections of $q^{\ast}\omega^{\otimes k}$ are all of the form $f\cdot \mathfrak s^{\otimes k}$ for $f\in \OO(\XX^{\ast}_{\Gamma(p^\infty)}(\epsilon)_a)$. Of these sections, the ones coming from sections of $\omega^{\otimes k}$ -- that is, those defined over $\XX^{\ast}_{\Gamma_{\!\scaleto{0}{3pt}}(p)}(\epsilon)_a$ -- are exactly the $\Gamma_{\!0}(p)$-equivariant ones. But
		\[\gamma^{\ast}(f\cdot\mathfrak s^{\otimes k})=\gamma^\ast f \cdot \gamma^\ast\mathfrak s^{\otimes k}\stackrel{\eqref{equation: 	gamma* mf s = (c mf z +d) mf s}}{=\joinrel=}\gamma^{\ast}f \cdot (c\mathfrak z+d)^k\mathfrak s^{\otimes k} \quad\text{ for all }\gamma\in\Gamma_{\!0}(p).\]
		The $\Gamma_{\!0}(p)$-equivariance of $\gamma^{\ast}(f\cdot\mathfrak s^{\otimes k})$ is thus equivalent to $\gamma^\ast f= (c\mathfrak z+d)^{-k}f = \kappa^{-1}(c\mathfrak z+d)f$.
	\end{proof}
	\begin{remark}\label{r:sign-difference-autom-factor-padic-vs-cpx}
		While the analogy to the complex situation is very close, one notable difference is that on the complex upper half plane $\mathcal H$ the canonical differential $\eta_{\mathrm {can}}$ satisfies $\gamma^\ast  {\eta_{\mathrm {can}}} = (cz+d)^{-1} \eta_{\mathrm {can}}$, whereas on $\XX^{\ast}_{\Gamma(p^\infty)}(\epsilon)_a$ one has $\gamma^\ast \mathfrak s = (c\mathfrak z+d)\mathfrak s$. The different signs can be explained as follows. Both constructions of modular forms depend on a canonical trivialisation of the automorphic bundle $\omega$ on the covering space, which is $\mathcal H$ in the complex case and $\XX^{\ast}_{\Gamma(p^\infty)}(\epsilon)_a$ in the $p$-adic case. But there is a sign difference in the canonical trivialisation: consider the universal trivialisation $\alpha:\ZZ^2\isorightarrow  H_1(E,\ZZ)$ on $\mathcal H$ and let $\alpha_i$ denote the image of the standard basis vector $e_i$ of $\ZZ^2$. Then the canonical non-vanishing differential $\eta_{\mathrm {can}}$ is defined to be the unique differential such that $\int_{\alpha_1}\eta_{\mathrm {can}}=1$ under the pairing 
		\[\textstyle\int:H_1(E,\ZZ)\rightarrow \omega^{\vee}_E,\quad \alpha\mapsto \left(w\mapsto \int_{\alpha}w\right).\]
		As a consequence, when we denote by $q^{\ast}\omega$ the pullback of $\omega$ via $\mathcal H\to Y=\SL_2(\ZZ)\backslash\mathcal H$, then using the natural period map $\iota:\mathcal H\hookrightarrow \PP^1$, $\alpha\mapsto \left(\int_{\alpha_2}w_{\mathrm {can}}:\int_{\alpha_1}w_{\mathrm {can}}\right)$
		we have a natural isomorphism
		$q^{\ast}\omega \cong \iota^{\ast}\OO(-1).$
		On the other hand, in the $p$-adic case, the trivialisation is given using the image $\alpha_1$ of $e_1$ under the Hodge--Tate map $\alpha:\ZZ_p^2\to T_pE$. The canonical differential is then the image of $\alpha_1$ under
		$\HT:T_pE=H^{\operatorname{\acute{e}t}}_1(E,\ZZ_p)\rightarrow \omega_{E^{\vee}}=\omega_E$
		and as discussed earlier, the period map $\pi_{\HT}:\XX^{\ast}_{\Gamma(p^\infty)}(\epsilon)_a \to \PP^1$ therefore induces an isomorphism
		$q^{\ast}\omega_E\cong\pi_{\HT}^{\ast}\OO(1).$
		
		In summary, in the complex case one trivialises $\omega_E^{\vee}$ whereas in the $p$-adic case it is $\omega_{E^{\vee}}$, therefore one description of $\omega_E$ is by comparing to $\OO(-1)$ whereas the other uses $\OO(1)$. It is this difference that ultimately leads to the different signs in the definition of modular forms.		
	\end{remark}	
	{\color{black}
	\subsection{Comparison to Katz' convergent modular forms}
    For $\epsilon=0$, it has long been known how to construct sheaves of $p$-adic modular forms, going back to \cite[\S4]{kertz}. We briefly present the construction here in the adic language, and sketch how it compares to our setting.
	   
	   Let us for simplicity assume $\mathcal U=\Spa(L)$; the discussion applies without changes for general $\mathcal U$.
	  Let $\XX^\ast_{\Ig(p^n)}(0)\to \XX^\ast(0)$ be the $n$-th Igusa curve, i.e.\ the $(\Z/p^n\Z)^\times$-torsor parametrising trivialisations $\Z/p^n\Z{\isorightarrow} H_n^\vee$ of the canonical subgroup. Since this has a natural finite \'etale formal model, we can form the inverse limit $\XX^\ast_{\Ig(p^\infty)}(0)\to \XX^\ast(0)$ over $n$ as a sousperfectoid space. This is a pro-\'etale $\Z_p^\times$-torsor known as the Igusa tower, relatively parametrising isomorphisms $\Z_p{\isorightarrow} T_pH^\vee$. As we will see in more detail in the next section, there is a commutative diagram
	    \[\begin{tikzcd}
        \mathcal X^{\ast}_{\Gamma(p^\infty)}(0)_a \arrow[r] \arrow[d,"\mathfrak t"] & \mathcal X^\ast_{\Gamma_0(p)}(0)_a \arrow[d] \\
        \mathcal{X}^\ast_{\Ig(p^\infty)}(0) \arrow[r] & \mathcal X^\ast(0)
        \end{tikzcd}\]
        where $\mathfrak t$ is given by using the canonical isomorphism $H_n^\vee=E[p^n]/H_n$ and sending a trivialisation $\alpha:\Z_p^2{\isorightarrow} T_pE$ to
        \[ 
        \Z_p\xrightarrow{(1,0)}\Z_p^2{\color{red}\isorightarrow} T_pE\xrightarrow{\alpha} T_pH^\vee.
        \]
        
        One now observes that in the case of $\epsilon=0$, the function $c\mathfrak z+d$ is of the form
	    \[c\mathfrak z+d:\XX^{\ast}_{\mathcal U,\Gamma(p^\infty)}(0)\to \Z_p^\times,\]
	    where as usual we consider $\Z_p^\times$ as a profinite adic space.
        One now checks that $\mathfrak t$ is equivariant with respect to the map sending $\smallmatrd{a}{b}{c}{d}\mapsto c\mathfrak z+d$, in the sense that the following diagram commutes:
         \[\begin{tikzcd}
        \Gamma_0(p)\times \mathcal X^{\ast}_{\Gamma(p^\infty)}(0)_a \arrow[r] \arrow[d,"(c\mathfrak z+d)\times \mathfrak t"] & \mathcal X^{\ast}_{\Gamma(p^\infty)}(0)_a \arrow[d,"\mathfrak t"] \\
        \Z_p^\times\times\mathcal{X}^\ast_{\Ig(p^\infty)}(0) \arrow[r] & \mathcal{X}^\ast_{\Ig(p^\infty)}(0).
        \end{tikzcd}\]
        It follows formally (e.g.\ \cite[Lemma~2.8.4]{heuer-thesis}) that $\omega^{\kappa}$ is the pullback of the pro-\'etale line bundle on $\mathcal X^\ast(0)$ associated to the cocycle
        \[ \kappa:\Z_p^\times\to \O(\mathcal U)^\times.\]
    
    Returning to general $\epsilon\geq 0$ and $\mathcal U$,
	we now use this to give a new proof that $\omega^{\kappa}$ is analytic:
\begin{proposition}\label{p:w-is-analytic}
        For any $0 \leq \e< \ed$, the sheaf $\omega^{\kappa}$ is an analytic line bundle on $\XX^{\ast}_{\mathcal U}(\epsilon)$.
	\end{proposition}
	\begin{proof}
	    Away from the boundary, it is clear that $\omega^{\kappa}$ is a pro-\'etale line bundle, i.e.\ an invertible module over the completed structure sheaf of $\XX_{\mathcal U}(\epsilon)_{\proet}$, as it is defined via a descent datum for the pro-\'etale torsor $\mathcal X_{\Gamma_0(p^\infty)}(\epsilon)_a\to \mathcal X_{\Gamma_0(p)}(\epsilon)_a$. The crucial point is now that by
	    \cite[Cor.~3.5]{heuer-v_lb_rigid}, such a pro-\'etale line bundle is already  an analytic line bundle if it is analytic on any Zariski-dense open subspace of $\XX_{\mathcal U}(\epsilon)$. We can thus reduce to proving the statement over the ordinary locus, including the boundary; the analyticity will then automatically overconverge.

        We now use that the Igusa tower admits a formal model which is still a pro-\'etale $\Z_p^\times$-torsor. It follows from \cite[Prop.~3.8]{heuer-v_lb_rigid} that $\omega^\kappa$ on $\mathcal X^{\ast}(0)$ is locally trivial in the analytic topology.
	\end{proof}
	\begin{remark}
	    Alternatively, one could use the analyticity criterion \cite[Cor.~3.6]{heuer-v_lb_rigid}, which says that a pro-\'etale line bundle on $\mathcal X^\ast(\epsilon)\times \mathcal U$ is analytic if it is analytic in each fibre of a Zariski-dense subset of points in each factor. It is clear that $\omega^{\kappa}$ becomes trivial over the fibre of any $x\in \mathcal X^\ast(\epsilon)(C)$ because the torsor $\mathcal X^\ast_{\Gamma(p^\infty)}(\epsilon)\to \mathcal X^{\ast}(\epsilon)$ becomes split over $x$. On the other hand, Prop.~\ref{prop: comparison to classical modular forms} says that $\omega^\kappa$ is analytic over the Zariski-dense set of classical points of $\mathcal W$.
	\end{remark}
	}

	
	\section{Comparison with Andreatta--Iovita--Pilloni's modular forms}\label{sec:AIP elliptic}
	In this section, we prove that the sheaves of $p$-adic modular forms defined above are canonically isomorphic to those defined  \textcolor{black}{by Pilloni in \cite{pilloni2013formes}} and Andreatta--Iovita--Pilloni in \cite{AIP2}. In order to distinguish {\color{black}their construction from ours}, we shall denote the latter sheaf by $\omega_{\AIP}^{\kappa}$. This extends the comparison for classical weights proved in Prop.~\ref{prop: comparison to classical modular forms} above.
	We first briefly summarise the construction.
	\subsection{The Pilloni-torsor}
	The construction of $\omega^{\k}_{\mathrm{AIP}}$ relies on the Pilloni-torsor $\mathcal F_m(\epsilon)$. In this subsection, we will recall its definition, and show how the section $\mathfrak s$ from Defn~\ref{d:section-mathfrak-s-elliptic} induces a natural map $\XX^{\ast}_{\Gamma(p^\infty)}(\epsilon)_a\to \mathcal F_m(\epsilon)$ into the Pilloni-torsor, allowing a direct comparison of the modular forms in \cite{AIP2} with those defined in Defn.~\ref{Definition:overconvergent modular forms via perfectoid varieties} above. {\color{black}This is very similar to  \cite[\S 2.7]{CHJ}, which also relies on the section $\mathfrak s$, but the present setup makes the comparison slightly easier:  it allows for a global comparison map from the perfectoid modular curve to the Pilloni torsor, which makes it possible to avoid auxiliary choices.}
	
	We shall discuss the Pilloni-torsor in an analytic setting, rather than dealing with normal formal schemes like in \cite{AIP2}. 
	This is mainly to avoid discussing normalisations in our non-Noetherian setting over $\O_L$ --  while this is still possible, it would require more work.
	
	\begin{definition}\label{defn: Ig tower}
		For any $m\in \Z_{\geq 1}$, let $0\leq \epsilon\leq \epsilon^{\mathrm{can}}_{m}:=1/p^{m+1}$. Like in \cite[Definition III.2.12]{torsion}, one can define a canonical formal model $\mathfrak X^{\ast}(\epsilon)$ of $\mathcal X^{\ast}(\epsilon)$ with a semi-abelian formal scheme $\mathfrak E\to \mathfrak X^{\ast}(\epsilon)$.
		By \cite[Cor.~A.2]{AIP2}, this admits a canonical subgroup $\mathfrak H_m\subseteq  \mathfrak E$ of rank $p^m$ characterised by the property that its reduction mod $p^{1-\epsilon}$ equals $\ker F^m$ where $F$ is the relative Frobenius. Let $H_m\subseteq \mathcal E\to  \mathcal X^{\ast}(\epsilon)$ be the adic generic fibre, this is the universal canonical subgroup of the semi-abelian adic space over $\mathcal X^{\ast}(\epsilon)$  (cf the discussion in \S\ref{modular curves}).
		
		The Igusa curve $\XX^{\ast}_{\mathrm{Ig}(p^m)}(\epsilon)\to \mathcal X^{\ast}(\epsilon)$ is now the finite \'etale $(\Z/p^m\Z)^{\times}$-torsor which relatively represents isomorphisms of group schemes $\Z/p^m\Z\to H_m^{\vee}$.
	\end{definition}
	Consider the pullback  $\omega_{\mathrm{Ig}(p^m)}$  of the conormal sheaf $\omega$ to $\XX^{\ast}_{\mathrm{Ig}(p^m)}(\epsilon)$, and denote the total space of  $\omega_{\mathrm{Ig}(p^m)}$ by $\mathcal T_m(\epsilon)\rightarrow \XX^{\ast}_{\mathrm{Ig}(p^m)}(\epsilon)$.
	Following \cite{pilloni2013formes}, the Pilloni-torsor is now a certain open subspace $\mathcal{F}_{m}(\epsilon)\subseteq \mathcal T_m(\epsilon)$. We recall this in an analytic setting.
	\begin{definition}
		For any formal scheme $\mathfrak Y$, the generic fibre $\mathcal Y=\mathfrak Y_{\eta}^{\ad}$ comes equipped with a  morphism of locally ringed spaces $s:(\mathcal Y,\O_{\mathcal Y}^+)\to \mathfrak Y$. For any coherent sheaf $\mathfrak G$  on $\mathfrak Y$ with generic fibre $\mathcal G$, this gives rise to an integral $\O_{\mathcal Y}^+$-module $\mathcal G^+:=s^{\ast}\mathfrak G$ on $\mathcal Y$.
	\end{definition}
	Applying this to the conormal sheaf of $\mathfrak E\to \mathfrak X^{\ast}(\epsilon)$  and pulling back to $\XX^{\ast}_{\Ig(p^m)}(\epsilon)$, we see that the sheaf $\omega_{\mathrm{Ig}(p^m)}$ carries a natural integral structure $\omega_{\mathrm{Ig}(p^m)}^+\subseteq \omega_{\mathrm{Ig}(p^m)}$. The same construction applied to the conormal sheaf of the canonical subgroup $\mathfrak H_m\subseteq \mathfrak E\to \mathfrak X^{\ast}(\epsilon)$ gives a morphism of $\O_{\XX^{\ast}_{\Ig(p^m)}(\epsilon)}^+$-modules $\pi:\omega_{\mathrm{Ig}(p^m)}^+\to \omega_{H_m}^+$. Let $\mathrm{Hdg}$ be the Hodge ideal on $\mathcal X_{\Ig(p^m)}^{\ast}(\epsilon)$ \cite[\S3.1]{AIP2}.
	\begin{lem}\label{l:ses-of-integ-conormal-sheaves}
		The following sequence of $\O_{\mathcal X_{\Ig(p^m)}^{\ast}(\epsilon)}^+$-modules is exact:
		\[ 0\to p^m\mathrm{Hdg}^{-\frac{p^m-1}{p-1}}\omega_{\mathrm{Ig}(p^m)}^+\lra \omega_{\mathrm{Ig}(p^m)}^+\xrightarrow{\ \ \pi\ \ } \omega_{H_m}^+\to 0.\]
	\end{lem}
	\begin{proof}
		We have an exact sequence $0 \to \omega_{\mathfrak E[p^m]/\mathfrak H_m} \to \omega_{\mathfrak E[p^m]} \to \omega_{\mathfrak H_m} \to 0$ over $\mathfrak X^{\ast}(\epsilon)$. The middle term is $\omega_{\mathfrak E}/p^m$, whilst the first term has annihilator $\mathrm{Hdg}^{(p^m-1)/(p-1)}$ by \cite[Cor. A.4.2]{AIP2}. We thus have an isomorphism $\omega_{\mathfrak E}/p^m\mathrm{Hdg}^{-(p^m-1)/(p-1)} \cong \omega_{\mathfrak H_m}$. Pulling this back under the morphism of ringed spaces
		$(\mathcal X^{\ast}_{\Ig(p^m)}(\epsilon),\O_{\mathcal X^{\ast}_{\Ig(p^m)}(\epsilon)}^+)\to (\mathcal X^{\ast}(\epsilon),\O_{\mathcal X^{\ast}(\epsilon)}^+)\to \mathfrak X^{\ast}(\epsilon)$ gives the result.
	\end{proof}
	When we now regard $H_m^\vee$ as a sheaf of sections over $\mathcal X_{\Ig(p^m)}^{\ast}(\epsilon)$, we have a morphism of sheaves
	\begin{equation}\label{definition:Hodge--Tate morphism in definition of Pilloni-torsor}
	\begin{tikzcd}
	\psi:\ZZ/p^m\ZZ \arrow[r, "\alpha"] & H_m^\vee  \arrow[r,"\HT"] & \omega^+_{H_m}
	\end{tikzcd}
	\end{equation}
	that defines a canonical section $\psi(1)\in \omega^+_{H_m}$.
	\begin{definition}
		Let $0 \leq \epsilon \leq \epsilon^{\mathrm{can}}_{m}$. The \emph{Pilloni-torsor} $\mathfrak{F}_m(\epsilon)$ is the $\O_{\mathcal X^{\ast}_{\Ig(p^m)}(\epsilon)}^{+}$-module defined by
		\[\mathfrak{F}_m(\epsilon):=\{r\in \omega_{\mathrm{Ig}(p^m)}^+\mid\pi(r)=\psi(1)\}. \]
		Let $\mathcal{F}_m(\e)\hookrightarrow \mathcal T_m(\e) \to\mathcal X^{\ast}_{\Ig(p^m)}(\epsilon)$ be its total space.
		By Lem.~\ref{l:ses-of-integ-conormal-sheaves}, this is an analytic torsor under the group $(1+p^m\Hdg^{-(p^m-1)/(p-1)}\hat{\G}_a)\to \mathcal X^{\ast}_{\Ig(p^m)}(\epsilon)$ and an \'etale torsor over $\XX^{\ast}(\e)$ for the group $B_m:=\Z_p^\times(1+p^m\Hdg^{-(p^m-1)/(p-1)}\hat{\G}_a)$ when combined with $\XX_{\Ig(p^n)}^{\ast}(\e)\to \XX^{\ast}(\e)$.
	\end{definition}
	\begin{defn}
		\begin{enumerate}
			\item For the universal weight $\kappa^{\mathrm{un}}:\W\to \W$, let $T$ be the weight space parameter given by the function $\kappa^{\mathrm{un}}(q)-1$ on $\W$ for a fixed topological generator $q\in\Z_p^\times$. 
			\item For any $k\in \Z_{\geq 1}$, we define the annulus $\mathcal W_{k}:= \mathcal W(|T|^{p^{k}}\leq |p|\leq |T|^{p^{k-1}})$. For $k=0$, we simply take the disc $\W_0:=\mathcal W(|T|\leq |p|)$.
			Then $ \mathcal W=\cup_{k\in \Z_{\geq 0}}\mathcal W_k$.
			\item Let now $\kappa:\mathcal U\to \mathcal W$ be any bounded smooth weight, which we may regard as a function $\kappa:\Z_p^\times\times \mathcal U\to \hat{\G}_m$. We let $|T_{\kappa}|:=\sup_{(x,u)\in\Z_p^\times\times \U} \lvert\kappa(x,u)-1\rvert$ and $|\delta_{\kappa}|:=\max(|p|,|T_{\kappa}|)$. Let $r:=3$ if $p>2$ and $r:=5$ if $p=2$. Then we let $0<\epsilon_{\kappa}$ be implicitly defined by $|p|^{\epsilon_{\kappa}}=|\delta_{\kappa}|^{1/p^{r+1}}$. Finally, for any $k\in \Z_{\geq 0}$, we let $\mathcal U_k:=\kappa^{-1}(\mathcal W_k)$; then $\mathcal U=\cup_{k\in \Z_{\geq 0}}\mathcal U_k$.
		\end{enumerate}
	\end{defn}
	
	\begin{rmrk}
		Once checks easily that $\e_{\k} \leq \ed$. Moreover, we note that by definition, $\ed$ is such that we can define our sheaf of modular forms on $\XX(\e)$ for $0 \leq \e <\ed$, while for $0 \leq \e < \e_\k$ this sheaf will be invertible (although we do not claim this is the optimal bound).
	\end{rmrk}

	The sheaf $\omega_{\mathrm{AIP}}^{\k}$ is then defined in \cite{pilloni2013formes}  and \cite{AIP2} as follows. 
	\begin{definition}Let $\k:\U\to \W$ be a bounded smooth weight. In order to define $\omega_{\AIP}^{\kappa}$, we need to split this up into opens $\U_k$ as in Defn.~\ref{contiuation notation}: Let $0 \leq \e < \e_\kappa$. Fix now any $k \in \ZZ_{\geq 0}$ and set $m=r+k-1$ (this implies $\epsilon_{\kappa}\leq \epsilon_m^{\mathrm{can}}$ and $\e_{\k} \leq \ed$).
		\begin{enumerate}
			\item	Using the projection $\mathrm{pr}:\mathcal F_m(\epsilon)\times \U\to \mathcal X_{\U}^{\ast}(\epsilon)$, we define
			\[\omega^{\k}_{\mathrm{AIP}}|_{\U_k} (V) := \{f\in \mathrm{pr}_\ast\mathcal O_{\mathcal F_m(\epsilon)\times \U}(V)|\gamma^\ast f=\kU^{-1}(\gamma)f \text{ for all }\gamma\in \Z_p^\times(1+p^m\hat{\G}_a)\times \U \}.\] Prop.~\ref{prop: glue AIP sheaf} below shows that these sheaves for different $k\in \Z_{\geq 1}$ can then be glued over $\U$ to get a sheaf $\omega^{\k}_{\mathrm{AIP}}$. We similarly define $\omega^{\k,+}_{\mathrm{AIP}}$ by using $\mathcal O^+_{\mathcal F_m(\epsilon)\times \U}$ instead.
			\item For any $n\in\Z_{\geq 1}$ we set $\omega^{\k,+}_{\mathrm{AIP},n}:=\AL^{n*}\omega^{\k,+}_{\mathrm{AIP}}$ where $\AL^n:\XX^{\ast}_{\Gamma_{\!\scaleto{0}{3pt}}(p^n)}(p^n\epsilon)_a\isorightarrow \XX^{\ast}(\e)$ is the Atkin--Lehner isomorphism (corresponding to the matrix $\smallmatrd{p^n}{0}{0}{1}$). Let $i:\XX^{\ast}_{\Gamma_{\!\scaleto{0}{3pt}}(p^n)}(\epsilon)_a\to \XX^{\ast}_{\Gamma_{\!\scaleto{0}{3pt}}(p^n)}(p^n\epsilon)_a$ be the restriction map. Then by \cite[Th\'eor\`eme~6.2.4]{AIP2}, there is a canonical isomorphism
			\[
			i^{\ast}\omega^{\k,+}_{\mathrm{AIP},n} = q_n^{\ast}\omega^{\k,+}_{\mathrm{AIP}},
			\]
			where $q_n:\XX^{\ast}_{\Gamma_{\!\scaleto{0}{3pt}}(p^n)}(\epsilon)_a\to \XX^{\ast}(\epsilon)$ is the forgetful map. For $n=0$ we let $\omega^{\k,+}_{\mathrm{AIP},0}:=\omega^{\k,+}_{\mathrm{AIP}}$.
			\item We therefore set $\omega^{\k,+}_{\mathrm{AIP},\infty}:=q^{\ast}\omega^{\k,+}_{\mathrm{AIP}}$ where $q:\XX^{\ast}_{\Gamma_{\!\scaleto{0}{3pt}}(p^\infty)}(\epsilon)_a\to \XX^{\ast}(\epsilon)$ is the forgetful map.
		\end{enumerate}
	\end{definition}
	\subsection{The comparison morphism}\label{section: ell comparison}
	The following is the main result of this section.
	\begin{thm}\label{thm:comparison 1}
		\leavevmode
		Let $\k:\U \to \W$ be a bounded smooth weight, let $0\leq \epsilon\leq \epsilon_{\kappa}$, and let $n\in \Z_{\geq 0}\cup\{\infty\}$. Then there is a natural  isomorphism of $\O^+$-modules on $\mathcal X^{\ast}_{\U,\Gamma_{\!\scaleto{0}{3pt}}(p^n)}(\epsilon)_a$
		\[\omega_n^{\k,+} \isorightarrow \omega^{\k,+}_{\mathrm{AIP},n}.\]
		In particular, the $\omega_n^{\k,+}$ are invertible $\O^+$-modules. By inverting $p$, we also obtain isomorphisms of invertible $\O$-modules $\omega_n^{\k} \isorightarrow \omega^{\k}_{\mathrm{AIP},n}$ on $\mathcal X^{\ast}_{\U,\Gamma_{\!\scaleto{0}{3pt}}(p^n)}(\epsilon)_a$. Moreover, this induces a Hecke equivariant isomorphism between the respective spaces of modular forms.
	\end{thm}
	
	For $n=0$, this in particular gives a canonical isomorphism $\omega^{\k,+} \cong \omega^{\k,+}_{\mathrm{AIP}}$ on $\XX_{\U}^{\ast}(\e)$. 
	
	The proof of Thm.~\ref{thm:comparison 1} now relies on the following lemma (cf \cite[p 31]{AIP2}):
	\begin{lem}\label{l: HT from Gamma to F_n}
		Let $\epsilon\leq \epsilon_m^{\mathrm{can}}$. Then $\mathfrak s$ induces a morphism over $\XX^{\ast}(\epsilon)$
		\[\mathfrak s:\XX^{\ast}_{\Gamma(p^\infty)}(\epsilon)_{a} \to\mathcal F_m(\e).\]
		Let $\tilde{\mathfrak s}:= \mathfrak s \circ u_n$. Then for any $\gamma = \smallmatrd{a}{b}{c}{d}\in \Gamma_0(p^n)$, we have $\gamma^{\ast}\tilde{\mathfrak s}=(c\mathfrak z+d)\tilde{\mathfrak s}$, that is the diagram
		\begin{equation}\label{dg:comparison-mod-forms-AIP-forms}
		\begin{tikzcd}[row sep = 0.55cm]
		\mathcal X^{\ast}_{\Gamma(p^\infty)}(p^n\epsilon)_a \arrow[r, "\gamma"] \arrow[d, "(c\mathfrak z+d)\times\tilde{\mathfrak s}"'] & \mathcal X^{\ast}_{\Gamma(p^\infty)}(p^n\epsilon)_a \arrow[d, "\tilde{\mathfrak s}"'] \arrow[r] & \mathcal X^{\ast}_{\Gamma_{\!\scaleto{0}{3pt}}(p^n)}(p^n\epsilon)_a \arrow[d,"\AL^n"] \\
		B_m\times \mathcal F_m(\epsilon) \arrow[r, "\mathrm m"] & \mathcal F_m(\epsilon) \arrow[r] & \mathcal X^{\ast}(\epsilon),
		\end{tikzcd}
		\end{equation}
		commutes, where $\mathrm{m}$ denotes the respective action maps.
	\end{lem}
	\begin{rmrk}
		In fact, the map $\mathfrak s$ factors through the forgetful map $\mathcal X^{\ast}_{\Gamma(p^\infty)}(\epsilon)_{a}\to \mathcal X^{\ast}_{\Gamma_1(p^\infty)}(\epsilon)_a$.
	\end{rmrk}
	\begin{proof}
		Lem.~\ref{l:s-from-X-to-P} and Lem.~\ref{l:comparison-mod-forms-AIP-forms-Hilbert} below will prove a more general version of the first part away from the cusps. The morphism extends since the relative moduli interpretation of $\mathcal F_m(\e)\to \mathcal X_{\Ig(p^m)}(\epsilon)\to \XX(\epsilon)$ also holds over the cusps.
		
		For the second part, recall that $\mathfrak s = \pi_{\HT}^{\ast}s$. By $\GL_2(\Q_p)$-equivariance of $\pi_{\HT}$, we therefore have $\tilde{\mathfrak s} = \pi_{\HT}^{\ast}\tilde{s}$ for the section $\tilde{s}:=u_n^{\ast}s$ on $\hat{\G}_a\subseteq \P^1$. It therefore suffices to prove that $\gamma^{\ast}\tilde{s} = (cz+d)\tilde{s}$.
		For this we first note that the action of $u_n$ sends $(z:1)$ to $(p^nz:1)$, and therefore 
		\begin{equation}\label{eq:proof-of-gamma*tilde s = (cz+d)tilde s}
		u_n^{\ast}z = p^nz
		\end{equation}
		In particular, $u_n$ sends $\hat{\G}_a\subseteq \P^1$ onto $p^n\hat{\G}_a\subseteq \P^1$.
		Note that this is preserved by the action of $\Gamma^0(p^n):=\{\smallmatrd{\ast}{b}{\ast}{\ast}\subseteq \GL_2(\Z_p)| b\in p^n\Z_p\}$.  By the same computation as in \eqref{eq:proof-of-gamma*s = (cz+d)gamma}, we see that on $p^n\hat{\G}_a$, for any $\gamma'=\smallmatrd{a'}{b'}{c'}{d'}\in \Gamma^0(p^n)$, we have $\gamma'^{\ast}s=(c'z+d')s$. For $\gamma':=u_n \gamma u_n^{-1}$, this shows
		\[\gamma^{\ast}\tilde{ s}=\gamma^{\ast}u_n^{\ast} s=u_n^{\ast}(u_n \gamma u_n^{-1})^{\ast} s = u_n^{\ast} \smallmatrd{a}{p^nb}{p^{-n}c}{d}^{\ast} s = u_n^{\ast}((p^{-n}c z+d) s) \stackrel{\eqref{eq:proof-of-gamma*tilde s = (cz+d)tilde s}}{=} (c z+d)\tilde{ s}.\qedhere\]
	\end{proof}
	
	\begin{proof}[Proof of Thm.~\ref{thm:comparison 1}]
		It suffices to prove this locally on $\W$, so we may assume that $\k$ has image in $\W_k$ for some $k\in \Z_{\geq 0}$. Set $m=r+k-1$. We start with the case of $n\in \Z_{\geq 1}$.
		
		Let $f$ be a section of $\omega_{\mathrm{AIP}}^{\k,+}$. For simplicity of notation, let us assume that $f$ is a global section, even though the proof works for any section. We may then regard $f$ as a map $\mathcal F_{m}(\e)\times \U \to  \hat{\mathbb G}_a$. To see that $\tilde{\mathfrak s}^{\ast}f=f\circ \tilde{\mathfrak s}$ is a section of $\omega_n^{\k,+}$, we use that for any $\gamma \in \Gamma_{\!0}(p^n)$, the diagram
		\begin{center}
			\begin{tikzcd}[column sep = 1.8cm,row sep = 0.55cm] \mathcal X^{\ast}_{\Gamma(p^\infty)}(p^n\epsilon)_a\times \U \arrow[r,   "\gamma"] \arrow[d,   "(c\mathfrak z+d)\times\tilde{\mathfrak s}\times \id"'] & \mathcal X^{\ast}_{\Gamma(p^\infty)}(p^n\epsilon)_a\times \U  \arrow[d,   "\tilde{\mathfrak s}\times \mathrm {id}"']\\
			\mathbb{Z}^\times_p(1+p^m\hat{\G}_a)\times \mathcal F_{m}(\e)\times \U \arrow[r,   "\mathrm m\times \id"] \arrow[d,   "\kU^{-1}\times f"'] & \mathcal F_{m}(\e)\times \U \arrow[d,   "f"']\\
			\hat \G_m\times \hat \G_a \arrow[r,   "\mathrm m"] & \hat{\G}_a
			\end{tikzcd}
		\end{center}
		commutes, where $\mathrm m$ denotes the multiplication map.
		Here the top square commutes by Lem.~\ref{l: HT from Gamma to F_n} and the bottom square is commutative by definition of $\omega^{\k,+}_{\mathrm{AIP}}$. The outer square now shows that
		\[\mathfrak \gamma^{\ast}(\tilde{\mathfrak s}^{\ast}f)=\kU^{-1}(c\mathfrak z+d)\tilde{\mathfrak s}^{\ast}f \]
		as desired. This gives a natural morphism of $ \mathcal O^+_{\XX_{\U}^{\ast}(\epsilon)}$-modules
		$\omega_{\mathrm{AIP}}^{\k,+}\to \AL^n_{\ast}\omega_n^{\k,+}$ which in turn induces  $\tilde{\mathfrak s}^{\ast}:\omega_{\mathrm{AIP},n}^{\k,+}=\AL^{n\ast}\omega_{\mathrm{AIP}}^{\k,+}\to \omega^{\k}_n$.
		
		To see this is an isomorphism, recall that $\omega_{\mathrm{AIP}}^{\k,+}$ is invertible, so locally on some open $U\subseteq \XX_\U(\epsilon)$ we can find $f$ that is invertible as an element of $\O^+_{\mathcal F_m(\epsilon)\times \U}(U)$. Thus $\omega_{\mathrm{AIP}|\U}^{\k,+}=f\O^+_{U}$.  
		Let $V:=\AL^{-n}(U)$; then $\tilde{\mathfrak s}^{\ast}f\in q_{\ast} \O^{+}_{\XX^{\ast}_{\U,\Gamma(p^\infty)}(p^n\epsilon)_a}(V)$ is invertible. 
		It now follows from Prop.~\ref{prop: G-invs} that we have $\omega^{\kappa,+}|_V=\tilde{\mathfrak s}^{\ast}f \O^+_{V}$. 
		Thus $\tilde{\mathfrak s}^{\ast}$ is an isomorphism locally over $U$.
		Since $\XX_{\U}(\epsilon)$ can be covered by such $U$, this completes the proof in the case of $n\in \Z_{\geq 1}$.
		
		The case of $n=0$ follows from the case of $n=1$ since we have $\omega^{\kappa,+}=\AL_{ \ast}\omega_{1}^{\kappa,+}=\omega_{\AIP}^{\kappa,+}$. For $n=\infty$, the same argument works for the diagram
		\begin{equation}\label{dg:comparison-mod-forms-AIP-forms-inf}
		\begin{tikzcd}[row sep = 0.55cm]
		\mathcal X^{\ast}_{\Gamma(p^\infty)}(\epsilon)_a \arrow[r, "\gamma"] \arrow[d, "d\times\mathfrak s"'] & \mathcal X^{\ast}_{\Gamma(p^\infty)}(\epsilon)_a \arrow[d, "\mathfrak s"'] \arrow[r] & \mathcal X^{\ast}_{\Gamma_{\!\scaleto{0}{3pt}}(p^\infty)}(\epsilon)_a \arrow[d,"q"] \\
		\Z_p^\times\times \mathcal F_m(\epsilon) \arrow[r, "\mathrm m"] & \mathcal F_m(\epsilon) \arrow[r] & \mathcal X^{\ast}(\epsilon).
		\end{tikzcd}
		\end{equation}
		which induces an isomorphism $\mathfrak s^{\ast}:q^{\ast}\omega^+_{\AIP}\to \omega_\infty^{\kappa,+}$ as desired. 
		Lastly, the statement about Hecke equivariance  is a special case of Prop.~\ref{prop: hecke equiv}.
	\end{proof}
	
	%
	%
	
	\section{Perfectoid Hilbert modular varieties for $G^*$}
	\label{sec:hilbert modular varieties}
	For the remainder of the paper, we move on to Hilbert modular forms. In this section, we recall the classical Hilbert modular varieties for $G$ and $G^*$, and the perfectoid versions for $G^*$.  
	
	\subsection{Classical Hilbert modular varieties for $G$ and $G^*$}
	\textcolor{black}{The content of this section is mainly classical and can be found in many sources, e.g.\ \cite[\S2]{tian} and \cite[\S4]{hidapadic}.}
	\begin{nota}
		\begin{enumerate}
			\item  Let $F$ be a totally real field of degree $g$ over $\QQ$ with ring of integers $\O_F$ and absolute different $\mathfrak{d}\subseteq \O_F$. Let $\Sigma$ denote the set of infinite places of $F$.
			\item Recall that we fixed a rational prime $p$. We set $\OO_p:=\OF \otimes \ZZ_p = \oplus_{\ps|p} \OO_{F_\ps}$, where $\ps|p$ ranges over the prime ideals of $\OO_F$ over $p$ and where $F_{\ps}$ is the completion of $F$ at $\ps$.
			\item	For any fractional ideal $\mathfrak{r}$, let $\mathfrak{r}^+$ denote the totally positive elements in $\mathfrak{r}$; in general `$+$' will denote `totally positive'. We have a non-degenerate trace pairing $\ttr: \mathfrak{r} \times \mathfrak{r}^{-1}\mathfrak{d}^{-1} \to \ZZ$.
			\item  Let $G:=\Res_{F/\QQ} \GL_2$ and let $G^*:=G \times_{\Res_{F/\QQ} \mathbb{G}_m} \mathbb{G}_m$, where the map $G \to \Res_{F/\QQ} \mathbb{G}_m$ is given by the determinant morphism and {\color{black}$\mathbb{G}_m \to \Res_{F/\QQ} \mathbb{G}_m$} is given by the diagonal map.
		\end{enumerate}
	\end{nota}


	Let $\mathcal{S} := \mathcal{H}^\Sigma$, where $\mathcal{H} \subset \C$ is the standard upper half-plane. For $K$ an open compact subgroup of $G(\AA_f)$, there exists a Shimura variety $\Sh_K(G,\mathcal{S})$ of level $K$ over $\QQ$. Similarly, if $K^*$ is an open compact subgroup of $G^*(\AA_f)$, there exists a Shimura variety $\Sh_{K^*}(G^*,\mathcal S)$ over $\QQ$. 
	These Shimura varieties are the \emph{Hilbert modular varieties}. The Shimura variety $\Sh_{K^*}(G^*,\mathcal S)$ is of PEL type, therefore of Hodge type, whereas the Shimura variety $\Sh_K(G,\mathcal S)$ is of abelian type.
	
	\begin{defn}
		Let $\mathfrak c\subseteq \O_F$ be any nonzero ideal and let $N$ be coprime to $\mathfrak c$.
		\begin{enumerate}
			\item Let $K_{\mathfrak c}:= G(\AA_f)\cap \smallmatrd{\widehat{\O}_F}{(\mathfrak c\mathfrak d)^{-1}\widehat{\O}_F}{\mathfrak c\mathfrak d\widehat{\O}_F}{\widehat{\O}_F}$. Let $K^{\ast}_{\mathfrak c}:= G^{\ast}(\AA_f)\cap K_{\mathfrak c}$.
			\item Let $K_0(\mathfrak c,N):=\{\gamma\in K_{\mathfrak c}|\gamma\equiv \smallmatrd{\ast}{\ast}{0}{\ast}\bmod N\}$ and $K^{\ast}_0(\mathfrak c,N):=K_0(\mathfrak c,N)\cap K^{\ast}_{\mathfrak c}$.
			\item Let $K_1(\mathfrak c,N):=\{\gamma\in K_{\mathfrak c}|\gamma\equiv \smallmatrd{\ast}{\ast}{0}{1}\bmod N\}$. Let $K^{\ast}_1(\mathfrak c,N):=K_1(\mathfrak c,N)\cap K^{\ast}_{\mathfrak c}$.
			\item Let $K(\mathfrak c,N):=\{\gamma\in K_{\mathfrak c}|\gamma\equiv \smallmatrd{1}{0}{0}{1}\bmod N\}$. Let $K^{\ast}(\mathfrak c,N):=K(\mathfrak c,N)\cap K^{\ast}_{\mathfrak c}$.
		\end{enumerate}
	\end{defn}


	\subsubsection{Moduli problems for Hilbert--Blumenthal abelian varieties}\label{Section: lvl HMFs}

	
	Hilbert modular varieties arise as solutions to moduli problems of abelian varieties with level structures.
	
	\begin{defn}\label{HBAVC} \label{rap cond}
		Let $S$ be any scheme. Let $\gothc\subseteq \OO_F$ be an ideal. 
		\begin{enumerate}
			\item A Hilbert--Blumenthal abelian variety (HBAV) over $S$ is a triple $(A,\iota,\lambda)$ consisting of an abelian variety $A$ over $S$ with real multiplication $\iota:\O_F \hookrightarrow \End(A)$ and a $\mathfrak c$-polarisation $\lam : A\otimes \gothc \isorightarrow A^\vee$, such that $\iota$ is stable under the Rosati-involution.
			\item Given a HBAV $(A,\iota,\lambda)$, we refer to $(A,\iota)$ as the underlying abelian variety with real multiplication (AVRM).
			\item A morphisms of HBAVs is an $\OF$-linear morphisms $f:A \to A'$ of abelian $S$-schemes for which $ \lam=f^{\vee} \circ \lam' \circ f$.
		\end{enumerate}
	\end{defn}
	
	The Shimura varieties for $G^{\ast}$ represent moduli problems given by $\mathfrak c$-polarised HBAVs with additional rigidifying level structure, which we shall discuss next. In contrast, the Shimura varieties for $G$ are only coarse moduli spaces parametrising triples $(A,\iota,[\lambda])$ where $(A,\iota)$ is an AVRM and $[\lambda]=\{\nu \lam | \nu \in \OF^{\times,+} \}$ is a \emph{polarisation class}, plus level structure.
	

	\begin{defn}[Level structures for $G$]\label{lvls for G*}

		Let  $(A,\iota,\lambda)$ be a $\gothc$-polarised HBAV, let $N\in \ZZ_{\geq 4}$ with $(N,p)=1$ and let $n\in \Z_{\geq 0}$.
		
		\begin{enumerate}
			\item  A $\mu_{N}$-level structure is a closed immersion of $\OO_F$-module schemes $\mathfrak d^{-1}\otimes_{\ZZ} \mu_N \hookrightarrow A[N]$.
			
			\item  A $\Gamma_{0}(p^n)$-level structure is an $\OO_F$-submodule scheme $\Phi_n:C\hookrightarrow A[p^n]$ that is \'etale locally isomorphic to $\OO_F/p^n\OO_F$. Via $\lambda$, this is equivalent to giving an $\OO_F$-submodule scheme \[C':=\lambda(C\otimes\gothc) \hookrightarrow A^{\vee}[p^n].\]
			In the context of Shimura varieties for $G^*$, we also call the same data a $\Gamma^{\ast}_{0}(p^n)$-level structure, since this level structure appears for both $G^*$ and $G$.
			\item  A $\Gamma_{1}(p^n)$-level structure is a closed immersion of $\OO_F$-module schemes \[\Phi_n:\OF/p^n\OF\hookrightarrow A^\vee[p^n].\]
			Again, we also call this a $\Gamma^{\ast}_{1}(p^n)$-level structure,
			\item A $\Gamma(p^n)$-level structure is an isomorphism of $\OO_F$-module schemes
			\[
			\alpha_{n} : (\OO_F/p^n\OO_F)^2\isorightarrow A^\vee[p^n].
			\]
		\end{enumerate}
		
	\end{defn}
	\begin{remark}\label{rem:tate dual c}
		Our definition of level structures at $p$ is slightly non-standard: usually one would define a $\Gamma(p^n)$-structure to be an isomorphism $(\OO_F/p^n\OO_F)^2\isorightarrow A[p^n]$. The reason we use the above modified version parametrising $A^\vee[p^n]$ is that the Hodge--Tate morphism is of the form $T_pA^{\vee}\to \omega_{A}$, so it is \emph{this} level structure which gives rise at infinite level to canonical sections of $\omega_{A}$, as required for the definition of modular forms. The isomorphisms $(\OO_F/p^n\OO_F)^2 \isorightarrow A[p^n]$ would instead give sections of $\omega_{A^\vee}$.	
		We note, however, that given a fixed $\gothc$-polarisation $\lambda$ as in the moduli problem for $G^*$, a $\Gamma(p^n)$-level structure is equivalent to an isomorphism \textcolor{black}{\[(\OO_F/p^n\OO_F)^2 \otimes \gothc^{-1} \xrightarrow{\alpha_n\otimes \gothc^{-1}} A^\vee[p^n]\otimes \gothc^{-1} \xrightarrow{\lambda^{-1}} A[p^n].\]} In this case our notion is (non-canonically) isomorphic to the more standard definition.
	\end{remark}
	
	We now define level structures for $G^*$. Recall that we have defined $\Gamma_{\!0}^* = \Gamma_{\!0}$ and $\Gamma_1^* = \Gamma_1$-level structures to be the same, but at full level we need a slightly different definition, analogous to \cite[\S1.21]{RapoportCompactifications}. To motivate this, observe that we can see $G$ as a group preserving a pairing up to similitude, whilst $G^*$ is the subgroup that preserves certain rational structures within this. For the Shimura varieties associated to $G^{\ast}$, we therefore need isomorphisms $\alpha_{n}$ preserving a rational structure. We will now define the $\OO_F$-structure in which the rational structure should live, via an $\OO_F$-linear version of the Weil pairing.

	\begin{definition}\label{def:linearisation}
		
		Let $(A,\iota,\lambda)$ be a HBAV. The Weil pairing 
		$e_{p^n}:A[p^n]\times A^{\vee}[p^n]\to \mu_{p^n}$
		satisfies $e_{p^n}(ax,y)=e_{p^n}(x,ay)$ for $a \in \OF$ (see \cite[Section 20]{mumford}). The Weil pairing can therefore be extended to an $\OO_F$-linear version, by using the trace map to write $\mathfrak d^{-1} \simeq \Hom(\OO_F,\ZZ)$, setting
		\begin{align}\label{df:tilde e_n}
		\widetilde{\mathbf e}_{n}:A[p^n]\times A^{\vee}[p^n]&\longrightarrow \mathfrak d^{-1}\otimes_{\ZZ}\mu_{p^n}, \\
		\quad x,y &\longmapsto (a\mapsto e_{p^n}(ax,y))\notag
		\end{align}
		which is $\OO_F$-bilinear and perfect. We call $\widetilde{\mathbf{e}}_n$ the \emph{$\OO_F$-linearisation} of $e_{p^n}$. Note $e_{p^n} = \mathrm{Tr}\circ\widetilde{\mathbf{e}}_n$.
		
	\end{definition}
	
	We fix a non-degenerate $\O_F$-linear pairing on $(\OO_F/p^n\OF)^2 \otimes \gothc^{-1}\times (\OO_F/p^n\OF)^2$ to compare to the Weil pairing. To this end, fix an isomorphism of free $\OO_p$-modules of rank 1 \[\beta:\mathfrak c^{-1} \OO_p\to \mathfrak d^{-1}\otimes_{\Z}T_p\mu_{p^\infty}=:\d^{-1}(1),\] or equivalently an $\O_p$-module generator $\beta\in \mathfrak c\mathfrak d^{-1}\otimes_{\Z}T_p\mu_{p^\infty}=\gothc \d^{-1}(1)$. Then we get a pairing
	\begin{align*}
	\langle -,-\rangle_{\beta,n}:(\OO_F/p^n\OO_F)^2\otimes \mathfrak c^{-1}\times (\OO_F/p^n\OO_F)^2&\xrightarrow{\beta\times\id} (\mathfrak d^{-1}/p^n)^2\otimes_{\Z}\mu_{p^n} \times(\OO_F/p^n\OO_F)^2\\ &\xrightarrow{\det}  \mathfrak{d}^{-1}\otimes_{\Z}\mu_{p^n},
	\end{align*}
	where $\det : [(a,b),(c,d)] \mapsto ad-bc$. Given a $\Gamma(p^n)$-level structure $\alpha_n$, this fits into a diagram
	\begin{equation}\label{eq:similitude}
	\begin{tikzcd}[column sep = {1.6cm,between origins},row sep = 0.55cm]
	(\mathcal O_F/p^n\mathcal O_F)^2\otimes\mathfrak c^{-1} \arrow[d,"{\lambda^{-1}\circ(\a_{n}\otimes \mathrm{id})}"] &\quad\times& (\mathcal O_F/p^n\mathcal O_F)^2 \arrow[d,"\alpha_{n}"]\arrow[rrr,"{\langle -,-\rangle_{\beta,n}}"] &&& \mathfrak{d}^{-1}\otimes_{\Z}\mu_{p^n}\arrow[rr,"\mathrm{Tr}"] \arrow[d,"b","\sim"labl] && \mu_{p^n} \arrow[d,dotted] \\
	{A[p^n]} &\quad\times & {A^{\vee}[p^n]} \arrow[rrr,"\widetilde{\mathbf{e}}_{n}"] &&& \mathfrak{d}^{-1} \otimes_{\Z}\mu_{p^n} \arrow[rr,"\mathrm{Tr}"]&& \mu_{p^n}.
	\end{tikzcd}
	\end{equation}
	The two pairings into $\mathfrak{d}^{-1}\otimes_{\Z}\mu_{p^n}$ will always be similar, that is there always exists some $b \in \mathrm{Aut}(\mathfrak{d}^{-1} \otimes_{\Z}\mu_{p^n}) = (\OO_F/p^n\OO_F)^\times$ that makes the above diagram commute. 
	
	\begin{definition}\label{d:Gamma^* lvl struct and action by precompose with dual}
		A $\Gamma^\ast(p^n)$-level structure is a $\Gamma(p^n)$-level structure $\alpha_{n}:(\OO_F/p^n)^2\isorightarrow A^\vee[p^n]$ such that the similitude $b$ in \eqref{eq:similitude} lies in the subgroup $(\ZZ/p^n\ZZ)^\times\subseteq (\O_F/p^n\O_F)^\times$. Equivalently, it is an $\alpha_{n}$ such that after composing both pairings with the trace map -- the final horizontal maps of \eqref{eq:similitude} -- the two pairings into $\mu_{p^n}$ remain similar via $b \in \mathrm{Aut}(\mu_{p^n}) = (\ZZ/p^n\ZZ)^\times$.
		
		Via its natural action on $(\OO_F/p^n\OO_F)^2$, the group $G^{\ast}(\Z_p/p^n\Z)\subseteq \GL_2(\OO_F/p^n\OO_F)$ acts on $\Gamma^\ast(p^n)$-level structures by letting $\gamma$ act as pre-composition with $\gamma^{\vee}:=\det(\gamma)\gamma^{-1}$. 
	\end{definition}
	\begin{rmrk}
		
		We note that the dual ensures that we obtain a left-action. One reason for us to use $\gamma^{\vee}$ rather than $\gamma^{-1}$ to define the action is Lem.~\ref{l:action-of-G(Q_p)} below.
	\end{rmrk}

	In the limit $n\to \infty$, the pairings $\langle -,-\rangle_{\beta, n}$ are compatible and define a pairing $\langle -,-\rangle_{\beta}:\O_p^2\otimes \mathfrak c^{-1}\times \O_p^2\to \mathfrak{d}^{-1}(1)$. A $\Gamma^{\ast}(p^{\infty})$-structure is a compatible collection of $\Gamma^{\ast}(p^{n})$-level structures. Equivalently, this is an isomorphism $\O_p^2\to T_pA^\vee$ inducing a (rational) similitude of pairings on $\O_p^2\otimes \c^{-1}\times \O_p^2\to T_pA\times T_pA^{\vee}$ as above.\footnote{The above choices may not seem natural, but we remark that they do not affect our construction for modular forms $G$, and are just an auxiliary choice in our construction of $G^*$-forms. We also note that the Shimura varieties are of course completely canonical, it is the moduli interpretation which requires the choice.}

	\subsubsection{Hilbert modular varieties as moduli spaces}\label{s:Hilbert modular varieties as moduli spaces}
	We now recall the moduli problems that the Shimura varieties $\Sh_{K^*}(G^*,\mathcal S)$ represent. Consider the functor 
	\[
	\Sch/_{\ZZ[1/N]}  \lra \Set,
	\]
	sending a scheme $S$ to the set of isomorphism classes of $(A,\iota,\lam,\mu)$, where $(A,\iota,\lam)$ is a $\mathfrak c$-polarised HBAV and $\mu$ is a $\mu_N$-level structure which we call the tame level. This functor is represented by the Hilbert moduli scheme $X(\gothc,\mu_N)_{\ZZ[1/N]}$ over ${\ZZ[1/N]}$. We denote by $X(\gothc,\mu_N)_{R}$ the base change to any $\ZZ[1/N]$-algebra $R$. Then there is an isomorphism 
	\[X(\gothc,\mu_N)_{\C}\cong \Sh_{K^{\ast}_1(\mathfrak c,N)}(G^*,\mathcal S).\]
	Thus $X(\gothc,\mu_N)_{\ZZ[1/N]}$ is a model for $\Sh_{K^{\ast}_1(\mathfrak c,N)}(G^*,\mathcal S)$ over $\Z[1/N]$. We write $X:=X(\gothc,\mu_N)_{L}$ 
	
	In the case of $G$, the Shimura variety $\Sh_{K_1(\mathfrak c,N)}(G,\mathcal{S})$ also has a canonical model $X_G(\gothc,\mu_N)_{\ZZ[1/N]}$, but this is only a coarse moduli space; for an algebraically closed field $C$, and appropriate $K$, the $C$-points of $X_G(\gothc,\mu_N)_{\C}$ parametrise isomorphism classes of tuples $(A, \iota, [\lambda], \mu_N)$, where  $[\lambda]$ is a polarisation class. We shall write $X_G:=X_G(\gothc,\mu_N)_{L}$.
	
	We obtain the analogous results when we now add level structures at $p$. Let $\Gamma^{\ast}_p$ be one of $\Gamma^{\ast}_0(p^n)$, $\Gamma^{\ast}_1(p^n)$ or $\Gamma^{\ast}(p^n)$, and let $K^{\ast}_p$ be the corresponding subgroup $K^{\ast}_0(\gothc,p^n)$, $K^{\ast}_1(\gothc,p^n)$ or $K^{\ast}(\gothc, p^n)$.  Consider the functor $\Sch/_{\QQ}  \to \Set$
	sending a scheme $S$ to the set of isomorphism classes of $(A,\iota,\lam,\mu,\alpha)$, where $(A,\iota,\lam,\mu)$ is as above, and $\alpha$ is a $\Gamma^{\ast}_p$-level structure.  This functor is represented by a Hilbert moduli scheme $X(\gothc,\mu_N,\Gamma^{\ast}_p)_{\QQ}$ whose base-change to $\C$ is isomorphic to $\Sh_{K^{\ast}_1(\mathfrak c,N)\cap K^{\ast}_p}(G^*,\mathcal S)$. Here for $\Gamma^{\ast}_p=\Gamma^{\ast}(p^n)$ we recall that our definition of this level structure depends on our chosen isomorphism $\beta$, and the isomorphism is given by transforming our notion of a $\Gamma^{\ast}(p^n)$-level structure into the usual notion, by using $\beta$, the complex unit root $\exp(2 \pi i/p^n)$ and the given polarisation to make $(\O_F/p^n\OF)^2\to A^{\vee}[p^n]$ into an isomorphism $\mathfrak d^{-1}\otimes \mu_{p^n}\times \O_F/p^n\OF\to A[p^n]$.
	
	Again we abbreviate $X_{\Gamma^{\ast}_p}:=X(\mathfrak c,\mu_N,\Gamma^{\ast}_p)_L$. By the natural forgetful map between the moduli functors, we can summarise these Hilbert moduli schemes in the tower of moduli schemes
	\[ X_{\Gamma^{\ast}(p^n)}\longrightarrow X_{\Gamma^{\ast}_1(p^n)} \longrightarrow X_{\Gamma^{\ast}_0(p^n)} \longrightarrow X.\]
	Let $X_{\Gamma_p^*}^\ast$ denote the minimal compactification of $X_{\Gamma_p^*}$ over $L$. For tame level, we also have the minimal compactification $X^{\ast}_{\O_L}$ over $\O_L$. These can all be defined via base-change from $\Q$ or $\Z[1/N]$, respectively.
	
	Finally in this subsection, we pass to $p$-adic geometric spaces:
	\begin{defn}
		Let $\Gamma^{\ast}_p$ be a level at $p$ of the form $\Gamma_0^*(p^n), \Gamma_1^*(p^n),\Gamma^*(p^n)$ for some $n\in \Z_{\geq 0}$.
		\begin{enumerate}
			\item Let $\GX^\ast\to \Spf(\O_L)$ be the $p$-adic completion of $X_{\O_L}\to \Spec(\O_L)$.
			\item We denote by $\XX_{\Gamma^{\ast}_p}$ the adic analytification of $X_{\Gamma^{\ast}_p}$. We thus obtain a tower of adic spaces
			\[ \XX_{\Gamma^{\ast}(p^n)}\longrightarrow \XX_{\Gamma^{\ast}_1(p^n)} \longrightarrow \XX_{\Gamma^{\ast}_0(p^n)} \longrightarrow\XX.\]
			\item In the case of tame level, we also have the compactification $\GX^\ast\to\Spf(\OO_L)$ obtained by completion of $X_{\O_L}^{\ast}$. Its adic generic fibre coincides with the analytification  $\XX^\ast$  of  $X^{\ast}$. 
			
		\end{enumerate}
		
	\end{defn}
	
	\begin{rmrk}
		Our notation suppresses the dependence on the polarisation ideal $\gothc$ (and on the tame level). If required, we will make this clear by writing $\GX_{\gothc,K_p^*}^*$, $\XX_{\gothc,K_p^*}^*$, etc.
	\end{rmrk}

	%
	%
	
	\subsection{Hilbert modular varieties for $G^*$ at infinite level}
	\label{sec:hbv at infinite level}
	Next, we recall the perfectoid Hilbert modular varieties for the group $G^*$, following \cite{torsion} and \cite{CarScho}.

	Let $\XXt(\e)$ denote the admissible open subset of $\XXt$ defined by $|\tHa| \geq |p|^{\e}$, where $\tHa$ denote any local lift of the (total) Hasse invariant. Now, for $n \in \ZZ_{\geq 0}$ take $\e \in [0, (p-1)/p^n) \cap \log |L|$. Then by \cite[Section 3.2]{AIP3} we have an integral model  $\GX^*(\e)$ of $\XXt(\e)$.  Moreover, abelian schemes parametrised by $\XXt(\e)$ have a canonical subgroup of level $n$ which agrees with the kernel of the $n$-th iterated power of the Frobenius map modulo $p^{1-\e}$. Following \cite{torsion}, we let $\XXt_{\Gamma_{\!0}^*(p^n)}(\e)_a$ denote the anticanonical locus, parametrising subgroups that intersect the canonical subgroup trivially.
	The forgetful maps then extend to the minimal compactification and give the anticanonical tower 
	\[\dots \longrightarrow \XXt_{\Gamma_{\!\scaleto{0}{3pt}}^*(p^2)}(\e)_a \longrightarrow \XXt_{\Gamma_{\!\scaleto{0}{3pt}}^*(p)}(\e)_a \longrightarrow \XXt(\e).\]
	One then has the following analogue of the results in \cite[\S III]{torsion}:
	\begin{thm}\label{p:X_Gamma*_0(p^infty) for G*}
		There are perfectoid spaces that are the tilde-limits
		\begin{enumerate}
			\item $\XX_{\Gamma_{\!\scaleto{0}{3pt}}^*(p^\infty)}(\e)_a \sim \varprojlim_n \XX_{\Gamma_{\!\scaleto{0}{3pt}}^*(p^n)}(\e)_a$.
			\item $\XX_{\Gamma_{\!\scaleto{1}{3pt}}^*(p^\infty)}(\e)_a \sim \varprojlim_n \XX_{\Gamma_{\!\scaleto{1}{3pt}}^*(p^n)}(\e)_a$.
			\item $\XX_{\Gamma^*(p^\infty)}(\e)_a \sim \varprojlim_n \XX_{\Gamma^*(p^n)}(\e)_a$.
			\item $\XX_{\Gamma^*(p^\infty)} \sim \varprojlim_n \XX_{\Gamma^*(p^n)}.$
		\end{enumerate}
		They fit into the following tower of pro-\'etale torsors for the indicated profinite groups
		\begin{center}
			\begin{tikzcd}[row sep = {1cm,between origins}]
			&  &  & \mathcal X_{\Gamma^*(p^\infty)}(\epsilon)_a  \arrow[d] \arrow[dd,dash,dashed,swap,"\Gamma_{\!\scaleto{0}{3pt}}^*(p^\infty)" description, rounded corners, to path={	-- ([xshift=-1.5cm]\tikztostart.west)
				|- (\tikztotarget) [pos=0.25] \tikztonodes
			}   ]  \arrow[ddd,dash,dashed,swap,"\Gamma_{\!\scaleto{0}{3pt}}^*(p^n)" description,rounded corners, to path={	-- ([xshift=-2.5cm]\tikztostart.west)
				|- (\tikztotarget) [pos=0.25] \tikztonodes
			}   ]   & &  \\
			&  & & \mathcal X_{\Gamma_1^*(p^\infty)}(\epsilon)_a  \arrow[d]  \ar[d, dash,dashed,rounded corners,
			"\OO_p^\times" description, swap,
			to path={
				-- ([xshift=-0.5cm]\tikztostart.west)
				|- (\tikztotarget) [pos=0.25] \tikztonodes
			}]& &  \\
			& {} & {} & \mathcal X_{\Gamma_{\!\scaleto{0}{3pt}}^*(p^\infty)}(\epsilon)_a \arrow[d]  & &  \\
			{} &  & & \mathcal X_{\Gamma_{\!\scaleto{0}{3pt}}^*(p^n)}(\epsilon)_a.&  & 
			\end{tikzcd}
		\end{center}
		where $\Gamma_{\!0}^*(p^n)\subseteq G^{\ast}(\Z_p)$ is the subgroup of matrices that are upper triangular mod $p^n$, and $\Gamma_{\!0}^*(p^\infty)\subseteq G^{\ast}(\Z_p)$ is the subgroup of upper-triangular matrices.
	\end{thm}
	
	Part (4) is proved  in \cite[\S Theorem IV.1.1]{torsion} in the case that $L$ is algebraically closed, and under the additional assumption that the embedding $G^{\ast}\hookrightarrow \GSp_{2g}$ sends $K^p$ into a subgroup of $\GSp_{2g}(\hat{\Z})$. In our case, this means that $\mathfrak c=1$. One can deduce the general case from this by first shrinking the level and then quotienting by a Galois action, \textcolor{black}{which by \cite[Theorem 1.4]{Hansenquots} or \cite[Theorem 5.8]{HJperf} is again a perfectoid space}. Part (3) of the proposition then follows by restriction, and one can modify the argument to also obtain statements (1) and (2). 
	
	Alternatively, to prove Theorem~\ref{p:X_Gamma*_0(p^infty) for G*} one can follow Scholze's construction in the Siegel case \cite[\S III]{torsion}, as we shall now demonstrate: here we note that since we ignore the boundary, we do not need to worry about ramification. One first shows (cf \cite[Theorem III.2.15]{torsion}):
	\begin{lem}\label{l:Frob-lift}
		Division by the canonical subgroup defines a natural map
		\[\phi:\mathfrak X^{\ast}(p^{-1}\e)\to \mathfrak X^{\ast}(\epsilon)\]
		that reduces to the relative Frobenius mod $p^{1-\delta}$ where $\delta=\frac{p+1}{p}\epsilon$.
	\end{lem}
	\begin{proof}
		We argue as in \cite[Thm.~III.2.15]{torsion}. Let $ A\to\mathfrak X(p^{-1}\e)$ be the universal abelian scheme. Then the abelian scheme $A':=A/H_1$ defines a morphism
		$\phi:\mathfrak X(p^{-1}\e)\to \mathfrak X.$
		Locally on any affine open $\Spf(R)\subseteq \mathfrak X$ defined over $\Z_p$, this corresponds to a morphism $\phi:R\to R\langle T\rangle/(T\tHa-p^{p^{-1}\epsilon}).$ Since $H_1\equiv \ker F\bmod p^{1-p^{-1}\epsilon}$, we have $A'\equiv A^{(p)}\bmod p^{1-p^{-1}\epsilon}$. Consequently, $\phi$ reduces mod $p^{1-p^{-1}\epsilon}$ to the relative Frobenius on $R$. Here $R^{(p)}=R$ since $R$ is already defined over $\Z_p$. In particular, since $\Ha(A^{(p)})=\Ha(A)^{p}$, we have $\phi(\tHa)\equiv \tHa^p \bmod p^{1-p^{-1}\epsilon}$. Consequently,
		\[
		T^p\phi(\tHa)\equiv T^p\tHa^p =p^{\epsilon} \bmod p^{1-p^{-1}\epsilon}
		\]
		inside $R\langle T\rangle/(T\tHa-p^{p^{-1}\epsilon})$.
		We can therefore find $u\in R$ such that 
		\[ T^p\phi(\tHa) = p^{\epsilon}+p^{1-p^{-1}\epsilon}u=p^{\epsilon}(1+p^{1-\delta}u).\]
		Sending $T\mapsto \phi(T):=T^p(1+p^{1-\delta}u)^{-1}$ therefore defines a unique extension
		\[\phi:R\langle T\rangle/(T\tHa-p^{\epsilon})\to R\langle T\rangle/(T\tHa-p^{p^{-1}\epsilon})\]
		giving the desired lift to a map 
		\begin{equation}\label{m:proof-of-l:Frob-lift}
		\phi:\mathfrak X(p^{-1}\e)\to \mathfrak X(\e).
		\end{equation}
		Since $\phi \bmod p^{1-\epsilon}$ is given by the Frobenius on $R$, and by sending $T\mapsto T^p(1+p^{1-\delta}u)^{-1}\equiv T^p\bmod p^{1-\delta}$, this map reduces to the relative Frobenius mod $p^{1-\delta}$ as desired.
		
		It remains to extend this to the minimal compactification. Since the boundary is already contained in the ordinary locus, it suffices to consider the case of $\epsilon=0$, and thus $\delta=0$. It moreover suffices to consider the case of $L=\Q_p^\cyc$; the general case follows by base-change.
		
		Since $X^{\ast}_{\F_p}$ is normal  (see \cite[Thm.~4.3]{Chai-ArithCompact} and the discussion in \cite[1.8.1]{Kisin-Lai}), and the cusps are of codimension $\geq 2$, we can now apply Scholze's version of Hartog's extension principle, \cite[Lem.~III.2.10]{torsion} to see that $\phi$ extends uniquely over the boundary to a map
		\[\phi:\mathfrak X^{\ast}(0)\to \mathfrak X^{\ast}(0)\]
		which  still reduces to the relative Frobenius mod $p$.
		Glueing this to~\eqref{m:proof-of-l:Frob-lift} proves the lemma.
	\end{proof}
	\begin{proof}[Proof of Thm.~\ref{p:X_Gamma*_0(p^infty) for G*}]
		We can argue as in \cite[Cor.~III.2.19]{torsion}. For every $n\in \N$ we have by Lem.~\ref{l:Frob-lift} a morphism $\phi:\mathfrak X^{\ast}(p^{-n-1}\e)\to \mathfrak X^{\ast}(p^{-n}\e)$ that reduces to the relative Frobenius mod $p^{1-\frac{p+1}{p^{n+1}}\e}$ and in particular mod $p^{1-\delta}$ where $\delta:=\frac{p+1}{p}\e$. Consequently, $\mathfrak X^{\ast}(p^{-\infty}\e):=\varprojlim_{\phi} \mathfrak X^{\ast}(p^{-n}\e)$ is a flat formal scheme for which the relative Frobenius mod $p^{1-\delta}$ is an isomorphism. It follows that the generic fibre $\mathcal X^{\ast}(p^{-\infty}\e)$ is perfectoid and moreover, by \cite[Prop.~2.4.2]{SW}, on the generic fibres we have
		\[\mathcal X^{\ast}(p^{-\infty}\e)\sim \varprojlim_{\phi} \mathcal X^{\ast}(p^{-n}\e).\]
		By \cite[Prop.~2.4.3]{SW}, we can now restrict to the open modular curve to deduce that there is a perfectoid tilde-limit $\mathcal X(p^{-\infty}\e)\sim \varprojlim_{\phi} \mathcal X(p^{-n}\e)$.
		Since the Atkin--Lehner isomorphisms $\AL^n$ define an isomorphism of inverse systems of the anticanonical tower to the system
		\[
		\dots \longrightarrow \mathcal X(p^{-n-1}\epsilon)  \xlongrightarrow{\phi}  \mathcal X(p^{-n}\epsilon)  \longrightarrow \dots\]
		we equivalently have $\mathcal X_{\Gamma^{\ast}_{\!\scaleto{0}{3pt}}(p^\infty)}(\e)_a\sim \varprojlim_n\mathcal X_{\Gamma_{\!\scaleto{0}{3pt}}^{\ast}(p^n)}(\epsilon)_a$, as desired. This proves part (1).
		
		To deduce parts (2) and (3), we use that the forgetful morphism $\XX_{\Gamma_1^*(p^n)}(\e)_a \to \XX_{\Gamma_{\!\scaleto{0}{3pt}}^*(p^n)}(\e)_a$ is finite \'etale. By pulling these back from varying $n$, we obtain a tower of finite \'etale morphisms
		\begin{center}
			\begin{tikzcd}[row sep = 0.55cm]
			\XX_{\Gamma_{\!\scaleto{0}{3pt}}^*(p^\infty)}(\e)_a  \ar[d]  & \XX_{\Gamma_{\!\scaleto{0}{3pt}}^*(p^\infty) \cap \Gamma_1^*(p^n)}(\e)_a  \arrow[l] \ar[d] & \arrow[l] \dots \\
			\XX_{\Gamma_{\!\scaleto{0}{3pt}}^*(p^n)}(\e)_a & \arrow[l]	\XX_{\Gamma_1^*(p^n)}(\e)_a & 
			\end{tikzcd}
		\end{center}
		Since perfectoid tilde-limits of inverse systems of perfectoid spaces with affinoid transition maps exist, we obtain a perfectoid tilde-limit \[\XX_{\Gamma_1^*(p^\infty)}(\e)_a \sim \varprojlim_n \XX_{\Gamma_{\!\scaleto{0}{3pt}}^*(p^\infty) \cap \Gamma_1^*(p^n)}(\e)_a.\]
		This proves part (2). Part (3) follows similarly using that $\XX_{\Gamma_1^*(p^n)} \to \XX_{\Gamma_{\!\scaleto{0}{3pt}}^*(p^n)}$ is finite \'etale.
		
		Since the morphisms $\XX_{\Gamma_1^*(p^n)}(\e)_a \to \XX_{\Gamma_{\!\scaleto{0}{3pt}}^*(p^n)}(\e)_a$ and  $\XX_{\Gamma^*(p^n)}(\e)_a \to \XX_{\Gamma_{\!\scaleto{0}{3pt}}^*(p)}(\e)_a$ and $\XX_{\Gamma^*(p^n)}(\e)_a \to \XX_{\Gamma_{\!\scaleto{0}{3pt}}^*(p^n)}(\e)_a$ are  finite \'etale torsors for the groups $(\OF/p^n\OF)^\times$, $\{\smallmatrd{\ast}{\ast}{c}{\ast}\in \G^{\ast}(\Z/p^n\Z)|c\in p\OF/p^n\OF\}$ and $\{\smallmatrd{\ast}{\ast}{0}{\ast}\in \G^{\ast}(\Z/p^n\Z)\}$, respectively, the last statement follows from the fact that perfectoid tilde-limits commute with fibre products.
		
		It remains to prove (4), which we deduce from (3) using the $G^{\ast}(\Q_p)$-action at infinite level recalled in  \S\ref{s:the action og G(Q_p)}: like in \cite[\S III.3]{torsion}, it suffices to prove that on the level of topological spaces we have $G^{\ast}(\Q_p)|\XX_{\Gamma(p^\infty)}(\e)_a|=|\XX_{\Gamma(p^\infty)}|:=|\varprojlim_n \XX_{\Gamma(p^n)}|$. But as it suffices to prove this after passing to a smaller $K^p$ and any field extension of $C$, we can reduce to the case considered in \cite[Theorem~IV.1.1]{torsion}. This finishes the proof of the theorem.
	\end{proof}
	
	\begin{rmrk}\label{rem: hilb moduli}
		As in the elliptic case, we have a moduli description of the $(C,C^+)$-points of $\XX_{\Gamma^*(p^\infty)}$ for any perfectoid extension $C$ of $L$: They are in functorial one-to-one correspondence with isomorphism classes of tuples $(A,\iota,\lam,\mu_N,\a)$ where $(A,\iota,\lam)$ is an $\epsilon$-nearly ordinary $\gothc$-polarised HBAV over $C$ with tame level $\mu_N$,  together with a $\Gamma^{\ast}(p^\infty)$-level structure $\a: \OO_p^2 \isorightarrow T_p A^{\vee}$. The subspace $\XX_{\Gamma_0^*(p)}(\e)_a$ represents those tuples for which $\a(1,0)$ generates a subgroup of $A^{\vee}[p]$ that is different to the canonical subgroup.
	\end{rmrk}
	
	\begin{definition}
		The $G^{\ast}(\ZZ/p^n\ZZ)$-actions on $\XX_{\Gamma^*(p^n)}$ in the limit give rise to a $G^{\ast}(\ZZ_p)$-action on  $\XX_{\Gamma^*(p^\infty)}$ which in terms of moduli can be described as follows:  the action of $\gamma\in  G^{\ast}(\ZZ_p)\subseteq \GL_2(\O_p)$ sends any HBAV $(A,\iota,\lambda,\alpha:\O_p^2\to T_pA^\vee)$ to $(A,\iota,\lambda,\alpha\circ \gamma^{\vee})$.
	\end{definition}

	\subsection{The Hodge--Tate period morphism and its image}\label{section: ht morph in hilbert case}

\textcolor{black}{For any adic space $S$ over $\Spa(L)$, we denote by $\Res_{\O_F|\Z}S$ the functor on affinoid  $(L,\O_L)$-algebras given by $(R,R^+)\mapsto S(R_F,R_F^+),$ 
where $R_F:=R\otimes_{\Q}F$ and $R_F^+$ is the integral closure of $R^+\otimes_{\Z}\O_F$ in $R_F$. If $S=X^{\an}$ is the analytification of a variety $X$ over $L$ for which the usual restriction of scalars $\Res_{\O_F|\Z}X$ is representable by a variety, we have $\Res_{\O_F|\Z}S=(\Res_{\O_F|\Z}X)^{\an}$. For all spaces we need below, this shows that $\Res_{\O_F|\Z}S$ is representable by an adic space.}

 For example, $\Res_{\O_F|\Z}\PP^1$ is the adic analytification of the finite type scheme representing the functor that sends any $L$-algebra $R$ to the set $\P^1(R\otimes_{\Z}\O_F)$. This is the flag variety of $G^{\ast}$. By \cite[Thm.~2.1.3]{CarScho}, there is a Hodge--Tate period map of the form 
	\[
	\pi_{\HT}:\XX_{\Gamma^*(p^\infty)} \to \Res_{\O_F|\Z}\PP^{1}.
	\]
	
	\begin{rmrk}\label{r:mod-interpret-of-pi_HT-Hilbert}
		On points, this map has the following moduli interpretation: let $C/L$ be a complete algebraically closed field and let $A$ a $\mathfrak c$-polarised HBAV over $C$. Then the Hodge--Tate filtration is  a short exact sequence of $C\otimes_{\Z_p}\O_p$-modules
		\[ 
		0 \lra \Lie(A^{\vee})(1) \lra T_pA^{\vee} \otimes_{\ZZ_p} C \xrightarrow{\HT_A} \omega_{A} \lra 0
		\]
		Now, a point $x\in \XX_{\Gamma^*(p^\infty)}(C,C^+)$ gives rise to a trivialisation $\O_{p}^{2}\xrightarrow{\sim}T_p A^{\vee}$ which we can use to consider the above as a filtration of $\O_{p}^{2}\otimes_{\Z_p} C$ of rank $1$. This defines the desired point $\pi_{\HT}(x)\in \Res_{\O_F|\Z}\P^1(C,C^+)=\P^1(\O_{p}\otimes_{\Z_p} C)$. 
	\end{rmrk}
	
	For the definition of Hilbert modular forms, it will be important for us to bound the image of the anticanonical locus under the Hodge--Tate period map.
	More precisely, our goal is to compare this to a family of neighbourhoods of $\P^1(\O_p)\subseteq \Res_{\OO_F|\ZZ}\P^1$ which we shall now define.
	\begin{definition}
		Recall from Defn.~\ref{d:G_a,G_m,and hats} that we had defined adic groups $\G_a$, $\G_m$, $\hat{\G}_a$, $\hat{\G}_m$.
		\begin{enumerate}
			\item By applying the functor $\Res_{\OO_F|\ZZ}-$, we obtain adic spaces $\Res_{\OO_F|\ZZ}\G_m$,  $\Res_{\OO_F|\ZZ}\G_a$, and open subspaces $\Res_{\OO_F|\ZZ}\hat{\G}_m$,  $\Res_{\OO_F|\ZZ}\hat{\G}_a$.
			\item Given a point $x\in \Res_{\OO_F|\ZZ}\G_a$, and an element $z\in L$ with $|z|=r$, we shall call the open subspace $B_r(x):=x+z\Res_{\OO_F|\ZZ}\hat{\G}_a\subseteq {\G}_a$ the ball of radius $r$ around $x$.
		\end{enumerate}
	\end{definition}
	\begin{defn}\label{d:embed-profinite-O_p-into-ResP^1}
		The subspace $\PP^1(\OO_p)=\Res_{\OF|\ZZ}\PP^1(\Z_p)$ is a profinite set, and therefore has a geometric incarnation as a morphism
		$\PP^1(\OO_p) \rightarrow \Res_{\OF|\ZZ}\PP^1$, where as usual we also write $\PP^1(\OO_p)$ for the associated profinite perfectoid space.

		We embedded $\G_a\hookrightarrow \PP^1$ via $z\mapsto (z:1)$. By applying $\Res_{\OF|\ZZ}$, this defines an open subspace $\Res_{\OF|\ZZ}\G_a\hookrightarrow \Res_{\OF|\ZZ}\PP^1$.
		We also have $\OO_p=:B_0(\OO_p:1)\hookrightarrow\PP^1(\OO_p)$ via $a \mapsto (a:1)$ for $a\in \OO_p$. For $r \in (0,1]\cap |L|$, we define the open neighbourhood  $B_r(\OO_p:1)\subseteq \Res_{\OO_F|\ZZ}\hat{\G}_a\subseteq \Res_{\OF|\ZZ}\PP^1$ of $B_0(\O_p:1)$ to be the union of all balls of radius $r$ around points in $\O_p\hookrightarrow \hat{\G}_a\subseteq \Res_{\OF|\ZZ}\PP^1$. We make analogous definitions for open subspaces $B_r(\OO^\times_p:1)$ and $B_r(1:p\OO_p)$ of $\Res_{\OF|\ZZ}\PP^1$.
	\end{defn}
	
	\begin{prop}\label{prop: epsilon w for HMFS}
		Let  $1>r\geq 0$. Then for any $m\in\ZZ_{\geq 1}$ with $1/p^{m}\leq r$ and any $0\leq \epsilon\leq 1/2p^{m}$, or $\epsilon\leq 1/3p^{m}$ if $p=3$, or $\epsilon\leq 1/4p^{m}$ if $p=2$, we have:
		\begin{enumerate}
			\item $ \pi_{\HT}(\XX_{\Gamma^*(p^\infty)}(\epsilon)_c)\subseteq B_r(1:p\OO_p)$,
			\item $ \pi_{\HT}(\XX_{\Gamma^*(p^\infty)}(\epsilon)_a)\subseteq B_r(\O_p:1)$.
		\end{enumerate}
	\end{prop}
	For the proof, we need the following technical input on the Hodge--Tate morphism:
	\begin{prop}\label{p:properties-of-can-subgroup}
		Let $K$ be a completely valued extension of $\Q_p$ with algebraic closure $\overline{K}$. For any $v\in |\RR|$, let $(p^{v}):=\{x\in K\mid |x|\leq v\}$. Let $D$ be a $p$-divisible group over $\O_K$ of dimension $d$ and height $h$. Let $0\leq \epsilon$ be such that the Hodge ideal is $\Hdg(D)=(p^{\epsilon})$. Let $n\in \Z_{\geq 0}$ be such that $\epsilon\leq 1/2p^{n-1}$ if $p\geq 5$, or $\epsilon\leq 1/3p^{n-1}$ if $p=3$, or $\epsilon\leq 1/4p^{n-1}$ if $p=2$. Let $\delta:=\epsilon \frac{p^n-1}{p-1} <1$.
		\begin{enumerate}
			\item $D$ has canonical subgroups $1\subseteq H_1\subseteq \dots \subseteq H_m \dots \subseteq H_n\subseteq D$ of level $m$, finite locally free of rank $p^{md}$, for all $m\leq n$. They reduce to the kernel of Frobenius on $D \bmod (p^{1-\delta})$.
			\item The map $\omega_{D[p^n]}\to \omega_{H_n}$ induces an isomorphism  $\omega_{D[p^n]}/(p^{n-\delta})=\omega_{H_n}/(p^{n-\delta})$.
			\item The Hodge--Tate map $H_n(\overline{K})^{\vee}\otimes_{\ZZ}\O_{\overline{K}}\to \omega_{H_n}\otimes_{\O_K}\O_{\overline{K}}$ has cokernel of degree $\epsilon/(p-1)$.
		\end{enumerate}
	\end{prop}
	\begin{proof}
		(1) is a special case of \cite[Cor.~A.2, parts 1,2]{AIP2}. For (2), in the case of $p>2$, this is  \cite[ Prop.~3.2.2]{AIP4}. For the case of $p=2$, we can in the proof replace \cite[Thm.~3.1.1]{AIP4} by \cite[ Cor. A.2.4]{AIP2}. 
		Finally, for $p>2$, (3) is again \cite[Prop.~3.2.2]{AIP4}. The case of (3) for $p=2$ follows from \cite[Prop.~A.3]{AIP2}, which applies by $2\in \Hdg(D)^{4}$ and which says that $\det \coker=\det\omega_{H_n}/\mathrm {HdgT}$ for an ideal $\mathrm{HdgT}\subseteq \O_K$ satisfying $\mathrm{HdgT}^{p-1}=\Hdg(D)=(p^{\epsilon})$.  
	\end{proof}

	\begin{proof}[Proof of Prop.~\ref{prop: epsilon w for HMFS}]
		
		It suffices to check this on $(C,C^+)$-points for $C$ algebraically closed.
		
		Let $z\in \XX_{\Gamma(p^\infty)}(\epsilon)(C,C^+)$ correspond to a $\mathfrak c$-polarised HBAV $A/C$ with extra data and an isomorphism $\alpha:\O_p^2\to T_pA^{\vee}$. 
		Let $A_0^{\vee}$ be the semi-abelian scheme  over $\O_C$ associated to $A^{\vee}$ and let
		$V$ be the kernel of the integral Hodge--Tate-map; then there is a left exact sequence
		\[0\to V\to T_pA_0^{\vee}\otimes \O_C\xrightarrow{\HT} \omega_{A_0}\]
		and by definition, $V\subseteq T_pA_0^{\vee}\otimes \O_C$ is saturated. Via $\alpha$, it thus gives a point $(a:b) \in  \PP^1(\OO_p \otimes_{\Zp}\OO_C) \cong \PP^1(\OO_C)^\Sigma$ with $a = (a_v)_v, b = (b_v)_v \in \OO_C^\Sigma$, which is the image of $z$ under $\pi_{\HT}$.
		
		Let $n:=m+1$. Upon reduction mod $p^n$, we get an injection  $V/p^n\to A_0^{\vee}[p^n]\otimes_{\Z} \O_C$ which fits into a (not necessarily exact) complex
		\[0\to V/p^n\to A_0^{\vee}[p^n]\otimes \O_C\to \omega_{A_0[p^n]}.\]
		
		The Hodge ideals of the $p$-divisible groups of $A$ and $A^\vee$ are the same (e.g. \cite[Thm.~3.1.1]{AIP4}); thus by Prop.~\ref{p:properties-of-can-subgroup} (1) and our choice of $\e$, there is a canonical subgroup $H_n\subseteq A_0^{\vee}[p^n]$ of rank $p^n$. Modulo a certain power of $p$, the position of $V/p^n$ coincides with that of $H_n$ inside $A_0^\vee[p^n]$:
		
		\begin{claim}
			Let $x=n-\frac{p^{n}}{p-1}\epsilon$. Then inside $A^{\vee}_0[p^n]\otimes \O_C/p^x$, we have 
			$V/p^x=H_n\otimes \O_C/p^x.$
		\end{claim} 
		
		To see that this proves the proposition, note that $H_n\otimes_{\Z}\O_C\subseteq A^{\vee}_0[p^n]\otimes_{\Z}\O_C$ has $\ZZ_p/p^n\ZZ_p$-coordinates. Moreover, the case~(1) that $z\in \XX_{\Gamma(p^\infty)}(\epsilon)_c(C,C^+)$ is equivalent to the coordinates of $H_n$ being of the form $(1:0) \in \PP^1(\ZZ/p\ZZ_p)^\Sigma$ after reducing modulo $p$. The claim then implies that $b/a \in  p\OO_p + p^x\OO_C^\Sigma$, and hence $\pi_{\HT}(z)=(a:b) \in B_{|p^x|}(1:p\OO_p).$
		Since $x>n-1$, we have $
		|p^x|=1/p^x\leq 1/p^{n-1}\leq r$.
		This implies $\pi_{\HT}(z)\in B_r(1:p\OO_p)$, as desired.
		
		The proof of~(2) follows in the same way, using that  $z\in \XX_{\Gamma(p^\infty)}(\epsilon)_a(C,C^+)$ is equivalent to the coordinates of $H_n$ being of the form $(c:1) \in \PP^1(\ZZ/p\ZZ_p)^\Sigma$ for some $c\in \ZZ/p\ZZ_p$, and therefore
		\begin{equation}\label{eq:sharp-estimate-on-image-of-pi_HT-in-proof-of-image-of-HT-Hilbert-case}
		\pi_{\HT}(z)=(a:b) \in B_{|p^x|}(\OO_p:1).
		\end{equation}
		\emph{Proof of claim}: Let $y:= n - \delta = n-\frac{p^{n}-1}{p-1}\epsilon$. By Prop.~\ref{p:properties-of-can-subgroup} (2),
		modulo $p^y$ the Hodge--Tate map can be described as
		\[ 
		\HT_y:A_0^{\vee}[p^n]\otimes \O_C/p^{y}\to\w_{A_0^{\vee}[p^n]}\otimes \O_C/p^y= \omega_{H_n}\otimes \O_C/p^{y}.
		\]
		Let now $N:=\ker \HT_y$ and $Q:=\coker \HT_y$ and
		consider the exact sequence
		\[0\to N\to A_0^\vee[p^n]\otimes_{\Z_p} \O_C/p^y\xrightarrow{\HT_y} \omega_{H_n}/p^{y}\to Q\to 0  \]
		By Prop.~\ref{p:properties-of-can-subgroup} (3), the $\O_C$-module $Q=\coker \HT_y$ has degree $\partial:=\epsilon/(p-1)$. Using additivity of degrees of $\O_C$-modules in extensions, we calculate that
		\[
		\deg  N=\deg(A_0[p^n]\otimes_{\Z_p} \O_C/p^y)-\deg \omega_{H_n}/p^{y} +\deg Q = 2gy-gy+\partial=gy+\partial.
		\]
		
		Observe now that $M_1:=V/p^y$ and $M_2:=H_n\otimes_{\Z_p} \O_C/p^y$ are both free $\O_C/p^y$-submodules of rank $g$ of $A_0[p^n]\otimes_{\Z_p} \O_C/p^y$ that are contained in $N$. Since  $N$ is $p^y$-torsion and of degree $gy+\partial$, we conclude from this that $N$ is of the form $(\O_C/p^y)^g\oplus T$ as an $\O_C$-module, where $T$ is $p^{\partial}$-torsion. Second, this shows that inside $p^{\partial}N$, the modules $p^{\partial}M_1$ and $p^{\partial}M_2$ coincide. Thus the same is true inside $A_0^{\vee}[p^n]\otimes p^{\partial}\O_C/p^y$. Via multiplication by $p^{\partial} :\O_C/p^{y-\partial}\isorightarrow p^{\partial}\O_C/p^{y}$, this shows that the images of $M_1$ and $M_2$ in $\O_C/p^{y-\partial}$ coincide. Since by definition $x=y-\partial$, this gives the desired statement, proving the claim, and hence the proposition.
	\end{proof}
	
	\begin{defn}
		We write $\mathfrak z$ for the restriction of $\pi_{\HT}$ to the open subspaces 
		\[\mathfrak z:\mathcal X^{}_{\Gamma^*(p^\infty)}(\epsilon)_a\to B_r(\O_p:1)\subseteq\Res_{\OF|\ZZ} \hat{\G}_a\subseteq\Res_{\OF|\ZZ} \P^1.\]
	\end{defn}
	\begin{rmrk}
		If $F$ is split in $L$, we also consider for any $v:F\to L$ the projection
		\[\mathfrak z_v:\mathcal X^{}_{\Gamma^*(p^\infty)}(\epsilon)_a\to  \Res_{\OF|\ZZ} \hat{\G}_a=\hat{\G}_a^{\Sigma}\xrightarrow{\pi_v}\hat{\G}_a.\]
		By the universal property of $\hat{\G}_a$, we can interpret each $\mathfrak z_v$ as a function in $\O^+(\mathcal X^{}_{\Gamma^*(p^\infty)}(\epsilon)_a)$. However, we caution that for general $L$, the morphism $\mathfrak z$ admits no such canonical interpretation.	
	\end{rmrk}
	\subsection{The canonical differential}
	\begin{defn}
		We define a $G^*$-equivariant vector bundle $\Res_{\OF|\ZZ}\O(1)$ of rank $g$  on $\Res_{\OF|\ZZ}\PP^1$ as follows: recall that on $\PP^1$ we have the line bundle $\O(1)$ whose total space $\pi:\T(1)\to \PP^1$ is therefore a $\GG_m$-bundle with fibres $\A^1$.  It moreover has a natural $\GL_2$-equivariant action. By applying the functor $\Res_{\OO_F|\ZZ}$, we see that $\Res_{\OO_F|\ZZ}\pi:\Res_{\OO_F|\ZZ}\T(1)\to \Res_{\OO_F|\ZZ}\PP^1$ is a $\Res_{\OO_F|\ZZ}\G_m$-bundle with fibres $\Res_{\OO_F|\ZZ}\A^1$. As any choice of $\Z$-basis of $\O_F$ induces an isomorphism $\Res_{\OO_F|\ZZ}\A^1\cong\A^g$, we conclude that $\Res_{\OO_F|\ZZ}\pi$ is a vector bundle of rank $g$. It moreover receives a natural equivariant $\Res_{\O_F|\ZZ}\GL_2=G$-action (and hence a $G^*$-action) by functoriality.
	\end{defn}
	\begin{rmrk}
		The vector bundle $\Res_{\OO_F|\ZZ}\T(1)$ has the following moduli interpretation: for any $\Z_p$-algebra $R$, the $R$-points of $\Res_{\OF|\ZZ}\PP^1$ parametrise quotients $R^2\otimes_{\Z}\OF \to Q$ of rank 1 as $R\otimes_{\Z}\OF$-modules. Then $\Res_{\OO_F|\ZZ}\T(1)\to \Res_{\OF|\ZZ}\PP^1$ represents the choice of a point of $Q$.
	\end{rmrk}

	\begin{defn}\label{defn: omega notation}
		Let $\omega_{\mathcal A}$ be the conormal sheaf of the universal abelian variety $\mathcal A\to\mathcal X$, an invertible $\O_{\mathcal X}\otimes_{\Z}\O_F$-module. Its total space $\T(\omega_{\mathcal A})\to \mathcal X$ is a $\Res_{\OO_F|\ZZ}\G_m$-bundle. 
		As before, if $q:\XX_{K_p} \to \XX$ is the forgetful map with $K_p$ any of our wild levels, we let $\omega_{K_p}:=q^*\omega_{\cA}$.
	\end{defn}
	
	As a special case of \cite[Thm.~2.1.3. (2)]{CarScho}, we then have the following result which forms the basis of our definition of Hilbert modular forms.
	\begin{prop}\label{prop: CS sheaf under pi_HT}
		There is a $\Res_{\OO_F|\ZZ}\G_m$-equivariant isomorphism
		\[\omega_{\Gamma^*(p^\infty)}=\pi_{\HT}^{\ast}\Res_{\O_F|\Z}\OO(1).\]
	\end{prop}
	Recall that in \S4 we have defined a canonical section $s:\P^1\to\T(1)$ of $\O(1)$, non-vanishing away from $\infty$. We shall now change notation and denote this by $s_{\mathrm{ell}}:\A^1\to \T(1)$. We now set:
	\begin{defn}
		Let $s:=\Res_{\OO_F|\ZZ}s_{\mathrm{ell}}:\Res_{\O_F|\ZZ}\P^1\to \Res_{\OO_F|\ZZ}\T(1)$. This is a section of the vector bundle $\Res_{\OO_F|\ZZ}\O(1)$, non-vanishing over $\Res_{\OO_F|\ZZ}\A^1\subseteq \Res_{\OO_F|\ZZ}\P^1$.
	\end{defn}
	\begin{rmrk}\label{r:moduli-interpreation-of-s-Hilbert-case}
		From the moduli description in the case of $g=1$, we see that in the moduli interpretation, $s$ sends a quotient $R^2\otimes_{\Z}\OF \to Q$ to the image of $(1,0)\otimes 1$.
	\end{rmrk}
	
	\begin{rmrk}\label{r:section-s-split-case}
		If $F$ is split in $L$, we have $\OF \otimes_{\ZZ} L=\prod_{\v \in \Sigma}L$  where we interpret $\Sigma$ as the set $\Hom_{\ZZ}(\OF,L)$ and
		where the morphism into the $v$-component comes from the natural map $\OF\otimes_{\ZZ}L\xrightarrow{v\otimes \mathrm{id}} L$.
		Consequently, we then get a canonical splitting $\Res_{\OF|\ZZ}\P^1 =(\P^1)^{\Sigma}$.
		Similarly, we see on total spaces that the vector bundle $\Res_{\OF|\ZZ}\O(1)$ becomes the direct sum $\Res_{\OF|\ZZ}\O(1)=\bigoplus_{v\in\Sigma} \pi_v^{\ast}\O(1)$
		of the pullbacks of $\O(1)$ on $\PP^1$ along the projections $\pi_{v}:(\PP^1)^{\Sigma}\to \PP^1$. 
		The section $s$ then decomposes into partial sections $\mathfrak s=\sum_{v:F\hookrightarrow L} s_v$ where $s_v:=\pi^{\ast}_vs_{\mathrm{ell}}$. 
	\end{rmrk}

	\begin{lem}\label{l:transform-of-s-Hilbert}
		For any $\gamma=\smallmatrd{a}{b}{c}{d}\in\Gamma^{\ast}_0(p)$, let $(cz+d)$ be the map $\Res_{\OF|\ZZ}\hat{\G}_a\xrightarrow{\cdot c}p\Res_{\OF|\ZZ}\hat{\G}_a\xrightarrow{+d}\Res_{\OF|\ZZ}\hat{\G}_m$. Then we have $\gamma^{\ast}s = (cz+d)s$, in the sense that the following diagram commutes:
		\begin{center}
			\begin{tikzcd}[row sep = 0.55cm]
			\Res_{\OF|\ZZ}\hat{\G}_m\times 	\Res_{\OF|\ZZ}\T(1) \arrow[r, "\mathrm m"] & 	\Res_{\OF|\ZZ}\T(1) & 	\Res_{\OF|\ZZ}\T(1) \arrow[l, "\gamma"'] \\
			& 	\Res_{\OF|\ZZ}\hat{\G}_a \arrow[r, "\gamma"] \arrow[lu, "(cz+d)\times s"] \arrow[u, "\gamma^{\ast}s"'] & \Res_{\OF|\ZZ}\hat{\G}_a. \arrow[u, "s"']
			\end{tikzcd}
		\end{center}
	\end{lem}
	\begin{proof}
		It suffices to show that this diagram commutes after extending $L$, so we may without loss of generality assume that $F$ is split in $L$. Then by Rem.~\ref{r:section-s-split-case},  $\Res_{\OO_F|\ZZ}\P^1=(\P^1)^{\Sigma}$ is canonically split, as is the bundle $\Res_{\OO_F|\ZZ}\T(1)=\oplus_{\Sigma} \O(1)$, and the diagram becomes a product over $\Sigma$ of the diagram in Lem.~\ref{l: action of Gamma_0(p) on the section}.
	\end{proof}

	\begin{definition}\label{d:mathfrak-s-Hilbert-case}
		Let $\mathfrak s:=\pi_{\HT}^*s$. This is a section of $\pi_{\HT}^{\ast}\Res_{\OF|\ZZ}\O(1)=\omega_{\Gamma^*(p^\infty)}$. Write $\mathcal T(\omega_{\Gamma^*(p^\infty)})\to \XX$ for the total space of $\omega_{\Gamma^*(p^\infty)}$; then
		we may regard $\mathfrak s$ as a morphism
		\[\mathfrak s:\mathcal X_{\Gamma^*(p^\infty)}\to \mathcal T(\omega_{\Gamma^*(p^\infty)}). \]
	\end{definition}
	
	As in the elliptic case, one checks that: $\mathscr{s}$
	\begin{lem}\label{l:transform-for-mathfrak-s-Hilbert-case}
		For any $\gamma=\smallmatrd{a}{b}{c}{d}\in \Gamma^{\ast}_0(p)$, we write $c\mathfrak z+d$ for the composition 
		\[c\mathfrak z+d:\XX_{\Gamma^\ast(p^\infty)}(\epsilon)_a\xrightarrow{\mathfrak z} \Res_{\OO_F|\ZZ}\hat{\G}_a\xrightarrow{z\mapsto cz+d} \Res_{\OO_F|\ZZ}\hat{\G}_m.\]
		Then we have $\gamma^{\ast}\mathfrak s = (c\mathfrak z+d)\mathfrak s$,
		in the sense that the following diagram commutes:
		\begin{center}
			\begin{tikzcd}[row sep = 0.55cm]
			\mathcal X_{\Gamma^{\ast}(p^\infty)}(\epsilon)_a \arrow[r, "\gamma"] \arrow[d, "(c\mathfrak z+d)\times\mathfrak s"'] & \mathcal X_{\Gamma^{\ast}(p^\infty)}(\epsilon)_a \arrow[d, "\mathfrak s"'] \\
			\Res_{\OO_F|\ZZ}\hat{\G}_m\times\T(\omega_{\mathcal A}) \arrow[r, "\mathrm m"] & \T(\omega_{\mathcal A})
			\end{tikzcd}
		\end{center}
	\end{lem}
	
	\begin{proof}
		This follows  from Lem.~\ref{l:transform-of-s-Hilbert} by pullback along $\pi_{\HT}$.
	\end{proof}
	
	The crucial property of $\mathfrak s$ is given by  the following moduli interpretation.
	\begin{lem}\label{l:moduli interpretation of s in terms of HT in Hilbert}
		Let $x\in \mathcal X_{\Gamma^*(p^\infty)}(C,C^+)$ be a point corresponding to a HBAV $A$ equipped with a $\Gamma^{\ast}(p^\infty)$-level $\alpha:\OO_p^2 \isorightarrow T_pA^{\vee}$ and extra structures. Then via $\pi_{\HT}^{\ast}\Res_{\OO_F|\ZZ}\mathcal O(1)=\omega_{\Gamma^*(p^\infty)}$,
		\[ \mathfrak s(x)=\HT_{A}(\alpha(1,0))\in \omega_{A}.\]
		
	\end{lem}
	\begin{proof}
		In terms of the total spaces $\mathcal T(\omega_{\Gamma^*(p^\infty)})\rightarrow \mathcal X_{\Gamma^*(p^\infty)}$ and $\Res_{\OO_F|\ZZ}\mathcal T(1)\rightarrow \Res_{\OF|\ZZ}\PP^1$, by Rem.~\ref{r:mod-interpret-of-pi_HT-Hilbert} the isomorphism $\omega_{\Gamma^*(p^\infty)}=\pi_{\HT}^{\ast}\Res_{\OF|\ZZ}\mathcal O(1)$ is defined in the fibre of $x$ by sending
		\begin{alignat*}{2}
		(A,\alpha,\eta\in \omega_{A})\mapsto (\O_p^2\otimes_{\Z_p} C\xrightarrow{\alpha}T_pA^\vee\otimes_{\Z_p} C\xrightarrow{\HT} \omega_{A},\eta\in \omega_A ).
		\end{alignat*}
		Since $s$ by Rem.~\ref{r:moduli-interpreation-of-s-Hilbert-case} sends a quotient $x:\O_p^2\otimes_{\Z_p} C\to Q$ to the image of $(1,0)\otimes 1$ under $x$, it follows that $\mathfrak s$ sends $x$ to the image of $(1,0)$ under $\HT\circ\alpha$.
	\end{proof}


	%
	%
	
	\section{Geometric overconvergent Hilbert modular forms}
	\label{sec:G* forms}
	
	\subsection{Weights and analytic continuation}

	Next we define the relevant weight spaces for $G^*$ and $G$, and set up some notational conventions as to how they are related.	
	\begin{defn}\label{def:Hilbert weight space}
		
		Let $\TT:=\Res_{\OF|\ZZ}\GG_m$, then define:
		\begin{itemize}
			\item[(i)]$\W := \mathrm{Spf}(\ZZ_p\llbracket \TT(\ZZ_p) \times \ZZ_p^\times \rrbracket)^{\mathrm{an}}_\eta \times L$,  the \emph{weight space for $G$}. 
			\item[(ii)]  $\W^* := \mathrm{Spf}(\ZZ_p \llbracket \TT(\ZZ_p) \rrbracket)^{\mathrm{an}}_\eta \times L$, the \emph{weight space for $G^*$}.
			
		\end{itemize}
		An $L$-point $(w,t) \in \W(L)$ is a pair of maps $w:\TT(\ZZ_p) \to L^\times$ and $t: \ZZ_p^\times \to L^\times$ (and analogously, an $L$-point of $\W^*$ is just a map $\TT(\ZZ_p) \to L^\times$). Following \cite{AIP}, we let $\rho: \W \to \W^*$ be the morphism associated to the map $\TT(\ZZ_p) \to \TT(\ZZ_p) \times \ZZ_p^\times$ defined by $x \mapsto (x^{2}, \Norm(x))$. For $(w,t) \in \W(\CC_p)$ we write $\kappa = w^{2} \cdot (t^{-1} \circ \Norm)$ for its image in  $\W^*(\CC_p)$, noting that $\k(x)\cdot w(x^{-2})$ factors through some power of the norm.

	\end{defn} 
	
	\begin{defn}\label{nota: weight conv}
		In order to be able to treat single weights and families in a uniform way, we define a weight to be a morphism $\k:\U \to \W$ or $\k:\U \to \W^*$  for $G$ and $G^*$ respectively, where $\U$ is a smooth rigid space over some perfectoid field extension of $L$. We say that $\kappa$ is bounded if its image in $\W$ or $\W^{\ast}$ is contained in some affinoid open subspace. This generalises Defn.~\ref{def:smooth weight}.
		
		By unravelling the definitions, a weight $\kappa : \U \to \W^*$ determines a morphism
		\[
		\kappa:\OO_p^\times\times \U \to \hat{\G}_m,
		\]
		which in an abuse of notation we also denote $\kappa$.
		The weight $\kappa$ is then bounded if and only if \[|T_{\kappa}|:=\sup_{(t,x)\in \O_p^\times \times \U} |\kappa(t,x)-1|<1\]
		
		Similarly, for $G$ we have associated to any $\kappa : \U \to \W$ a pair of maps $(\wU,\rU)$ of the form $\wU :	\OO_p^\times\times \U \to \hat{\G}_m$ and $\rU : 	\ZZ_p^\times\times \U \to \hat{\G}_m$. By composing with $\rho$, we get an associated weight $\rho(\wU,\rU) = \wU^{2} \cdot (\rU^{-1}\circ N_{F/\Q})$ for $G^*$, which we use to see any weight for $G$ as a weight for $G^*$.
	\end{defn}
	
	Recall from Definition~\ref{d:embed-profinite-O_p-into-ResP^1} that we embed $\O_p$ as a profinite set into $\Res_{\O_F|\ZZ} \P^1$ by sending $z\mapsto (z:1)$. Given a bounded weight, one can then always find an analytic continuation of $\kappa$ to a neighbourhood of $\OO_p^\times$ in $\Res_{\OF|\ZZ}\P^1$. More precisely:
	\begin{prop}\label{an cnt hmfs}
		Let $\k:\U \to\W^*$ be a bounded smooth weight.  
		Let $r_0=1$ if $p >2$ and $r_0=3$ if $p=2$.
		Let $r_\k:=|p|^{r_0}|T_\kappa|$, then for $r_\k \geq r > 0$, the morphism $\kU$ extends uniquely to a morphism
		\[\kU^{\an}: B_r(\OO_p^{\times}:1)\times \U\to \hat{\G}_m.\]
	\end{prop}
	\begin{proof}
		We first prove that such a bound exists. In case that $F$ is split in $L$, this is completely analogous to Prop.~\ref{prop: an cont fam}. In general, we first pass to a finite Galois extension $L'|L$, with group $H$ and in which $F$ is split, to obtain a morphism $B_r(\OO_p^{\times}:1)\times_{L}L'\times \U\to \hat{\G}_m$. Passing to the quotient by $H$, the result follows. 
		
		The precise value of $r_\kappa$ follows from \cite[Prop.~2.8]{AIP3}. 
	\end{proof}
	
	\begin{definition}\label{d:cocycle-kappa(cz+d)}
		Let $\kappa:\mathcal U\to \mathcal W^{\ast}$ be a smooth bounded weight.
		Let $\ed>0$  be such that $\XX_{\Gamma(p^\infty)}(\ed)_a\subseteq B_{r_\k}(\O_p:1)$, see Prop.~\ref{prop: epsilon w for HMFS} for a precise bound on $\ed$. Then for any $c\in p\O_p$, $d\in \O_p^\times$, we define the invertible function $\kappa(c\mathfrak z+d)\in \O^+(\mathcal X^{}_{\U,\Gamma^*(p^\infty)}(\epsilon)_a)^\times$ to be the composition
		\[\kU(c\mathfrak z+d):\XX_{\U,\GA^*(p^\infty)}(\e)_a \xrightarrow{\pi_{\HT}\times \id} B_r(\OO_p:1) \times \U\xrightarrow{(cz+d)\times \id} B_r(\OO^\times_p:1)\times \U \xrightarrow{\kU^{\an}} \hat{\G}_m,\] where $\XX_{\U,\GA^*(p^\infty)}(\e)_a:= \XX_{\GA^*(p^\infty)}(\e)_a \times_L \U$.
	\end{definition}
	
	
	\subsection{Definition of overconvergent Hilbert modular forms}

	
	\begin{defn}\label{def:bundle for G*}

		For $\k: \U \to \W^*$ a bounded smooth weight, $0 \leq \e \leq \ed$ and $n \in \ZZ_{\geq 1} \cup \{\infty\}$, we define a sheaf $\omega^{\k}_{n}$ on $\mathcal X_{\U,\Gamma_{\!\scaleto{0}{3pt}}^*(p^n)}(\e)_a$ by setting
		\[\omega^{\k}_{n}(U):=\{f \in q_{*} \mathcal O_{\mathcal X_{\U,\Gamma^*(p^\infty)}(\e)_a}(U)| \gamma^{\ast}f = \kU^{-1}(c\mathfrak z+d)f  \text{ for all }\gamma=\smallmatrd{a}{b}{c}{d}\in \Gamma_{\!0}^*(p^n) \}, \] 
		where  $q:\mathcal X_{\U,\Gamma^*(p^\infty)}(\e)_a\rightarrow \mathcal X_{\U,\Gamma_{\!\scaleto{0}{3pt}}^*(p^n)}(\e)_a$ is the projection. We similarly get the integral subsheaf
		\[\omega^{\k,+}_{n}(U):=\{f \in q_{*} \mathcal O^+_{\mathcal X_{\U,\Gamma^*(p^\infty)}(\e)_a}(U)| \gamma^{\ast}f = \kU^{-1}(c\mathfrak z+d)f  \text{ for all }\gamma=\smallmatrd{a}{b}{c}{d}\in \Gamma_{\!0}^*(p^n) \}, \] 
		by using the $\mathcal O^+$-sheaf instead. For $n=0$, as before, via the Atkin--Lehner isomorphism $\AL: \mathcal X_{\U,\Gamma_{\!\scaleto{0}{3pt}}^*(p)}(p\e)_a\isorightarrow \mathcal X_{\U}(\e)$ we define the sheaves $\w^{\k}:=\omega_{0}^{\k} :=\AL_{\ast}\omega_{ 1}^{\k}$ and $\w^{\k,+}:=\omega_{0}^{\k} :=\AL_{\ast}\omega_{ 1}^{\k,+}$ on $\mathcal X_{\U}(\e)$ thus giving a sheaf on the tame level Hilbert modular variety. If needed we will add a subscript $G^*$ to make clear these are sheaves for $G^*$.

	\end{defn}	
	
\textcolor{black}{
Exactly like in Prop.~\ref{p:w-is-analytic}, we see:}
\begin{proposition}
    $\omega_{n}^{\k}$ is an analytic line bundle on $\mathcal X_{\mathcal U}(\epsilon)$.
\end{proposition}
We will also see this in Thm.~\ref{thm:comparison hilbert}, which moreover shows that $\omega_{n}^{\k,+}$  is an invertible $\O^+$-modules.
\begin{proof}
    \textcolor{black}{Exactly as in the elliptic case,  \cite[Cor.~4.1]{heuer-v_lb_rigid} shows that the analyticity overconverges if we can prove it for $\epsilon=0$. By the same argument, we may restrict to the good reduction locus, as this is Zariski-dense in $\mathcal X_{\U}$. Over this, we again have an Igusa tower with a pro-\'etale formal model, and like in the elliptic case, \cite[Prop.~4.8]{heuer-v_lb_rigid} gives the desired statement.   }
\end{proof}
	
	\begin{warn}
		We caution the reader that $\omega_{n}^1$ is not the same as $\omega_{\Gamma_0^*(p^n)}$ from Defn.~$\ref{defn: omega notation}$, as the latter is not an invertible sheaf, when $[F:\QQ] >1$. Instead, we have $\omega_{n}^1= \det \omega_{\Gamma^*(p^n)}$.
	\end{warn}
	
	\begin{defn}
		Let $\k:\U \to \W^*$ be a bounded smooth weight,  $0 \leq \e \leq \ed$ and $n \in \ZZ_{\geq 0} \cup \{\infty\}$.
		We define the space of $\gothc$-polarised overconvergent Hilbert modular forms for $G^*$ of weight $\k$, wild level $\Gamma_{\!0}^*(p^n)$, tame level $\mu_N$ and radius of overconvergence $\e$ to be the $L$-vector space
		\[
		M^{G^*}_{\k}(\Gamma_{\!0}^*(p^n),\mu_N,\epsilon,\gothc):=\hH^0(\XX_{\gothc,\U,\Gamma_{\!\scaleto{0}{3pt}}^*(p^n),\mu_N}(\e)_a,\  \omega_{G^*,n}^{\k}).
		\]
		Similarly, we define the space of integral overconvergent Hilbert modular forms for $G^*$ to be \[M^{G^*,+}_{\k}(\Gamma_{\!0}^*(p^n),\mu_N,\epsilon,\gothc):=\hH^0(\XX_{\gothc,\U,\Gamma_{\!\scaleto{0}{3pt}}^*(p^n),\mu_N}(\e)_a, \  \omega_{G^*,n}^{\k,+}).\]

	\end{defn}
	
	
	\begin{remark}\label{rem:cusp forms}
		\textcolor{black}{By the Koecher principle (see \cite[Prop.~8.4]{AIP3}) or \cite[Theorem 5.5.1]{AIP}, $\w^{\k}_{\gothc}$ extends uniquely to a line bundle on a suitable toroidal compactification $\XX^{\tor}(\gothc)$. Let  $\partial$ denote the boundary divisor, then one can define the subspaces of cusp forms as sections of the subsheaf $\w^{\k}_{\gothc}(-\partial)$. Via Thm.~\ref{thm:comparison hilbert} below, these agrees with the spaces of cusp forms defined in \cite{AIP}. In particular, they will be projective Banach modules with surjective specialisation maps (see \cite[Thm.~3.16]{AIP3}).} 
	\end{remark}



	\section{Comparison to Andreatta--Iovita--Pilloni's geometric Hilbert modular forms}\label{section: comparison HMFs}
	
	In this section, we will show that our spaces of overconvergent Hilbert modular forms for $G^*$ coincide with those defined in \cite{AIP3}.
	
	\subsection{The Andreatta--Iovita--Pilloni-torsor}
	Like in the elliptic case, Andreatta--Iovita--Pilloni construct integral sheaves of Hilbert modular forms on the Hilbert modular variety as a formal scheme over $\O_F$.
	\textcolor{black}{In order to define such a sheaf on the full Hilbert modular variety over $\O_F$, the definition of the Pilloni-torsor in the Hilbert case is not just the straightforward adaptation of the elliptic case (the issue appears away from the Rapoport locus, i.e.\ on the closed subscheme concentrated in the special fibre where the abelian scheme does not satisfy the Rapoport condition). Instead, Andreatta--Iovita--Pilloni in \cite[\S4.1]{AIP3} explain how this definition needs to be modified by endowing the sheaf $\omega_{\mathcal A}$ with an integral structure $\omega^{\mathrm{int}}$ (denoted by $\mathcal F$ \textit{op.\ cit.}) which, when $p$ is ramified in $F$, is different to the canonical one. We briefly recall the construction, with the minor modification that as before we present it in the analytic setting over $L$ rather than in the excellent Noetherian setting of \cite[\S4.1]{AIP3}.}
	
	\begin{definition}
		\begin{enumerate}
			\item For any $m\in \Z_{\geq 1}$, let $\epsilon^{\mathrm{can}}_{m}:=1/p^{m+1}$ as before. Then  \cite[Cor.~A.2]{AIP2} implies that, for $0\leq \epsilon\leq \epsilon^{\mathrm{can}}_{m}$, the universal semi-abelian variety $\mathcal A$ on $\mathcal X(\e)$  admits a canonical subgroup $H_m\subseteq \mathcal{A}$ of order $p^{m}$, \'etale locally isomorphic to $\O_F/p^m\O_F$.
			\item
			We denote by $\mathcal X_{\Ig(p^m)}(\epsilon)\to \mathcal X(\epsilon)$ the finite \'etale $(\O_F/p^m\O_F)^\times$-torsor which relatively represents isomorphisms $\O_F/p^m\O_F\to H_m^\vee$ of adic spaces with $\O_F$-module structure 
		\end{enumerate} 
	\end{definition}

	Let $\omega_{\Ig(p^m)}$ be the conormal sheaf of the pullback of $\mathcal A$ to $\mathcal X_{\Ig(p^m)}(\epsilon)$. It has an integral subsheaf $\omega^+_{\Ig(p^m)}	$ obtained from its formal model on $\mathfrak X^{\ast}(\epsilon)$. The canonical subgroup $H_m\subseteq \mathcal A$, considered as a finite flat group over $\mathcal X_{\Ig(p^m)}(\epsilon)$ induces a map $\pi:\omega_{\Ig(p^m)}^+\to \omega_{H_m}^{+}$. As in Lem.~\ref{l:ses-of-integ-conormal-sheaves}, we see:
	\begin{lem}[{\cite[Cor.~A.4]{AIP2}}]\label{l:Cor-A.4-AIP2}
		We have a right exact sequence of $\O^+_{\mathcal X_{\Ig(p^m)}(\epsilon)}$-modules
		\[ I_m\cdot \omega_{\Ig(p^m)}^+\to \omega_{\Ig(p^m)}^+\xrightarrow{\pi}\omega_{H_m}^{+}\to 0,\quad \text{where}\quad I_m:=p^m\Hdg^{-\frac{p^m-1}{p-1}}.\]
	\end{lem}
	
	The Hodge--Tate map now defines a morphism of sheaves of $\O_F$-modules over $\mathcal X_{\Ig(p^m)}(\epsilon)$
	\[\psi:\O_F/p^m\O_F\to H_m^\vee\xrightarrow{\HT}\omega^{+}_{H_m}\rightarrow \omega_{\Ig(p^m)}^+/I_m .\]
	\begin{definition}
		Let $\omega^{\mathrm{int}}_{\Ig(p^m)}$ be the $\O_F\otimes_{\Z}\O_{\mathcal X_{\Ig(p^m)}(\epsilon)}^+$-submodule of $\omega^+_{\Ig(p^m)}$ defined as the preimage of the $\O_F$-submodule of $\omega_{\Ig(p^m)}^+/I_m$ generated by $\psi(1)$. 
	\end{definition}
	The sheaf $\omega^{\mathrm{int}}_{\Ig(p^m)}$ gives a second integral structure on $\omega_{\Ig(p^m)}$. If $p$ is ramified in $\O_F$, it is better behaved than $\omega_{\Ig(p^m)}^+$, because it always satisfies the analogue of the Rapoport condition:
	\begin{prop}[{\cite[Prop.~4.1]{AIP3}}]\label{p:prop-of-w^int}\leavevmode
		\begin{enumerate}
			\item The sheaf $\omega^{\mathrm{int}}_{\Ig(p^m)}$ is a locally free  $\O_F\otimes_{\Z}\O_{\mathcal X_{\Ig(p^m)}(\epsilon)}^+$-module on $\mathcal X_{\Ig(p^m)}(\epsilon)$.
			\item The cokernel of $\omega^{\mathrm{int}}_{\Ig(p^m)}\subseteq \omega_{\Ig(p^m)}^+$ is annihilated by $\Hdg^{1/(p-1)}$. We thus have an injection
			\[ \omega_{\Ig(p^m)}^{\mathrm{int}}/I_m\hookrightarrow \omega^{+}_{\Ig(p^m)}/I_m \]
			whose image is precisely the $\O_F\otimes_{\Z}\O_{\mathcal X_{\Ig(p^m)}(\epsilon)}^+$-submodule generated by $\psi(1)$.
			\item Let $I_m':=p^m\Hdg^{-\frac{p^m}{p-1}}\supseteq I_m$. Sending $1\mapsto \psi(1)$ induces an isomorphism
			\[\HT':\O_F\otimes_{\Z}\O^+_{\mathcal X_{\Ig(p^m)}(\epsilon)}/I'_m \isorightarrow  \omega^{\mathrm{int}}_{\Ig(p^m)}/I_m'.\]
		\end{enumerate}
	\end{prop}
	\begin{definition}\label{d:AIP-torsor}
		The Andreatta--Iovita--Pilloni-torsor is the subsheaf of $\omega^{\mathrm{int}}_{\Ig(p^m)}$ defined by 
		\[
		\mathfrak F_m:=\{w \in \omega^{\mathrm{int}}_{\Ig(p^m)}| w\equiv \HT'(1) \bmod I'_m\}.
		\]
		We denote the analytic total space of $\mathfrak F_m\subseteq \omega_{\Ig(p^n)}^{+}$ over $\mathcal X_{\Ig(p^m)}(\epsilon)$ by
		\[\mathcal F_m(\epsilon)\to \mathcal X_{\Ig(p^m)}(\epsilon).\]
		By Prop.~\ref{p:prop-of-w^int}.(3), this is a torsor in the analytic topology for the subgroup $1+I'_m\Res_{\OF|\ZZ}\hat{\G}_a|_{\XX(\epsilon)}$ of $\Res_{\OF|\ZZ}\hat{\G}_m|_{\XX(\epsilon)}$, where $\Res_{\OF|\ZZ}\hat{\G}_a|_{\XX(\epsilon)}$ is the pullback of $\Res_{\OF|\ZZ}\hat{\G}_a\to \Spa(L)$ to $\XX(\epsilon)$, and similarly for $\Res_{\OF|\ZZ}\hat{\G}_m|_{\XX(\epsilon)}$. In particular, the composition $\mathcal F_m(\epsilon)\to \mathcal X(\epsilon)$ is a torsor in the \'etale topology for the subgroup $B_m:=\O_p^\times(1+I'_m\Res_{\OF|\ZZ}\hat{\G}_a|_{\XX(\epsilon)})\subseteq \Res_{\OF|\ZZ}\hat{\G}_m|_{\XX(\epsilon)}$. This also shows that the natural map $\mathcal F_m(\epsilon)\to \mathcal T(\omega)$ into the total space of $\omega$ over $\mathcal X$ is an open immersion. Finally, we note that since $p^{\epsilon}\in \Hdg$, we have for $x:=m-\epsilon p^{m}/(p-1)$ that
		\begin{equation}\label{eq:estimate-for-B_n}
		\O_p^\times(1+p^{x}\Res_{\OF|\ZZ}\hat{\G}_a)|_{\XX(\epsilon)}\subseteq B_m.
		\end{equation}
	\end{definition}
	The following corollary relates $\mathfrak F_m$ to the definition in the elliptic case:
	\begin{corollary}\label{l:comparison-of-AIP-torsor-to-Pilloni-torsor}
		Let $w \in \omega_{\Ig(p^m)}^+$ be any lift of $\psi(1)\in \omega^+_{H_m}$ under $\omega_{\Ig(p^m)}^+\to \omega^+_{H_m}$. Then $w\in \mathfrak F_m$. 
	\end{corollary}
	\begin{proof}
		By Lem.~\ref{l:Cor-A.4-AIP2}, for $w$ to be a lift of $\psi(1)$ means that $w$ and the image of $\psi(1)$ in $\omega_{\Ig(p^m)}^+/I_m$ agree. Thus $w\in \w^{\mathrm{int}}_{\Ig(p^m)}$. By the injective morphism from Proposition~\ref{p:prop-of-w^int}.(2), this shows that $\psi(1)$ and $w$ also agree in $\w^{\mathrm{int}}_{\Ig(p^m)}/I_m$ and thus in its quotient $\w^{\mathrm{int}}_{\Ig(p^m)}/I'_m$. This means $w\in \mathfrak F_m$.
	\end{proof}
	
	\begin{defn}
		
		Let $\kappa:\mathcal U\to \mathcal W^{\ast}$ be a bounded smooth weight. Recall that we may regard $\kappa$ as a morphism $\O_p^\times \times U\to \hat{\G}_m$. As before, we let $|\delta_{\kappa}|:=\max(|p|,|T_{\kappa}|)$. Let $r=3$ if $p>2$ and $r=5$ if $p=2$. Let $\epsilon_{\kappa}>0$ be implicitly defined by $|p|^{\epsilon_{\kappa}}=|\delta_{\kappa}|^{1/p^{r+1}}$. We note that $\e_{\k} \leq \ed$.
		
	\end{defn}
	\begin{defn}
		For any $k\in \Z_{\geq 1}$, let $\mathcal W^{\ast}_{k}$ be the open in weight space denoted by $\mathcal W_{F,[p^{k-1},p^{k}]}$ in \cite[\S2]{AIP3}, and let $\mathcal W_0^{\ast}$ be the open denoted by $\mathcal W_{F,[0,1]}$. Explicitly, for any $k\in \Z_{\geq 0}$, we have
		$\mathcal W^{\ast}_{k}:= \mathcal W^{\ast}(|\delta_{\kappa}|^{p^{k}}\leq |p|\leq |\delta_{\kappa}|^{p^{k-1}})$. Then $ \mathcal W^{\ast}=\cup_{k\in \Z_{\geq 0}}\mathcal W_k^*$.
		
		If $\kappa:\mathcal U\to \mathcal W^{\ast}$ is a bounded smooth weight, we let $\mathcal U_k:=\kappa^{-1}(\mathcal W^{\ast}_k)$, then $\mathcal U=\cup_{k\in \Z_{\geq 0}}\mathcal U_k$.
	\end{defn}
	
	\begin{definition}
		Let $\kappa:\mathcal U\to \mathcal W^{\ast}$ be a smooth bounded weight and let $0\leq \epsilon\leq \epsilon_{\kappa}$. For each $k\in \Z_{\geq 0}$, let $m=k+r$ (this is the variable ``$n$'' in \cite{AIP3}), so that $\epsilon_{\kappa}\leq \epsilon_m^{\mathrm{can}}$.
		The sheaf $\omega^{\kappa}_{\AIP}$ on $\XX(\e)\times \mathcal U_k$ of modular forms of weight $\kappa$, as defined in \cite{AIP3}, is given locally as
		\[
		\omega_{\AIP|\U_k}^{\kappa}:=\O_{\mathcal F_m(\epsilon)\times \mathcal U_k}[\kappa^{-1}]=\left\{f\in \O_{\mathcal F_m(\epsilon)\times \mathcal U_k}|\gamma^{\ast}f=\kappa^{-1}(\gamma)f \text{ for all }\gamma \in B_m\right\}. 
		\]
		As usual, we define an integral subspace  $\omega^{\kappa,+}_{\AIP|\U_k}$ by using $\O^+$ instead.
	\end{definition}

	\begin{prop}[{\cite[Prop.~4.3]{AIP3}}]\label{prop: glue AIP sheaf}
		Let $\kappa:\mathcal U\to\mathcal W^{\ast}$ be a smooth bounded weight and let $0\leq \epsilon\leq \epsilon_{\kappa}$. 
		Then the sheaves $\omega_{\AIP|\U_k}^{\kappa,+}$ can be canonically identified on intersections $\mathcal X(\epsilon)\times(\mathcal U_k\cap \mathcal U_{k+1})$, so that they glue to give an $\O^+$-module $\omega_{\AIP}^{\kappa,+}$ on $\XX_{\U}(\epsilon)=\XX(\e)\times\U$. This $\O^+$-module is invertible. Similarly, we can glue the   $\omega_{\AIP|\U_k}^{\kappa}$ to get a line bundle  $\omega_{\AIP}^{\kappa}$ on $\XX_\U(\e)$.
	\end{prop}
	\begin{proof}
		We first need to explain how our variable $\epsilon$ is related to the radius variable $r$ used in \cite[\S3]{AIP3}. The sheaf ``$\mathfrak w_{n,r,I}$'' constructed in \cite[\S4]{AIP3} for $I=[p^{k-1},p^k]$ lives on a formal scheme ``$\mathfrak X_{r,[p^{k-1},p^{k}]}$'' over $\Z_p$. The base-change to $L$ of the generic fibre of this formal scheme is the subspace of $\mathcal X\times \mathcal W^{\ast}_k$ cut out by the condition $|\Ha^{p^{r+1}}|\geq |\delta_{\kappa}|$. Since we have set $|p|^{\epsilon_{\kappa}}= |\delta_{\kappa}|^{1/p^{r+1}}$, this is contained in $\mathcal X(\epsilon)\times \mathcal W^{\ast}_k$. 
		By \cite[Prop.~4.3]{AIP3}, the sheaf  ``$\mathfrak w_{n,r,I}$'' is a line bundle. It now follows from our definition that the sheaf  $\omega_{\AIP|\U_k}^{\kappa,+}$ is the pullback of  ``$\mathfrak w_{n,r,I}$'' along the morphism of ringed spaces $(\mathcal X(\epsilon)\times \mathcal U_k,\O^+)\to \mathfrak X_{r,k}$. In particular, $\omega_{\AIP|\U_k}^{\kappa,+}$ is invertible.
		
		By \cite[Prop.~4.7]{AIP3}, there is a canonical isomorphism between this sheaf and the one defined using $m+1$ instead of $m$. This shows that the $\omega_{\AIP|\U_k}^{\kappa,+}$ glue on intersections $\mathcal U_k\cap \mathcal U_{k+1}$. Since $\omega_{\AIP|\U_k}^{\kappa,+}$ is invertible on each open $\mathcal X(\epsilon)\times \U_k$, it is clearly also invertible on $\mathcal X_\U(\epsilon)$.
	\end{proof}

	\subsection{The comparison morphism}\label{s:comparison-morphism-Hilbert}
	Recall from Defn.~\ref{d:mathfrak-s-Hilbert-case} that over $\XX_{\Gamma(p^\infty)}$, the conormal sheaf $\omega_{\Gamma^*(p^\infty)}$ has a canonical section $\mathfrak s$ that we may regard as a morphism into the total space $\mathfrak s:\XX^*_{\Gamma(p^\infty)}\to\mathcal T(\omega_{\Gamma^*(p^\infty)})\to \mathcal T(\omega)$.
	Let us simply write $\mathcal T(\epsilon)\to \mathcal X(\epsilon)$ for the restriction of $\T(\omega)$ to $\mathcal X(\epsilon)$, then $\mathfrak s$ restricts to
	\[\mathfrak s:\mathcal X_{\Gamma^*(p^\infty)}(\epsilon)_a\to \mathcal T(\epsilon). \]
	
	The comparison of our Hilbert modular forms to the ones defined by \cite{AIP3} relies on the following proposition saying that $\mathfrak s$ compares $\mathcal X_{\Gamma^*(p^\infty)}(\epsilon)_a$ to the Andreatta--Iovita--Pilloni-torsor.
	\begin{prop}\label{l:s-from-X-to-P}
		Let  $0\leq \epsilon\leq \epsilon^{\mathrm{can}}_{m}$. Then $\mathfrak s:\mathcal X_{\Gamma^*(p^\infty)}(\epsilon)_a\to \mathcal T(\epsilon)$ factors through the subspace
		\[\mathfrak s:\mathcal X_{\Gamma^*(p^\infty)}(\epsilon)_a\to \mathcal F_m(\epsilon) \subset \mathcal{T}(\e)\]
		defined by the Andreatta--Iovita--Pilloni torsor (see Definition~\ref{d:AIP-torsor}).
	\end{prop}
	\begin{proof}
		It suffices to check that for $C$ any complete algebraically closed extension of $L$, the $(C,C^+)$-points that the image of $\mathfrak s$ are contained in the open subspace $\mathcal F_m(\epsilon)\subseteq \mathcal T(\epsilon)$.
		
		There is a natural map 
		$\varphi:\XX_{\Gamma^*(p^\infty)}(\epsilon)_a\to \XX_{\mathrm{Ig}(p^m)}(\epsilon)$  defined by sending a point valued in some stably uniform adic ring $(R,R^+)$ corresponding to an abelian variety $A$ and a trivialisation $\alpha:\OO_p^2 \isorightarrow T_pA^{\vee}$ to the trivialisation of $H_m^\vee(A)$ given by the composition
		\[\O_p/p^m\OO_p\xrightarrow{(1,0)}(\O_p/p^m\OO_p)^2\xrightarrow{\alpha\bmod p^m}A^{\vee}[p^m]\to H_m^\vee\]
		with the dual of the inclusion $H_n\to A[p^m]$, where as usual we identify $A[p^m]^\vee=A^{\vee}[p^m]$ via the Weil pairing.
		By functoriality of the Hodge--Tate map we then have a commutative diagram
		\begin{center}
			\begin{tikzcd}[row sep = 0.55cm]
			\O_p^2\arrow[r,"\alpha"]\arrow[d,"{(1,0)\bmod p^n}"']&T_pA^{\vee} \arrow[d] \arrow[r, "\HT"] & \omega^+_A \arrow[d] \\
			\psi:\O_p/p^m\O_p\arrow[r]&	H_m^\vee \arrow[r, "\HT"] & \omega^+_{H_m}.
			\end{tikzcd}
		\end{center}
		For $(R,R^+)=(C,C^+)$, we then have $\mathfrak s(x)=\HT\circ\alpha(1,0)$ by Lem.~\ref{l:moduli interpretation of s in terms of HT in Hilbert}. This shows that $\mathfrak s(x)$ is a lift of $\psi(1)\in \omega^+_{H_m}$. By Cor.~\ref{l:comparison-of-AIP-torsor-to-Pilloni-torsor}, this implies that $\mathfrak s(x)\in\mathcal F_m(\epsilon)(C,C^+)$ as desired.
	\end{proof}
	\begin{lem}\label{l:comparison-mod-forms-AIP-forms-Hilbert}
		The following diagram commutes
		\begin{center}
			\begin{tikzcd}[row sep = 0.55cm]
			\Gamma_{\!0}^*(p)\times\mathcal X_{\Gamma^{\ast}(p^\infty)}(\epsilon)_a \arrow[r, "\mathrm{m}"] \arrow[d, "(c\mathfrak z+d)\times\mathfrak s"'] & \mathcal X_{\Gamma^{\ast}(p^\infty)}(\epsilon)_a \arrow[d, "\mathfrak s"'] \arrow[r] &  \mathcal X_{\Gamma^{\ast}_{\!\scaleto{0}{3pt}}(p^\infty)}(\epsilon)_a \arrow[d,"q"] \\
			B_m\times \mathcal F_m(\epsilon) \arrow[r, "\mathrm{m}"] & \mathcal F_m(\epsilon) \arrow[r] & \mathcal X(\epsilon).
			\end{tikzcd}
		\end{center}
	\end{lem}
	\begin{proof}
		We first note that the morphism $\mathcal X_{\Gamma^{\ast}(p^\infty)}(\epsilon)_a\to B_m$ is well-defined: for this we use that by \eqref{eq:estimate-for-B_n}, we have $\O_p^\times(1+p^x\Res_{\OO_F|\ZZ}\hat{\G}_a)\mid_{\mathcal X(\epsilon)} \subseteq B_m$. Moreover, by \eqref{eq:sharp-estimate-on-image-of-pi_HT-in-proof-of-image-of-HT-Hilbert-case}, the map $\mathfrak z$ already restricts to $\mathcal X_{\Gamma^{\ast}(p^\infty)}(\epsilon)_a\to B_{|p^x|}(\O_p:1)=\O_p^\times(1+p^x\Res_{\OO_F|\ZZ}\hat{\G}_a)$.
		The left square now commutes by Lem.~\ref{l:transform-for-mathfrak-s-Hilbert-case}. Commutativity of the right square is clear.
	\end{proof}
	
	Combining this with the morphism $u_n:\mathcal X_{\Gamma^{\ast}(p^\infty)}(p^n\epsilon)_a\to \mathcal X_{\Gamma^{\ast}(p^\infty)}(\epsilon)_a$ defined by the action of the matrix $\smallmatrd{p^n}{0}{0}{1}$, we obtain from the lemma a commutative diagram
	\begin{equation}\label{dg:comparison-mod-forms-AIP-forms_up}
	\begin{tikzcd}[row sep = 0.55cm]
	\Gamma_{\!0}^*(p)\times\mathcal X_{\Gamma^{\ast}(p^\infty)}(p^n\epsilon)_a \arrow[r, "\mathrm{m}"] \arrow[d, "(c\mathfrak z+d)\times \tilde{\mathfrak s}"'] & \mathcal X_{\Gamma^{\ast}(p^\infty)}(p^n\epsilon)_a \arrow[d, "  \tilde{\mathfrak s}"'] \arrow[r] & \mathcal X_{\Gamma_{\!\scaleto{0}{3pt}}^{\ast}(p^n)}(p^n\epsilon)_a \arrow[d,"\AL^n","\sim"labl] \\
	B_m\times \mathcal F_m(\epsilon) \arrow[r, "\mathrm{m}"] & \mathcal F_m(\epsilon) \arrow[r] & \mathcal X(\epsilon)
	\end{tikzcd}
	\end{equation}
	where $\tilde{\mathfrak s}=\mathfrak s\circ u_n$. From this, we finally deduce the following Hilbert analogue of Thm.~\ref{thm:comparison 1}.
	
	\begin{definition}
		\begin{enumerate}
			\item For any $n\in\Z_{\geq 0}$ we set $\omega^{\k,+}_{\mathrm{AIP},n}:=\AL^{n\ast}\omega^{\k,+}_{\mathrm{AIP}}$ where 
			\[\AL^n:\XX_{\Gamma_{\!\scaleto{0}{3pt}}^*(p^n)}(p^n\epsilon)_a\isorightarrow \XX(\e)\]
			is the Atkin--Lehner isomorphism. By \cite[Thm.~6.7.3]{AIP3}, the restriction of $\omega^{\k,+}_{\mathrm{AIP},n}$ to $\XX_{\Gamma_{\!\scaleto{0}{3pt}}(p^n)}(\epsilon)_a$ equals $q_n^{\ast}\omega^{\k,+}_{\mathrm{AIP}}$, where $q_n:\XX_{\Gamma_{\!\scaleto{0}{3pt}}^*(p^n)}(\epsilon)_a\to \XX(\epsilon)$ is the forgetful map.
			\item We then set $\omega^{\k,+}_{\mathrm{AIP},\infty}:=q^{\ast}\omega^{\k,+}_{\mathrm{AIP}}$ where $q:\XX_{\Gamma_{\!\scaleto{0}{3pt}}^*(p^\infty)}(\epsilon)_a\to \XX(\epsilon)$ is the forgetful map.
		\end{enumerate}
	\end{definition}
	\begin{theorem}\label{thm:comparison hilbert}
		Let $\k:\U \to \W^*$ be a bounded smooth weight. Let $0\leq \epsilon\leq \epsilon_{\kappa}$.
		Then for any $n \in\ZZ_{\geq 0} \cup \{\infty\}$, the map $\tilde{\mathfrak s}^{\ast}$ induces a
		Hecke-equivariant isomorphism of $\OO^+_{\XX_{\U,\Gamma_{\!\scaleto{0}{3pt}}(p^n)}(\e)_a}$-modules 
		\[\tilde{\mathfrak s}^{\ast}:\w^{\k,+}_{n}\isorightarrow\w^{\k,+}_{\mathrm{AIP},n}.\]
		In particular, $\w^{\k,+}_{n}$ is an invertible $\OO^+_{\XX_{\U,\Gamma_{\!\scaleto{0}{3pt}}(p^n)}(\e)_a}$-module, and  $\w^{\k}\cong \w^{\k}_{\AIP}$.
	\end{theorem}
	\begin{proof}
		With the preparations from this and the last section, we can argue as in the elliptic case.
		It suffices to prove this locally on $\W^{*}$. We may therefore without loss of generality assume that $\kappa$ has image in $\W^{\ast}_k$ for some $k\in \Z_{\geq 1}$ and set $m:=k+r$. 
		
		We first check that $\mathfrak s$ induces a map $\w^{\k,+}_{\mathrm{AIP},n}\to \w^{\k,+}_{n}$. This is because for any section $f$ of $\w^{\k,+}_{\mathrm{AIP},n}$ and for any $\g \in \Gamma_{\!0}^*(p^n)$, diagram~\eqref{dg:comparison-mod-forms-AIP-forms_up} implies that the following diagram commutes:
		
		\begin{center}
			\begin{tikzcd}[column sep = 1.8cm,row sep = 0.55cm]
			\mathcal X_{\U,\Gamma^*(p^\infty)}(p^n\epsilon)_a \arrow[d,  "\gamma"'] \arrow[r,  "(c\mathfrak z+d)\times \tilde{\mathfrak s}"] & B_m\times \mathcal F_m(\epsilon) \arrow[d,  "\mathrm m"'] \arrow[r,  "\kappa^{-1}\times f"] & \mathbb G_m\times \mathbb A^{\an} \arrow[d,  "\mathrm m"']\\
			\mathcal X^{}_{\U,\Gamma^*(p^\infty)}(p^n\epsilon)_a \arrow[r,  "\tilde{\mathfrak s}"] & \mathcal F_m(\epsilon) \arrow[r,  "f"] & \mathbb A^{\an},
			\end{tikzcd}
		\end{center}
		As before, this together with $\w^{\k,+}_{\mathrm{AIP},n}$ being invertible proves the theorem for $n\in \Z_{\geq 1}$ since 
		\[(\OO^+_{\mathcal X_{\U,\Gamma^*(p^\infty)}(\epsilon)_a})^{\Gamma_{\!\scaleto{0}{3pt}}^*(p^n)}=\OO^+_{\mathcal X_{\U,\Gamma_{\!\scaleto{0}{3pt}}^*(p^n)}(\epsilon)_a}\]
		by Lem.~\ref{gen lemma for sous}.
		The case of $n=\infty$ follows by the same argument from the diagram in Lem.~\ref{l:comparison-mod-forms-AIP-forms-Hilbert}. The case of $n=0$ follows from $\w^{\k,+} = \AL^{\ast}\w^{\k,+}_{1}=\w^{\k,+}_{\mathrm{AIP},1}$. Finally, the isomorphism $\w^{\k}=\w^{\k}_{\mathrm{AIP}}$ is induced from the integral one by inverting $p$.
		
		We postpone the proof of Hecke equivariance to \S\ref{sec:hecke-operators} where we discuss the Hecke action.
	\end{proof}

	\section{Perfectoid Hilbert modular varieties for $G$}\label{sec:G vars}
	We now pass from modular forms for $G^*$ to those for $G$,  the so-called arithmetic Hilbert modular forms. This  requires a closer study of the perfectoid modular varieties attached to $G$, and their relation to the perfectoid modular varieties of $G^*$, which is the subject of this section.
	
	\begin{nota}
		Recall that in \S\ref{s:Hilbert modular varieties as moduli spaces} we have defined Hilbert modular varieties $X,X_{G}$ over $L$ which are base-changes of models for the Shimura varieties attached to $G^*$ and $G$ respectively. In doing so, we had fixed a choice of tame level $\mu_N$ as well as polarisation ideal $\gothc$ and omit these notation.  We denote by $\XX,\XX_{G}$ the adic analytifications of $X,X_G$.
	\end{nota}

	We begin by recalling that the action of $\OO_F$ is a source of isomorphisms of HBAV:
	\begin{lem}\label{l:isom-induced-by-real-mult}
		Let $S$ be any ring and let $(A,\iota,\lambda,\mu,\alpha)$ be a HBAV over $S$ where $\mu$ is a $\mu_N$-structure and $\alpha$ is either a $\Gamma_0(p^n)$, $\Gamma_1(p^n)$ or $\Gamma(p^n)$-level structure.
		Then for any $\eta \in \OO_F^\times$, the map $\iota(\eta):A\to A$ induces an isomorphism of HBAV 
		\[\eta:(A,\iota,\eta^2 \lambda,\eta^{-1} \mu_N,\eta \alpha)\isorightarrow (A,\iota,\lambda,\mu_N,\alpha).\]
		Here we write $\eta\alpha$ as a shorthand notation for the composition of the $\O_F$-linear map $\alpha$ with multiplication by $\eta$ on either side of $\alpha$. Similarly for $\eta^2 \lambda$ and $\eta^{-1} \mu_N$.
	\end{lem}
	\begin{proof}
		A morphism of HBAVs $(B,\iota',\lambda', \mu',\alpha')\to(A,\iota,\lambda,\mu,\alpha)$ is an $\O_F$-linear isogeny $\varphi:B\to A$ making the following diagrams commute:
		\begin{center}
			\begin{tikzcd}[column sep = 0.6cm]
			B^{\vee} & A^{\vee} \arrow[l, "\varphi^\vee"'] &  & \mathfrak d^{-1}\otimes\mu_N \arrow[d, "\mu'_N", hook] \arrow[r,equal] & \mathfrak d^{-1}\otimes \mu_N \arrow[d, "\mu_N", hook] &  & (\O/p^n)^2 \arrow[d, "\alpha'"] & (\O/p^n)^2 \arrow[d, "\alpha"] \arrow[l,equal] \\
			B\otimes\gothc \arrow[r, "\varphi"] \arrow[u, "\lambda'"'] & A\otimes \gothc \arrow[u, "\lambda"'] &  & B \arrow[r, "\varphi"] & A &  & {B^\vee[p^n]} & {A^{\vee}[p^n]}. \arrow[l, "\varphi^\vee"']
			\end{tikzcd}
		\end{center}
		Setting $B=A$, and $\varphi=[\eta]$, we see that $\lambda'=[\eta]^\vee\circ \lambda \circ [\eta]=\eta^2\lambda$, where in the second step we have used that $[\eta]^\vee\circ \lambda=\lambda \circ [\eta]$ since $\iota$ is stable under the Rosati involution. The second diagram implies $\mu'_N=[\eta]^{-1}\circ \mu_N=\eta^{-1}\mu_N$. The third diagram implies
		$\alpha' = [\eta]^\vee\circ \alpha=\eta\alpha.$
	\end{proof}
	
	\begin{definition}\label{d:polarisation action}
		The polarisation action of $\OO^{\times,+}_{F}$ on $X$ is given by letting
		$\eta\in \OO^{\times,+}_{F}$ act via
		\[\eta \cdot_{\pol} (A,\iota,\lambda,\mu_N)=(A,\iota,\eta\lambda,\mu_N).\]
	\end{definition} 
	
	As a consequence of Lemma~\ref{l:isom-induced-by-real-mult} we see:
	\begin{proposition}\cite[Lem.~8.1]{AIP3}.
		\label{prop:delta torsor}
		This action of $\OO^{\times,+}_{F}$ on $X$ factors through the finite group $\Delta(N):= \OFU/(1+\n\OO_F)^{\times 2}$, and makes $X\rightarrow X_{G}$ into a finite \'etale $\Delta(N)$-torsor.
	\end{proposition}

	\subsection{The anti-canonical tower for $G$}
	Our next goal is to construct from the anti-canonical tower of $G^*$ the anti-canonical tower for $G$. For this, we move on from schemes to adic spaces.
	
	Recall that, for either $G$ or $G^*$, the datum of a $\Gamma_{\!0}(p^n)$-level structure is a the choice of an $\OO_F$-submodule $ C \subseteq A[p^n]$ etale locally isomorphic to $\OO_F/p^n\OO_F$. If necessary we will denote this level structure  by $\Gamma_{\!0}(p^n)$ and $\Gamma_{\!0}^*(p^n)$ on $G$ and $G^*$ respectively. Similarly, on both $G^*$ and $G$ the notions of anticanonical level structures coincide. We thus obtain commutative diagrams
	
	\begin{center}
		\begin{tikzcd}[row sep = 0.55cm]
		{X_{\Gamma_{\!\scaleto{0}{3pt}}^*(p^n)}} \arrow[d] \arrow[r] & {X_{G,\Gamma_{\!\scaleto{0}{3pt}}(p^n)}} \arrow[d]&&{\XX_{\Gamma_{\!\scaleto{0}{3pt}}^*(p^n)}}(\epsilon)_a \arrow[d] \arrow[r] & {\XX_{G,\Gamma_{\!\scaleto{0}{3pt}}(p^n)}}(\epsilon)_a \arrow[d] \\
		X\arrow[r] & X_G&&\XX(\epsilon) \arrow[r] & \XX_G(\epsilon)
		\end{tikzcd}
	\end{center}
	\begin{lem}\label{lem: et tor}
		The above diagrams are both Cartesian. In particular, the morphism of adic spaces $\XX_{\Gamma_{\!\scaleto{0}{3pt}}^{\ast}(p^n)}(\epsilon)_a\to \XX_{G,\Gamma_{\!\scaleto{0}{3pt}}(p^n)}(\epsilon)_a$ is a finite \'etale $\Delta(N)$-torsor.
	\end{lem}
	\begin{proof}
		The points of $X_{\Gamma_{\!\scaleto{0}{3pt}}^{\ast}(p^n)}$ correspond to HBAV $(A,\iota,\lambda,\mu_N,D)$ where $D\subseteq A[p^n]$ is a subgroup \'etale locally isomorphic to $\O/p^n$. For any $\eta \in \OO_F^\times$, as $D$ is an $\OO_F$-module, we have $\eta D = D$, and hence the isomorphism of HBAVs induced by $\eta$ from Lem.~\ref{l:isom-induced-by-real-mult} identifies the tuples $(A,\iota,\eta^2\lambda, \eta^{-1}\cdot\mu_N, D)$ and $(A,\iota,\lambda,\mu_N,D)$ in $X_{\Gamma_{\!\scaleto{0}{3pt}}^{\ast}(p^n)}$. Thus, exactly as in Prop.~\ref{prop:delta torsor}, the polarisation action factors through $\Delta(\n)$.
		It follows on the level of relative moduli descriptions that the top map in the above diagram is a $\Delta(\nn)$-torsor. This shows that the left diagram is Cartesian.
		
		The case of the right hand side follows after adic analytification and restriction.
	\end{proof}
	
	By the following lemma, this allows us to pass to tilde-limits in the diagram on the right:
	
	\begin{lemma}\label{l:perf-tilde-limits-and-quotients}
		Let $\Delta$ be a finite group. Let $(\cX_n\to \cY_n)_{n\in\N}$ be an inverse system of \'etale $\Delta$-torsors of adic spaces over $L$. Suppose there is a perfectoid tilde-limit $\mathcal X_\infty\sim \varprojlim \cX_n$. We moreover impose the technical condition that there is a cover of $\cY_0$ by open affinoid spaces $V_0$, with affinoid pullbacks $U_n:=\Spa(R_n,R_n^+)$ to $\cX_n$, such that $R_\infty$ is affinoid perfectoid and $\varinjlim_n R_n\to R_\infty$ has dense image.
		Then $\cX_\infty/\Delta =: \cY_\infty$ is  perfectoid, $\cX_\infty\to\cY_\infty$ is an \'etale $\Delta$-torsor, and $\cY_\infty \sim \varprojlim \cY_n$.
	\end{lemma}
	This is a special case of the statement of \cite[Cor.~2.3.5]{Xu}. We focus on this special case since it suffices for our applications, and has the following simple proof:
	\begin{proof}
		Since the conclusions are local, by restricting to $V_0$ we are immediately reduced to the case that all $\cX_n=\Spa(R_n,R_n^+)$ are affinoid, $\cX_\infty=\Spa(R_\infty,R_\infty^+)$ is affinoid perfectoid and $\varinjlim R_n\to R_\infty$ has dense image. Then $\cY_n=\Spa(R_n^{\Delta},R_n^{+,\Delta})$.
		It follows from \cite[Thm.~1.4]{Hansenquots} that $\cY_\infty:=\cX_\infty/\Delta=\Spa(R_\infty^{\Delta},R_\infty^{+,\Delta})$ is perfectoid.  Note that the assumptions are satisfied because we work over the perfectoid field $L$ over $\Q_p$.  
		
		We claim that $\varinjlim_nR_n^\Delta\to R^{\Delta}_\infty$ has dense image. To see this, let $r\in R^{\Delta}_\infty$ and let $r':=r/|\Delta|$. Then we can find $r'_n\in R_n$ such that $r'_n\to r'$ inside $R_\infty$. Let now $r_n:=\sum_{g\in \Delta} gr'_n$, then clearly $r_n\in R_n^\Delta$. Since the $\Delta$-action is continuous, and $r\in R_\infty^\Delta$, we then have $r_n\to \sum_{g\in \Delta} gr'=\frac{1}{|\Delta|}\sum_{g\in \Delta} gr=r$, as desired. This shows that $\cY_\infty \sim \varprojlim \cY_n$, since the condition on topological spaces follows from $|\cY_\infty|=|\cX_\infty|/\Delta$.
		
		That $\cX_\infty \to \cY_\infty$ is a $\Delta$-torsor now follows from the Cartesian diagrams expressing that $\cX_n \to \cY_n$ is a $\Delta$-torsor, by commuting product and perfectoid tilde-limit.
	\end{proof}
	\begin{prop}
		There exists a perfectoid tilde-limit
		\[
		\XX_{G,\Gamma_{\!\scaleto{0}{3pt}}(p^\infty)}(\e)_a \sim \varprojlim \XX_{G,\Gamma_{\!\scaleto{0}{3pt}}(p^n)}(\e)_a.
		\]
		Moreover, the natural map $\XX_{\Gamma^{\ast}_{\!\scaleto{0}{3pt}}(p^\infty)}(\e)_a \to \XX_{G,\Gamma_{\!\scaleto{0}{3pt}}(p^\infty)}(\e)_a$ is a finite \'etale $\Delta(N)$-torsor.
	\end{prop}
	\begin{proof}
		We apply Lem.~\ref{l:perf-tilde-limits-and-quotients} to the system of $\Delta(\n)$-torsors $\XX_{\Gamma_{\!\scaleto{0}{3pt}}(p^n)}(\e)_a \rightarrow \XX_{G,\Gamma_{\!\scaleto{0}{3pt}}(p^n)}(\e)_a$. To see that the technical condition is satisfied, we note that by construction in Prop.~\ref{p:X_Gamma*_0(p^infty) for G*}, and by \cite[Prop.~2.4.2]{SW}, any affine open formal subscheme of the formal model $\mathfrak X^{\ast}(\epsilon)$ of $\XX^*(\epsilon)$ pulls back to opens $\Spa(R_n,R_n^+)\subseteq \XX_{\Gamma_{\!\scaleto{0}{3pt}}^*(p^n)}(\epsilon)_a$ for all $n\in \Z_{\geq 0}\cup \{\infty\}$ such that $R_\infty$ is affinoid perfectoid and $\varinjlim R_n\to R_\infty$ has dense image. Thus also for any rational subset $\Spa(R_0,R_0^+)(\frac{T}{s})$, the map $\varinjlim R_n\langle \frac{T}{s}\rangle \to R_\infty\langle \frac{T}{s}\rangle$ has dense image. 
		We conclude that any affinoid cover of $\XX_{G}$ that is a rational refinement of an affine formal cover of $\mathfrak X^{\ast}_{G}(\e):=\mathfrak X^{\ast}(\e)/\Delta(N)$ is of the desired form.
	\end{proof}

	\subsection{Wild full level structures}
	
	As before, we ultimately want to work with the full level structures $\Gamma(p^n)$ for $G$ defined in \S\ref{Section: lvl HMFs}, since this will provide the appropriate universal covering spaces we need to define overconvergent Hilbert modular forms. In this case, though, there are differences between $\Gamma(p^n)$ and $\Gamma^*(p^n)$-level structures which introduce new subtleties.
	
	There is a space $\XX_{G,\Gamma(p^n)} \to \XX_{G,\Gamma_{\!\scaleto{0}{3pt}}(p)}$ which relatively represents a choice of an isomorphism $\a_n:(\O/p^n)^2\to A^{\vee}[p^n]$ such that $\a_n(1,0)$ generates the corresponding subgroup for the $\Gamma_{\!0}(p^n)$-level, and there is a natural map $\XX_{\Gamma^*(p^n)} \rightarrow \XX_{G,\Gamma(p^n)}$. However, this map is not a torsor, as it is not surjective. In addition, $\cO_F^{\times,+}$ no longer admits a polarisation action on $\XX_{\Gamma^*(p^n)}$, due to the restrictive additional conditions on $\Gamma^*(p^n)$-level structures: indeed, changing $\lambda$ changes the isomorphism $b$ in diagram \eqref{eq:similitude}, and in general true this will not result in a similitude of pairings.
	
	\subsubsection{`Hybrid' full level structures}
	
	Bearing all of the above in mind, it is convenient to also introduce an intermediate space $\XX_{\Gamma(p^n)} \rightarrow \XX$, relatively representing a choice of $\Gamma(p^n)$-level structure $\a_n$ over the Shimura variety $\XX$ for $G^{\ast}$. This fits into a diagram
	\begin{equation}\label{eq:diag comparing G*,Gamma*(p^n) to G,Gamma(p^n) over tame}
	\begin{tikzcd}[row sep = 0.55cm]
	{\mathcal X_{\Gamma^{\ast}(p^n)}} \arrow[d] \arrow[rr, "\beta_1"] &  & {\mathcal X_{\Gamma(p^n)}} \arrow[d] \arrow[rr, "\beta_2"] &  & {\mathcal X_{G,\Gamma(p^n)}} \arrow[d] \\
	\mathcal X \arrow[rr, equal] &  & \mathcal X \arrow[rr] &  & \mathcal X_G.
	\end{tikzcd}
	\end{equation}
	
	The polarisation action \eqref{d:polarisation action} now gives a well-defined action on this `hybrid' space $\XX_{\Gamma(p^n)}$, since the $\Gamma(p^n)$-level structures require no Weil pairing compatibility. 
	
	We also have a second natural left-action on this space:
	\begin{definition}
		The level structure (LS) action of  $G(\ZZ/p^n\ZZ)=\GL_2(\O_F/p^n\O_F)$ on $\XX_{\Gamma(p^n)}$ is given by letting $\gamma\in G(\ZZ/p^n\ZZ)$ act as
		\[\gamma \cdot_{\mathrm{LS}}(A,\iota,\lambda,\mu_N,\alpha_n) \mapsto (A, \iota, \lambda, \mu_N, \alpha_n\circ \gamma^{\vee}),\hspace{12pt} \g^\vee=\det(\g)\g^{-1}.\]
	\end{definition}
	\subsubsection{Components and the $\OO_F$-linear Weil pairing}
	
	To understand the map $\XX_{\Gamma^*(p^n)} \rightarrow \XX_{G,\Gamma(p^n)}$, we begin by analysing $\beta_1$. For this, we require a description of the components of $\XX_{\Gamma(p^n)}$ through the $\OO_F$-linear Weil pairing $\widetilde{\mathbf{e}}_n$ (Defn.~\ref{def:linearisation}):
	Suppose $(A,\iota,\lambda)$ is a $\mathfrak c$-HBAV over $S$ with $\Gamma(p^n)$-level structure $\a_n:(\mathcal O_F/p^n\mathcal O_F)^2\isorightarrow  A^\vee[p^n]$. Then $\widetilde{\mathbf{e}}_n$ induces an $\OO_F$-linear isomorphism
	\begin{align}\label{TX pairing}
	\O_F/p^n\OF \otimes \gothc^{-1} =& (\mathcal O_F/p^n\mathcal O_F)^2  \wedge(\mathcal O_F/p^n\mathcal O_F \otimes \gothc^{-1})^2\\
	\xrightarrow{\a_n\wedge (\a_n \otimes \gothc^{-1})} &A^\vee[p^n]\wedge A^\vee[p^n]\otimes \gothc^{-1}\xrightarrow{(1,\lambda^{-1}\otimes\mathrm{id})}A^\vee[p^n]\wedge A[p^n]\xrightarrow{\widetilde{\mathbf{e}}_n}  \mathfrak  d^{-1}\otimes_{\Z}\mu_{p^n}.\nonumber
	\end{align}
	Equivalently, by tensoring with $\c$, this can be described as a generator of the $\O_F$-module scheme $\c \d^{-1}\otimes_{\Z}\mu_{p^n}$. We denote the subscheme of generators by $(\c \d^{-1}\otimes_{\Z}\mu_{p^n})^\times$.
	
	In the universal situation over $\XX_{\Gamma(p^n)}$, we conclude that the Weil pairing gives rise to a map
	\begin{equation}\label{m:Weil-pairing-morphisms}
	\mathbf e_{n,\beta}: \XX_{\Gamma(p^n)}\rightarrow (\c \d^{-1}\otimes_{\Z}\mu_{p^n})^\times\xrightarrow{\beta^{-1}}(\O_F/p^n\OF)^\times.
	\end{equation}
	where we recall that $\beta\in \c \d^{-1}\otimes_{\Z}\mu_{p^n}$ is the isomorphism chosen in Defn.~\ref{def:linearisation}.
	
	Similarly for $G$ we get a map $\mathbf e_{n}: \XX_{G,\Gamma(p^n)}\rightarrow (\c \d^{-1}\otimes_{\Z}\mu_{p^n})^\times$.
	Next, we record two equivariance properties of the linearised Weil pairing.
	\begin{lem}\label{l:equivariance of Weil pairing morphism from X_G^*,Gamma(p^n)}
		
		\begin{enumerate}
			\item For $\gamma\in G(\ZZ/p^n\ZZ)$, the action on $\XX_{\Gamma(p^n)}$ fits into a commutative diagram
			\begin{center}
				\begin{tikzcd}[row sep = 0.55cm]
				\XX_{\Gamma(p^n)} \arrow[d, "\gamma"] \arrow[r, "\mathbf e_{n,\beta}"] & {(\O_F/p^n\OF)^\times} \arrow[d, "\det(\gamma)"] \\
				\XX_{\Gamma(p^n)} \arrow[r, "\mathbf e_{n,\beta}"] & {(\O_F/p^n\OF)^\times}.
				\end{tikzcd}
			\end{center}
			\item For $\eta\in \OO_F^{\times,+}$, the polarisation action on $\XX_{\Gamma(p^n)}$ fits into a commutative diagram
			\begin{center}
				\begin{tikzcd}[row sep = 0.55cm]
				\XX_{\Gamma(p^n)} \arrow[d, "\eta"] \arrow[r, "\mathbf e_{n,\beta}"] & (\O_F/p^n\OF)^\times \arrow[d, "\eta^{-1}"] \\
				\XX_{\Gamma(p^n)} \arrow[r, "\mathbf e_{n,\beta}"] &  (\O_F/p^n\OF)^\times.
				\end{tikzcd}
			\end{center}
		\end{enumerate}
	\end{lem}
	
	\begin{proof}
		Recall that $\gamma$ acts by sending $\alpha_n$ to $\alpha_n\circ \gamma^\vee$. From \eqref{TX pairing} it is then clear that $\mathbf e_{n,\beta}\circ \gamma=\det(\gamma^\vee)\mathbf e_{n,\beta}$. The first part then follows from $\det\gamma^\vee=\det \gamma$. 
		The second part holds since replacing $\lambda \to \eta\circ\lambda$ multiplies \eqref{TX pairing} by $\eta^{-1}$ due to the $\lam^{-1}$ appearing.
	\end{proof}
	
	For any $c\in (\OF/p^n\OF)^\times$, the fibre of $\mathbf e_{n,\beta}$ over $c$ gives a component of $\XX_{\Gamma(p^n)}$. This will be a connected component, as we shall discuss in the next section. Recall that by Defn.~\ref{lvls for G*}, the $\Gamma^{\ast}(p^n)$-level depends on $\beta\in c \d^{-1}(1)$. We can now describe the space $\XX_{\Gamma^*(p^n)}$ as follows:
	
	\begin{lem}\label{l:Weil pairing map fits into commut diag comp G* to G}
		The morphism $\XX_{\Gamma^*(p^n)}\to \XX_{\Gamma(p^n)}$ fits into a Cartesian diagram
		\begin{center}
			\begin{tikzcd}[row sep = 0.55cm]
			{\XX_{\Gamma^*(p^n)}} \arrow[d, hook] \arrow[r, "\mathbf e_{n,\beta}"] & (\Z/p^n\Z)^\times \arrow[d, hook]\\
			{\XX_{\Gamma(p^n)}} \arrow[r] \arrow[r, "\mathbf e_{n,\beta}"] & (\OF/p^n\OF)^\times
			\end{tikzcd}
		\end{center}
	\end{lem}
	
	\begin{proof}
		We can check this on the level of schemes, where we can check on the level of moduli functors. Let $(A,\iota,\lambda,\mu_N,\alpha)$ be a HBAV over $S$ corresponding to a point $x\in\XX_{\Gamma^*(p^n)}(S)$.
		Then by \eqref{df:tilde e_n}, we recover the Weil pairing $e_{p^n}$ from $\widetilde{\mathbf{e}}_{n,\beta}$ by composing with $\Tr$. The level structure $\alpha$ therefore composes with $\lambda^{-1}$ and $\widetilde{\mathbf{e}}_{n,\beta}$ to a pairing
		\[(\O_F/p^n\OF)^2 \times (\O_F/p^n\OF)^2 \otimes \gothc^{-1}\to \mathfrak d^{-1}\otimes \mu_{p^n}\xrightarrow{\Tr}\mu_{p^n}.\]
		By definition, $\alpha$ is a $\Gamma_\gothc^{\ast}(p^n)$-level structure if and only if this pairing is similar to the pairing
		\[(\O_F/p^n\OF)^2 \times (\O_F/p^n\OF)^2  \otimes \gothc^{-1}\xrightarrow{ \operatorname{id} \otimes \beta} (\OF/p^n\OF)^2  \times   (\mathfrak d^{-1}\otimes_{\Z}\mu_{p^n})^2 \xrightarrow{\Tr}\mu_{p^n}.\]
		After evaluation at $1$ in the second factor, and tensoring with $\c$, these each induce isomorphisms   $\varphi_1,\varphi_2:\O_F/p^n\OF\to \mathfrak \c \d^{-1}\otimes \mu_{p^n}$. After composing with $\beta^{-1}$, the map $\varphi_2$ derived from the second pairing has image in $(\Z/p^n\Z)^\times$. The above pairings are now similar if and only if their ratio $\varphi_1/\varphi^{-1}_2$ is in $\Aut(\mu_{p^n})\subseteq \Aut(\mathfrak \c \d^{-1}\otimes \mu_{p^n})$, i.e.\ given by multiplication with $(\Z/p^n\Z)^\times\subseteq (\OF/p^n\OF)^\times$. 
		Thus $\alpha$ is a $\Gamma^{\ast}(p^n)$-level structure if and only if $\mathbf{e}_{n,\beta}(x)$ is in $(\Z/p^n\Z)^\times$.
	\end{proof}

	\subsubsection{The map $\beta_1$}
	Lems.~\ref{l:equivariance of Weil pairing morphism from X_G^*,Gamma(p^n)} and \ref{l:Weil pairing map fits into commut diag comp G* to G}  immediately imply that $\XX_{\Gamma(p^n)}$ is a disjoint union of copies of $\XX_{\Gamma^{\ast}(p^n)}$. More precisely, they imply the following corollary.
	\begin{corollary}\label{c:X_G*,Gamma* to X_G*,Gamma is open immersion}
		Let $(\OO_F/p^n\OO_F)^{\times}$ act on $\mathcal X_{\Gamma(p^n)}$ by letting $\eta$ act via the level structure action of the matrix $\smallmatrd{\eta}{0}{0}{1}$. The restriction of the action map to
		\[
		(\OO_F/p^n\OO_F)^\times\times\XX_{\Gamma^{\ast}(p^n)} \to \XX_{\Gamma(p^n)} 
		\]
		is then a $(\ZZ/p^n\ZZ)^{\times}$-torsor for the antidiagonal action 
		and induces an isomorphism
		\[
		\big[ (\OO_F/p^n\OO_F)^{\times}\times \mathcal X_{\Gamma^*(p^n)}\big]/(\ZZ/p^n\ZZ)^{\times} \to \mathcal X_{\Gamma(p^n)}. 
		\]
	\end{corollary}

	\subsubsection{The map $\beta_2$}
	Next, we study the second map from \eqref{eq:diag comparing G*,Gamma*(p^n) to G,Gamma(p^n) over tame}, namely $\beta_2:\XX_{\Gamma(p^n)}\to\XX_{G,\Gamma(p^n)}$.
	
	\begin{lem}\label{l:comparing Gamma-action to polarisation action}
		For any $\eta\in (1+N\OO_{F})^\times$, set $\g:=\smallmatrd{\eta}{0}{0}{\eta}$. Then the polarisation action of $\eta^{2}\in\OO_{F}^{\times,+}$ on $\XX_{\Gamma(p^n)}$ coincides with the LS action on $\XX_{\Gamma(p^n)}$  by $\g^{-1}$. In particular, if $\eta \in (1+p^n N \OO_{F})^\times$, then the polarisation action of $\eta^{2}$ on $\XX_{\Gamma(p^n)}$ is trivial. 
	\end{lem}
	\begin{proof}
		If $\eta \in (1 + N\OO_F)^\times$, then $\eta^{-1} \cdot \mu_N = \mu_N$. We note that acting on the polarisation via $\eta^2$ and then composing with the LS action by $\g=\g^\vee$ sends an HBAV $(A,\iota,\lambda,\mu_N,\alpha)$ to $(A,\iota,\eta^2\lambda,\mu_N,\eta\alpha)$ which is isomorphic to $(A,\iota,\lambda,\mu_N,\alpha)$ by Lem.~\ref{l:isom-induced-by-real-mult}, giving the first statement. The second is immediate, since if $\eta \in (1+ p^n\n\OO_F)^\times$, then $\eta$ acts trivially on the level structure.
	\end{proof}
	\begin{defn}
		For the tame levels we consider, the space $\cX$ is connected, but this is no longer true of the spaces  $\mathcal X_{\Gamma(p^n)}$ and $\mathcal X_{G,\Gamma(p^n)}$ for $n\geq 1$. We denote by $\mathcal X^0_{\Gamma(p^n)}$ and $\mathcal X^0_{G,\Gamma(p^n)}$ the respective identity components.  
	\end{defn}
	\begin{defn}
		Let $U_n$ be the cokernel in the exact sequence
		\[1\to (1+p^n\O_F)^{\times,+}\to\O_F^{\times,+}\to (\O/p^n\O)^\times\to U_n\to 1.\]
	\end{defn}
	\begin{lem}\label{l:connectec component groups}
		Via the Weil pairing, the sets of connected components are:
		
		\begin{enumerate} \item $\pi_0(X_{\mathfrak c,\Gamma(p^n)})=(\OF/p^n\OF)^\times$,
			\item $\pi_0(X_{G,\mathfrak c,\Gamma(p^n)})=U_n$.
			
		\end{enumerate}
	\end{lem}
	\begin{proof}
		It suffices to prove this for $L=\Q_p^\cyc$, and we may choose a $\Q$-linear embedding $L\hookrightarrow \C$.
		The $\CC$-points of $X_{\gothc,\Gamma^{\ast}(p^n)}$ then admit a description as $G^*(\Q)^+\backslash [G^*(\A_f) \times \mathcal{S}]/K^*$, where $K^*:=K^{\ast}(p^n)\cap K^\ast_1(N) \subset G^*(\A_f)$ is an open compact level subgroup. For $\ell\nmid p$, our choice of tame level ensures that $\mathrm{det}(K^*_{\ell}) = \ZZ_\ell^\times$, whilst $K^*_p = \Gamma^*(p^n)$ has determinant $1+p^n\Z_p$. By strong approximation, the determinant thus induces an isomorphism from the component group
		\[ G^{\ast}(\Q)^+\backslash G^{\ast}(\A_f)/K^* \isorightarrow \hat{\Z}^\times/(1+p^n\hat{\ZZ})^\times \cong(\Z/p^n\Z)^\times.\]
		Thus $\pi_0(X_{\mathfrak c,\Gamma^{\ast}(p^n)}(\C))=(\ZZ/p^n\ZZ)^\times$, which implies $\pi_0(X_{\mathfrak c,\Gamma(p^n)}(\C))=(\OF/p^n\OF)^\times$ by  Lem.~\ref{c:X_G*,Gamma* to X_G*,Gamma is open immersion}.
		Similarly, for $G$ we have $K:=K(p^n)\cap K_1(N) \subset G(\A_f)$ and $G(\Q)^+\backslash G(\A_f)/K$ equals
		\[\GL_2(\O_F)^+\backslash \GL_2(\hat{\ZZ}\otimes \O_F)/K(p^n)\cap K_1(N) \isorightarrow \O_F^{\times,+}\backslash(\hat{\ZZ}\otimes \O_F)^\times/(1+p^n\hat{\ZZ}\otimes \O_F)^\times.\]
		This is the strict ray class group of conductor $p^n$, which is an extension of $\Cl^{+}(\O_F)$ by $U_n$. After taking the fibre of $[\mathfrak c]\in\Cl^{+}(\O_F)$, this equals $U_n$ as desired.
	\end{proof}
	
	\begin{lem}\label{l:torsors-from-G*-to-G-fin-level}
		\leavevmode
		\begin{enumerate}
			\item The map $\beta_2:{\mathcal X_{\Gamma(p^n)}} \to {\mathcal X_{G,\Gamma(p^n)}}$ is a torsor for  $\Delta(p^nN):=\O_F^{\times,+}/(1+p^nN\O_F)^{\times2}$.
			\item The map $\beta_2:{\mathcal X^0_{\Gamma(p^n)}} \to {\mathcal X^0_{G,\Gamma(p^n)}}$ is a torsor for $\Delta_n(N):=(1+p^n\O_F)^{\times,+}/(1+p^nN\O_F)^{\times2}$.
		\end{enumerate}
	\end{lem}
	\begin{proof}
		Setup and notation like in the last proof,  it suffices to see this for  $X_{\Gamma(p^n)}(\C)\to X_{G,\Gamma(p^n)}(\C)$.  We first see from Lem.~\ref{l:connectec component groups} that $\beta_2$ on connected components is the quotient $(\OF/p^n\OF)^\times \to (\OF/p^n\OF)^\times/(\O_F^{\times,+}/(1+p^n\O_F)^{\times,+})$.
		It therefore suffices to prove that on identity components,
		\[ X^0_{\Gamma^*(p^n)}(\C)= X^0_{\Gamma(p^n)}(\C)\to X^0_{G,\Gamma(p^n)}(\C)\]
		is a torsor for the group $\Delta_n(N):=(1+p^n\O_F)^{\times,+}/(1+p^nN\O_F)^{\times,2}$. This map is the cover
		\begin{equation}\label{eq:Galois-cover-G*|H^g->G|H^g}
		\cG^{\ast}\backslash \mathcal H^g\to \cG\backslash \mathcal H^g
		\end{equation}
		where $\cG^{\ast} = K^*\cap G^*(\Q)^+$ where $K^{\ast}=K^{\ast}(p^n)\cap K^\ast_1(N)$, and analogously for $G$. Recall that the kernel for the action of $G(\Q)^+$ on $\mathcal H^g$ are the scalar matrices. Denoting by $\mathrm{P}\cG$ the quotient of $\cG$ by scalar matrices in $\cG$. We note that the only scalar matrix in $\cG^{\ast}$ is the identity. We therefore have a commutative diagram with exact rows and columns
		\begin{center}
			\begin{tikzcd}[row sep = 0.35cm]
			& 1 \arrow[d] \arrow[r]                   & (1+p^nN\OO_F)^\times \arrow[r, "x\mapsto x^2"] \arrow[d] & (1+p^nN\OO_F)^{\times,2} \arrow[d] \arrow[r]           & 1 \\
			1 \arrow[r] & \cG^{\ast} \arrow[r] \arrow[d,equal]  & \cG \arrow[d] \arrow[r, "\det"]                 & {(1+p^n\mathcal O_F)^{\times,+}} \arrow[d] \arrow[r] & 1 \\
			1 \arrow[r] & \cG^{\ast} \arrow[r]  & \mathrm{P}\cG \arrow[r]                         & \Delta_n(N) \arrow[r]                   & 1.
			\end{tikzcd}
		\end{center}
		The bottom row tells us that the Galois group of the cover \eqref{eq:Galois-cover-G*|H^g->G|H^g} is $\Delta_n(N)$, as desired.
	\end{proof}

	\subsection{Torsors over tame level}
	To define overconvergent modular forms, we also need to understand the torsor structures obtained as we vary the wild level.
	
	\begin{definition}Let $m \leq n$.
		\begin{enumerate}
			\item Let $\overline{\Gamma}_{\!0}(p^m,p^n) \subset G(\ZZ/p^n\ZZ)$ and $\overline{\Gamma}_{\!0}^*(p^m,p^n) \subset G^*(\ZZ/p^n\ZZ)$ 
			denote the subgroups of matrices of the form $\smallmatrd{\ast}{\ast}{c}{\ast}$ with $p^m|c$.
			\item Let $Z_n := (1+\n\OO_F)^\times/(1+p^n\n\OO_F)^\times$, embedded diagonally into $\overline{\Gamma}_{\!0}(p^m,p^n)$.
			\item Let $\PG_{\!0}(p^m,p^n) := \overline{\Gamma}_{\!0}(p^m,p^n)/Z_n$ be the quotient group.
		\end{enumerate}
	\end{definition} 
	By Lems.~\ref{l:equivariance of Weil pairing morphism from X_G^*,Gamma(p^n)} and \ref{l:Weil pairing map fits into commut diag comp G* to G}, the level structure action of  $\Gamma(p^n)$ on $\XX_{\Gamma(p^n)}$ restricts to an action of $\Gamma^{\ast}(p^n)$ on $\XX_{\Gamma^*(p^n)}$. We then have:
	\begin{proposition}Via actions on the level structure:
		\begin{enumerate}\setlength{\itemsep}{5pt}
			\item $\XX_{\Gamma^*(p^n)} \to \XX_{\Gamma_{\!\scaleto{0}{3pt}}^*(p^m)}$ is a finite \'etale torsor for the group $\overline{\Gamma}_{\!0}^*(p^m,p^n)$,
			\item $\XX_{\Gamma(p^n)} \to \XX_{\Gamma^{\ast}_{\!\scaleto{0}{3pt}}(p^m)}$  is a finite \'etale torsor for the group $\overline{\Gamma}_{\!0}(p^m,p^n)$, 
			\item $\XX_{G,\Gamma(p^n)} \to \XX_{G,\Gamma_{\!\scaleto{0}{3pt}}(p^m)}$  is a finite \'etale torsor for the group $\PG_{\!0}(p^m,p^n).$
		\end{enumerate}
	\end{proposition}
	\begin{proof}
		Parts (1) and (2) follow from the moduli description. 
		Part (3) follows from Lem.~\ref{l:torsors-from-G*-to-G-fin-level}.1 and Lem.~\ref{l:diagonal-torsor-finite-level} below: the proof only uses the left hand side of~\eqref{dg:def-of-diagonal-E-torsor}, so this is  not circular.
	\end{proof}

	\subsubsection{The diagonal torsor}
	We now have a commutative diagram of towers of finite \'etale torsors
	\begin{equation}\label{dg:def-of-diagonal-E-torsor}
	\begin{tikzcd}[row sep = 0.55cm]
	{\XX_{\Gamma(p^n)}} \arrow[d, "{\Gammabar_{\!\scaleto{0}{3pt}}(p^m, p^n)}"'] \arrow[r, "\Delta(p^nN)"] \arrow[rd, dotted] & {\XX_{G,\Gamma(p^n)}} \arrow[d, "{\PG_{\!\scaleto{0}{3pt}}(p^m, p^n)}"] \\
	\XX_{\Gamma_{\!\scaleto{0}{3pt}}(p^m)} \arrow[r, "\Delta(\n)"'] & \XX_{G,\Gamma_{\!\scaleto{0}{3pt}}(p^m)}
	\end{tikzcd}
	\end{equation}
	Next, we describe the diagonal map in the above diagram, which should be a torsor for some group $\overline{E}(p^m,p^n)$ which can be described as an extension in two ways:
	\[0\to \Gammabar_{\!0}(p^m,p^n)\to \overline{E}(p^m,p^n)\to \Delta(N)\to 0,\]
	\[0\to \Delta(p^nN)\to \overline{E}(p^m,p^n)\to \mathrm \PG_{\!0}(p^m,p^n)\to 0.\]
	It transpires that both extensions are non-split, reflecting our earlier observation that there is no polarisation action by $\Delta(\nn)$ on $\mathcal X_{\Gamma(p^n)}$.
	
	In order to describe $ \overline{E}(p^m,p^n)$ and its action, recall from Lem.~\ref{l:comparing Gamma-action to polarisation action} that for any $\eta\in (1+\n \OO_F)^\times$, the polarisation action of $\eta^{2}$ coincides with the action on the level structure via $\smallmatrd{\eta^{-1}}{0}{0}{\eta^{-1}}$.
	We conclude that the combined action of $\Gammabar_{\!0}(p^m,p^n)\times \OO^{\times,+}_{F}$ is such that the subgroup 
	\[(1+N\OO_F)^{\times}\hookrightarrow \Gammabar_{\!0}(p^m,p^n)\times \OO^{\times,+}_{F},\quad \eta\mapsto \left(\smallmatrd{\eta}{0}{0}{\eta},\eta^{2}\right)\]
	acts trivially on $\mathcal X_{\Gamma(p^n)}$. We therefore obtain an action of the quotient
	\[
	\overline{E}(p,p^n):= \left (\Gammabar_{\!0}(p^m,p^n)\times \OO^{\times,+}_{F} \right)/(1+N\OO_F)^{\times}.
	\]
	We now obtain a short exact sequences as above: first, we clearly have a sequence
	\[0\to \Gammabar_{\!0 }(p^m,p^n)\xrightarrow{\gamma\mapsto (\gamma,1)} \overline{E}(p^m,p^n)\xrightarrow{(\gamma,x)\mapsto x} \Delta(\n)\to 0.\]
	Second, projection to the first factor induces a natural map $\overline{E}(p^m,p^n)\to \PG_{\!0}(p^m,p^n)$. From the snake lemma diagram
	\begin{center}
		\begin{tikzpicture}[scale=0.5]
		
		\matrix[matrix of math nodes,column sep={10pt},row
		sep={30pt,between origins},nodes={asymmetrical rectangle}] (s)
		{
			& & |[name=kb]|0 &|[name=kc]| (1+p^n\n\OO_F)^\times \\
			&|[name=A]|  0  &|[name=B]|  (1+\n\OO_F)^\times  &|[name=C]|  (1+\n\OO_F)^\times  &|[name=01]| 0 \\
			|[name=02]| 0 &|[name=A']| \OO_F^{\times,+} &|[name=B']| \OO_F^{\times,+}\times \Gammabar_{\!0}(p^m,p^n) &|[name=C']| \Gammabar_{\!0}(p^m,p^n) &|[name=03]| 0 \\
			&|[name=ca]|\OO_F^{\times,+} &|[name=cb]| \overline{E}(p^m,p^n)&|[name=cc]| \PG_{\!0}(p^m,p^n)&|[name=04]| 0 \\
		};
		\draw[->]
		
		(kb) edge (kc)
		(kb) edge (B)
		(kc) edge (C)
		(A) edge (B)
		(A) edge (A')
		(C) edge (01)
		(B) edge  (B')
		(C) edge  (C')
		(02) edge (A')
		(A') edge (B')
		(B') edge (C')
		(B') edge (cb)
		(C') edge (cc)
		(ca) edge (cb)
		(cb) edge (cc)
		(C') edge (03)
		(cc) edge (04)
		
		(B') edge[bend right] node [left] {} (A');

		\draw[==]
		(A') edge (ca)
		(B) edge (C)
		
		;
		\draw[->,rounded corners]  (kc) -| node[auto,text=black,pos=0.8]
		{$(x\mapsto x^2)$} ($(01.east)+(.5,0)$) |- ($(B)!.4!(B')$) -|
		($(02.west)+(-.5,0)$) |- (ca);
		\end{tikzpicture}
	\end{center}
	we see the kernel of this map is $\Delta(p^n\n)$ embedded into $\overline{E}(p^m,p^n)$ via $x \mapsto (1,x)$, from which the second exact sequence follows. This shows:
	\begin{lem}\label{l:diagonal-torsor-finite-level}
		The map $\XX_{\Gamma(p^n)}\to \XX_{G,\Gamma_{\!\scaleto{0}{3pt}}(p^m)}$ is an \'etale torsor for the group $\overline{E}(p^m,p^n)$.
	\end{lem}
	\begin{proof}
		By definition in diagram~\eqref{dg:def-of-diagonal-E-torsor}, this map is the composition of an \'etale $\overline{\Gamma}_0(p^m,p^n)$-torsor with an \'etale $\Delta(N)$-torsor.
		It is therefore finite \'etale. From the fact that $\overline{E}(p^m,p^n)$ acts on $\XX_{\Gamma(p^n)}\to \XX_{G,\Gamma_{\!\scaleto{0}{3pt}}(p^m)}$, it is clear that the diagram defining the torsor property commutes. One then verifies that the diagram is Cartesian by decomposing it into smaller Cartesian diagrams induced from the torsor properties of $\XX_{\Gamma(p^n)}\to \XX_{\Gamma_{\!\scaleto{0}{3pt}}(p^m)}$ and $\XX_{\Gamma_{\!\scaleto{0}{3pt}}(p^m)}\to \XX_{G,\Gamma_{\!\scaleto{0}{3pt}}(p^m)}$.
	\end{proof}

	\subsection{Passing to infinite level}
	The following proposition is proven over $\C_p$ by Xu Shen \cite[Thm.~3.3.9]{Xu} in the much greater generality of Shimura varieties of abelian type. In our special case, the version over $\Q_p^\cyc$ is easy to deduce from our preparations. We first note:
	\begin{lem}
		The group $\Delta_\infty(N)=\varprojlim_n \Delta_n(N)$ is finite and $\Delta_\infty(N)=\Delta_n(N)$ for $n\gg 0$.
	\end{lem}
	\begin{proof}
		There is a natural injective map, compatible for varying $n$,
		\[\Delta_n(N)= (1+p^n\O_F)^{\times,+}/(1+p^nN\O_F)^{\times,2}\hookrightarrow \O_F^{\times,+}/(1+N\O_F)^{\times,2}=\Delta(N). \]
		Since $\Delta(N)$ is finite, it follows that $\Delta_n(N)$ stabilises for $n\gg 0$.
	\end{proof}
	\begin{proposition}
		\begin{enumerate}
			\item There exist perfectoid spaces $\XX_{\Gamma(p^\infty)} $ and $\XX_{G,\Gamma(p^\infty)}$ such that
			\[
			\XX_{\Gamma(p^\infty)} \sim \varprojlim \XX_{\Gamma(p^n)} \hspace{12pt} \text{and} \hspace{12pt} \XX_{G,\Gamma(p^\infty)} \sim \varprojlim \XX_{G,\Gamma(p^n)}.
			\]
			\item There also exist perfectoid spaces $\XX^0_{\Gamma(p^\infty)} \sim \varprojlim \XX^0_{\Gamma(p^n)}$ and $\XX^0_{G,\Gamma(p^\infty)} \sim \varprojlim \XX^0_{G,\Gamma(p^n)}$.
			\item There is a natural morphism $\XX_{\Gamma(p^\infty)}\to  \XX_{G,\Gamma(p^\infty)}$ which is a pro-\'etale torsor for the profinite group  $\Delta(p^\infty N)=\varprojlim_n \Delta(p^nN)$.
			\item When restricted to connected components of the identity, it is a finite \'etale torsor $\XX^0_{\Gamma(p^\infty)}\to  \XX^0_{G,\Gamma(p^\infty)}$ for the finite group $\Delta_\infty(N)=\varprojlim_n \Delta_n(N)$.
		\end{enumerate}
	\end{proposition}
	\begin{proof}
		By Theorem~\ref{p:X_Gamma*_0(p^infty) for G*}.4, there is a perfectoid space $\XX_{\Gamma^{\ast}(p^\infty)} \sim \varprojlim \XX_{\Gamma^{\ast}(p^n)}$. By Cor.~\ref{c:X_G*,Gamma* to X_G*,Gamma is open immersion} we have $\XX_{\Gamma(p^n)}=\XX_{\Gamma^{\ast}(p^n)}\times [(\OF/p^n\OF)^\times/(\Z/p^n\Z)^\times]$ on the level of adic spaces, and thus $\XX_{\Gamma(p^\infty)}:=\XX_{\Gamma^{\ast}(p^\infty)}\times \OO_p^\times/\Z_p^\times\sim \varprojlim_n \XX_{\Gamma^{\ast}(p^n)}\times [(\OF/p^n\OF)^\times/(\Z/p^n\Z)^\times] = \varprojlim_n \cX_{\Gamma(p^n)}$. 
		
		For (2) we note that $\XX^0_{\Gamma(p^n)}=\XX^0_{\Gamma^*(p^n)}$, and the existence of the perfectoid space $\XX^0_{\Gamma^*(p^\infty)} \sim \varprojlim \XX^0_{\Gamma^*(p^n)}$ follows from \cite[Cor.~3.3.4]{Xu}.
		From this we obtain the perfectoid space $\cX_{G,\Gamma(p^\infty)}^0$, using Lem.~\ref{l:perf-tilde-limits-and-quotients},  Lem.~\ref{l:torsors-from-G*-to-G-fin-level}.2 and the fact that $\Delta_\infty(N) = \Delta_n(N)$ for $n\gg 0$. This lemma also gives (4). We deduce the second part of (1) from the second part of (2) by \cite[Prop.~3.3.5]{Xu}.
		
		
		Finally, (3) follows from Lem.~\ref{l:torsors-from-G*-to-G-fin-level}.1 by as usual applying the fact that perfectoid tilde-limits commute with fibre products to the diagram defining the torsor property.
	\end{proof}
	\subsubsection{The action of $G(\Q_p)$}\label{s:the action og G(Q_p)}
	Since each $\XX_{\Gamma(p^n)} \rightarrow \XX$ is an \'etale $G(\Z/p^n\Z)=\GL_2(\O_F/p^n\OF)$-torsor, it follows that $\XX_{\Gamma(p^\infty)} \rightarrow \XX$ is a pro-\'etale $G(\Z_p)=\GL_2(\O_p)$-torsor. Here we recall that $\gamma \in \GL_2(\O_p)$ acts by precomposition with $\gamma^{\vee}=\det(\gamma)\gamma^{-1}$ on the level structure $\O_p^2\isorightarrow T_pA^\vee$.

	We shall now for a moment include the dependence on the polarisation ideal $\gothc$ into the notation because,
	as in the Siegel case, the $G(\Z_p)$-action extends naturally to a $G(\Q_p)$-action which, in our case, permutes the spaces $\XX_{\gothc,\Gamma(p^\infty)}$ over the polarisation ideals $\gothc$, as we shall now describe.
	\begin{lem}\label{l:KL-(1.9.1)}
		Let $(A,\iota,\lambda,\mu_N)$ be a HBAV. Let $\mathfrak a\subseteq \O_F$ be an ideal coprime to $N$ and let $D\subseteq A[\mathfrak a]$ be any $\O_F$-submodule scheme. Then there is a unique way to make the isogeny $\varphi:A\to B:=A/D$ into a morphism of HBAVs $(A,\iota,\lambda,\mu_N)\to (B,\iota',\lambda',\mu'_N)$. If $D\cong \oplus_{i=1}^k \O_F/\mathfrak b_i$ and $\mathfrak b:=\mathfrak b_1\cdots \mathfrak b_k$, then $\lambda'$ is the unique $\gothc\mathfrak b$-polarisation making the following diagram commute:
		\begin{equation}
		\begin{tikzcd}[row sep = 0.55cm]
		A\otimes \gothc\arrow[r, "\lambda"]                                                  & A^\vee                        \\
		B\otimes \gothc \mathfrak b \arrow[u, "\varphi^D\otimes \gothc"] \arrow[r, "\lambda'"] & B^\vee\arrow[u, "\varphi^{\vee}"].
		\end{tikzcd}
		\end{equation}
		Here $\varphi^D$ is such that $B\otimes \mathfrak b\xrightarrow{\varphi^D}A\xrightarrow{\varphi}B$ is the natural map $B\otimes \mathfrak{b} \to B\otimes \O_F = B$.
	\end{lem}
	\begin{proof}
		Let $\iota'$ be the quotient action and let $\mu'_N$ be the composition of $\mu_N$ with $A\to A/D$. It remains to construct $\lambda'$ and show that it is a Deligne--Pappas polarisation as described. We refer to \cite[\S1.9]{Kisin-Lai} for the construction if $D$ is of the form $D=\O_F/\mathfrak b_i$. The general case follows by decomposing into a chain of isogenies of this form.
	\end{proof}
	
	Let now $\gamma\in G(\QQ_p)=\GL_2(F_p)$. If $\gamma$ is an element of the form $\smallmatrd{x}{0}{0}{x}$ for some $x\in \O_p$, we will see that the action we now define sends $A\mapsto A/A[x]=A\otimes (x)^{-1}$. For general $\gamma=\smallmatrd{a}{b}{c}{d}$, we may therefore after rescaling assume that $\gamma \in M_2(\O_p) \cap \GL_2(F_p)$.

	We may regard $\gamma$ as acting on $\O_p^2\otimes\gothc^{-1}$. In particular, via $\lambda^{-1}\circ\alpha:\O_p^2\otimes\gothc^{-1}\isorightarrow T_pA$, the matrix $\gamma$ acts $\O_F$-linearly on $T_pA$ and thus on $A[p^n]$ for all $n$. For $n\to \infty$, the kernel $D$ of $\gamma:A[p^n]\to A[p^n]$ stabilises. 
	The automorphism $\gamma$ now sends $(A,\iota,\lambda,\mu_N,\alpha)$ to the HBAV $(B:=A/D,\iota',\lambda',\mu'_N,\alpha')$ from Lem.~\ref{l:KL-(1.9.1)}, where $\alpha':\O_p^2\isorightarrow T_pB^\vee$ is determined as follows:

	\begin{lem}\label{l:action-of-G(Q_p)}
		There is a unique $\alpha':\O_p^2\isorightarrow T_pB^{\vee}$ such that the following diagram commutes:
		\begin{center}
			\begin{tikzcd}[row sep = 0.55cm]
			\O_p^2 \arrow[r, "\alpha"]                             & T_pA^{\vee}                                 \\
			\O_p^2 \arrow[r, "\alpha'"] \arrow[u, "\gamma^{\vee}"] & T_pB^{\vee}. \arrow[u, "\varphi^{\vee}"]
			\end{tikzcd}
		\end{center}
	\end{lem}
	\begin{proof}
		By Lem.~\ref{l:KL-(1.9.1)}, it suffices to show that there is a unique dotted arrow making the diagram
		\begin{center}
			\begin{tikzcd}[row sep = 0.55cm]
			\O_p^2  \arrow[r, "\alpha"]                             & T_pA\otimes \gothc  \arrow[r,"\lambda"]                                  & T_pA^{\vee}                            \\
			\O_p^2 \arrow[r, dotted] \arrow[u, "\gamma^{\vee}"] & T_pB\otimes \gothc\mathfrak b \arrow[u, "\varphi^D"] \arrow[r, "\lambda'"] & T_pB^{\vee} \arrow[u, "\varphi^\vee"]
			\end{tikzcd}
		\end{center}
		commutative.
		Since all arrows become isomorphisms upon inverting $p$, it suffices to show that the cokernels of $\gamma^\vee$ and $\varphi^D$ are identified by $\alpha$. 
		As usual, one sees that the cokernel of $\varphi^D$ is given by $A[\mathfrak b]/D$. Since $\gamma\circ \gamma^\vee =\smallmatrd{\det \gamma}{0}{0}{\det \gamma}$, we have $\coker \gamma^{\vee}=\coker(\det \gamma)/\coker \gamma$. Let now $n$ be large enough that $p^n$ kills $\coker \gamma$, then the $\Tor$-sequence for quotienting by $p^n$ shows that $\alpha$ sends $\coker \gamma$ to $\coker (\gamma:T_pA\to T_pA)=(\ker \gamma:A[p^n]\to A[p^n])=D$. Second, we have $\det(\gamma)\O_p = \mathfrak b\O_p$, and the same $\Tor$-argument shows that $\alpha$ sends $\coker(\det \gamma)$ to $\ker(\det \gamma:A[p^n]\to A[p^n])=A[\mathfrak b]$. This shows that $\alpha$ sends $\coker \gamma^{\vee}$ to $A[\mathfrak b]/D=\coker \varphi^D$.
	\end{proof}
	It is clear from this characterisation of $\gamma$ and the contravariance of $-^{\vee}$ that this is compatible with the multiplication in $G(\Q_p)$, and thus defines an action as desired. We moreover note that for $\gamma\in G(\Z_p)$, we have $\varphi=\id$ and therefore the action thus defined coincides with the action by precomposition with $\gamma^\vee$. Thus the $G(\Q_p)$-action extends the $G(\Z_p)$-action defined earlier.
	
	\begin{lem}
		Let $n \in \ZZ_{\geq 1} \cup \{\infty\}$.
		\begin{enumerate}
			\item $\XX_{\Gamma(p^\infty)} \rightarrow \XX_{\Gamma_{\!\scaleto{0}{3pt}}(p^n)}$ is a pro-\'etale torsor for the group $\Gamma_{\!0}(p^n):=\varprojlim_{m} \Gammabar_{\!0}(p^n,p^m)$.
			\item $\XX_{G,\Gamma(p^\infty)}\rightarrow \XX_{G,\Gamma_{\!\scaleto{0}{3pt}}(p^n)}$ is a pro-\'etale torsor for the group  $\mathrm{P}\Gamma_{\!0}(p^n) := \varprojlim_{m}\PG_{\!0}(p^n,p^m)$.
			\item  $\mathcal X_{\Gamma(p^\infty)}\to \mathcal X_{G,\Gamma_{\!\scaleto{0}{3pt}}(p^n)}$ is a pro-\'etale torsor for the group   $E(p^n):= \varprojlim_{m}\overline{E}(p^n,p^m)$.
		\end{enumerate}
		(Notice we have swapped $n$ and $m$; this is for notational convenience later on).
	\end{lem}
	\begin{proof}
		The diagrams expressing the torsor property are Cartesian since the corresponding Cartesian diagrams at finite level are, and perfectoid tilde-limits commute with fibre products.
	\end{proof}
	In summary, taking the limit over diagram~\eqref{dg:def-of-diagonal-E-torsor}, we thus get a diagram of pro-\'etale torsors
	\begin{equation}\label{eq:diag comparing G*,Gamma*(p^infty) to G,Gamma(p^infty) over tame}
	\begin{tikzcd}[column sep = {3cm,between origins},row sep = 0.7cm]
	\mathcal X_{\Gamma^*(p^\infty)} \arrow[rd, swap, "\GIst"] \ar[r]	&\mathcal X_{\Gamma(p^\infty)} \arrow[d,"\GI"'] \arrow[rd, "{E(p^n)}"] \arrow[r, "\Delta(p^\infty \nn)"] & \XX_{G,\Gamma(p^\infty)} \arrow[d, "\PGI"]   \\
	&\mathcal{X}_{\Gamma_{\!\scaleto{0}{3pt}}(p^n)} \arrow[r,"\Delta(N)"'] & \mathcal{X}_{G,\Gamma_{\!\scaleto{0}{3pt}}(p^n)}
	\end{tikzcd} 
	\end{equation}
	where $\Delta(p^\infty N):= \varprojlim_{n} \Delta(p^n N)$.
	
	\subsubsection{Comparison to $\XX_{\Gamma^*(p^\infty)}$}
	Taking limits, the Weil pairing morphism \eqref{m:Weil-pairing-morphisms} induces maps 
	\[
	\varprojlim \mathbf{e}_{n,\beta} =: \mathbf{e}_\beta : \XX_{\Gamma(p^\infty)} \rightarrow \O_p^\times\qquad \text{ and }\qquad 
	\varprojlim \mathbf{e}_{n} =:\mathbf{e}: \XX_{G,\Gamma(p^\infty)} \rightarrow \gothc\d^{-1}(1)^\times,
	\]
	where the targets refer to the associated profinite perfectoid groups.  Moreover, in the limit we obtain a level action of $G(\Zp)$. Through this we define a level structure action of $\eta \in \OO_p^\times$ acting by $\smallmatrd{\eta}{0}{0}{1}$. In the limit, Lem.~\ref{l:equivariance of Weil pairing morphism from X_G^*,Gamma(p^n)} and Cor.~\ref{c:X_G*,Gamma* to X_G*,Gamma is open immersion} give the following two lemmas.
	
	\begin{lemma}\label{l:Gamma*(p^infty) x O_p^x->Gamma(p^infty) is torsor}
		The anti-diagonal action of $\OO_p^\times$ restricts to a $\Zp^\times$-torsor
		\[
		\XX_{\Gamma^*(p^\infty)} \times \OO_p^\times \longrightarrow \XX_{\Gamma(p^\infty)}.
		\]
	\end{lemma}
	
	\begin{lemma}\label{l:equivariance of Weil pairing morphism from X_G,Gamma(p^infty)}
		For any $(\gamma,x) \in E(p)$, the following diagram commutes:
		\begin{center}
			\begin{tikzcd}[row sep = 0.55cm]
			\XX_{\Gamma(p^\infty)} \arrow[d, "{(\gamma,x)}"] \arrow[r, "\mathbf e_\beta"] & \O_p^\times \arrow[d, "\det(\gamma)x^{-1}"] \\
			\XX_{\Gamma(p^\infty)} \arrow[r, "\mathbf e_{\beta}"] & \O_p^\times.
			\end{tikzcd}
		\end{center}
	\end{lemma} 
	
	\begin{rmrk}
		We note that if we take the diagram from Lem.~\ref{lem: et tor} and fiber it with (\ref{eq:diag comparing G*,Gamma*(p^infty) to G,Gamma(p^infty) over tame}) we can see that everything restricts to the anticanonical locus, i.e., in the statements above, we can replace $\	\XX_{\Gamma(p^\infty)}, \XX_{G,\Gamma(p^\infty)}$, etc with  $\	\XX_{\Gamma(p^\infty)}(\e)_a, \XX_{G,\Gamma(p^\infty)}(\e)_a$, etc.
	\end{rmrk}
	
	\subsection{Hodge--Tate period maps}
	\begin{lem}\label{l:compatibility of Hodge--Tate period map for G and G*}
		There exist Hodge--Tate period maps making the following diagram commute:
		\begin{center}
			\begin{tikzcd}[row sep = 0.55cm]
			\XX_{\Gamma^*(p^\infty)} \arrow[r] \arrow[d, "\pi_{\HT}"] & \XX_{\Gamma(p^\infty)} \arrow[r] \arrow[d, "\pi_{\HT}"] & \XX_{G,\Gamma(p^\infty)} \arrow[d, "\pi_{\HT}"] \\
			\Res_{\OF|\ZZ}\PP^1 \arrow[r] & \Res_{\OF|\ZZ}\PP^1 \arrow[r] & \Res_{\OF|\ZZ}\PP^1.
			\end{tikzcd}
		\end{center}
	\end{lem}
	The map $\mathcal X_{\Gamma(p^\infty)}\to \Res_{\OF|\ZZ}\PP^1$ is invariant for the polarisation action of $\O_F^{\times,+}$ on $\mathcal X_{\Gamma(p^\infty)}$.
	\begin{proof}
		
		We may define a map $\pi_{\HT}$ on $\mathcal X_{\Gamma^{\ast}(p^\infty)}\times \mathcal O_p^\times$ by projection from the first factor. This map is $\Z_p^\times$-invariant for the antidiagonal action since the $\Z_p^\times$-action on $\Gamma^{\ast}(p^\infty)$ just amounts to a rescaling of the basis vectors, which leaves the relative position of the kernel of the Hodge--Tate morphism invariant. Using the $\Z_p^\times$-torsor property from Lem.~\ref{l:Gamma*(p^infty) x O_p^x->Gamma(p^infty) is torsor} in the pro-\'etale site, we conclude that $\pi_{\HT}$ descends to the second vertical map in the following diagram.
		\begin{center}
			\begin{tikzcd}[row sep = 0.55cm]
			{\mathcal X_{\Gamma^{\ast}(p^\infty)}\times \mathcal O_p^\times} \arrow[r, "\Z_p^\times"] \arrow[d] & {\mathcal X_{\Gamma(p^\infty)}} \arrow[r, "\Delta(p^\infty N)"] \arrow[d, dotted] & {\mathcal X_{G,\Gamma(p^\infty)}} \arrow[d, dotted] \\
			\Res_{\OF|\ZZ}\PP^1 \arrow[r, equal] & \Res_{\OF|\ZZ}\PP^1 \arrow[r, equal] & \Res_{\OF|\ZZ}\PP^1
			\end{tikzcd}
			
		\end{center}
		Similarly, the polarisation action of $\Delta(p^\infty N)$ clearly leaves $\pi_{\HT}$ invariant since it does not change the level structure. The same descent argument gives the third vertical map.
		
		For the last statement, we note that $\mathcal X_{\Gamma^{\ast}(p^\infty)}\to \Res_{\OF|\ZZ}\PP^1$ is invariant under the polarisation action by $\Z_p^\times\cap \O_F^{\times,+}$ since the polarisation does not feature in the Hodge--Tate sequence. The result then follows from the construction via the $\Z_p^\times$-torsor $\mathcal X_{\Gamma^{\ast}(p^\infty)}\times \mathcal O_p^\times\to \mathcal X_{\Gamma(p^\infty)}$.
	\end{proof}

	\section{Arithmetic overconvergent Hilbert modular forms}\label{sec:G forms}
	In \S\ref{sec:G vars}, we exhibited various pro-\'etale torsors of perfectoid Hilbert modular varieties over $\XX_G$. In this section, we use these to define overconvergent modular forms for $G$ and compare them to those for $G^*$. We do this by defining four different sheaves, each using a different torsor, and then give a chain of comparisons that relate them. This will show that our modular forms for $G$ agree with the previous sheaves defined in \cite{AIP3} by descending the sheaves for $G^*$.

	
	Let $\k: \U \to \W$ be a bounded smooth weight. Recall from \ref{nota: weight conv} that by composition with $\rho:\W\to \W^*$ this also defines a bounded smooth weight for $G^*$ which by Prop.~\ref{an cnt hmfs} we may interpret as a morphism of adic spaces $\k : B_r(\OO_p^\times:1)\times \U \to \hat{\G}_m$. It follows from Lem.~\ref{l:compatibility of Hodge--Tate period map for G and G*} that Defn.~\ref{d:cocycle-kappa(cz+d)} goes through also for the infinite level spaces $\XX_{\Gamma^{\ast}(p^\infty)}(\epsilon)_a$ and $\XX_{G,\Gamma(p^\infty)}(\epsilon)_a$. In particular, this gives us invertible functions $\kappa(c\mathfrak{z}+d)$ on each of these spaces. We denote them by the same letter since by definition they are compatible via pull-back along $\XX_{\Gamma^{\ast}(p^\infty)}(\epsilon)_a\to \XX_{\Gamma(p^\infty)}(\epsilon)_a\to \XX_{G,\Gamma(p^\infty)}(\epsilon)_a$.
	
	\begin{definition}\label{defn: 4 sheaves for the price of 2} Let $n \in \ZZ_{\geq 1} \cup \{\infty\}$. 
		In the following, all infinite level sheaves are tacitly pushed forward to the indicated finite level base.
		\begin{enumerate}
			\item  The sheaf $\omega^{\k,+}_{G^{\ast},(G^*,\Gamma^*,n)}$ of \emph{integral modular forms for $G^*$ via $\XX_{\U,\Gamma^*(p^\infty)}$}  is 
			\begin{alignat*}{5}
			&\big\{f \in \mathcal O^+_{\mathcal X_{\U,\Gamma^*(p^\infty)}(\e)_a}&&\big|& \gamma^{\ast}f &= \kU^{-1}(c\mathfrak z+d)f  &&\text{ for all }\gamma=\smallmatrd{a}{b}{c}{d}\in \Gamma_{\!0}^{\ast}(p^n) &&\big\}.
			\intertext{	\item  The sheaf $\omega^{\k,+}_{G^{\ast},(G^{\ast},\Gamma,n)}$ of \emph{integral modular forms for $G^*$ via $\XX_{\U,\Gamma(p^\infty)}$} is}
			&\big\{f \in \mathcal O^+_{\mathcal X_{\U,\Gamma(p^\infty)}(\e)_a}  &&\big|& \gamma^{\ast}f &= \kU^{-1}(c\mathfrak z+d)f  &&\text{ for all }\gamma=\smallmatrd{a}{b}{c}{d}\in \Gamma_{\!0}(p^n)&&\big \}.
			\intertext{	\item  The sheaf $\omega^{\k,+}_{G,(G^{\ast},\Gamma,n)}$ of \emph{integral modular forms for $G$ via $\XX_{\U,\Gamma(p^\infty)}$} is}
			&\big\{f \in \mathcal O^+_{\mathcal X_{\U,\Gamma(p^\infty)}(\e)_a}&&\big |& (\gamma,x)^{\ast}f &= \kU^{-1}(c\mathfrak z+d)\wU(x)f  &&\text{ for all }(\gamma,x)\in E(p^n) &&\big\}.
			\intertext{\item  The sheaf $\omega^{\k,+}_{G,(G,\Gamma,n)}$ of \emph{integral modular forms for $G$ via $\XX_{G,\U,\Gamma(p^\infty)}$} is}
			&\big\{f \in \mathcal O^+_{\mathcal X_{G,\U,\Gamma(p^\infty)}(\e)_a}&&\big|& \gamma^{\ast}f &= \kU^{-1}(c\mathfrak z+d)\wU(\det \gamma)f \quad &&\text{ for all }\gamma\in \mathrm P\Gamma_{\!0}(p^n) &&\big\}.
			\end{alignat*}
		\end{enumerate}
	\end{definition}
	We note that (1) is the sheaf of modular forms for $G^{\ast}$ from Defn.~\ref{def:bundle for G*}, denoted there by $\omega_{n}^\kappa$. We have switched to the notation above for clearer comparison to (2), (3) and (4). 
	The goal of this secion is to relate these sheaves. More precisely:
	\begin{itemize}
		\item[(i)] in Lem.~\ref{lem G* sheaves} below, we will see that the sheaves (1) and (2) are isomorphic;
		\item[(ii)] in Lem.~\ref{lem:sheaf descent}, we will see that sheaf (3) is obtained from (2) by taking $\Delta(N)$-invariants;
		\item[(iii)]  in Lem.~\ref{prop G sheaves}, we will see that the sheaves (3) and (4) are isomorphic. 
	\end{itemize}
	Before we start giving comparison maps, however, we need to check:
	\begin{lemma}\label{lem:well-defined}
		The conditions in (3) and (4) above are well-defined; that is, they do not depend on the choice of representatives $(\gamma,x)$ or $\gamma$ respectively. 
	\end{lemma}
	\begin{proof}
		The relation of $\wU$ and $\kU$ given in Def. \ref{nota: weight conv} implies that
		\begin{equation}\label{eq:rel between kappa and w on global units}
		\kU^{-1}(\eta)\wU(\eta^{2}) = \rU\circ N_{F/\Q}(\eta) = 1 \text{ for all }\eta\in \OO_F^{\times,+}.
		\end{equation}
		For (3), this implies that for any $(\gamma,x)\in \Gamma_{\!0}(p^n)\times \OO_F^{\times,+}$,  the factor $\kU^{-1}(c\mathfrak z+d)\wU(x)$ only depends on the image of $(\gamma,x)$ in $E(p^n)$; indeed, for any $\eta\in \OO_F^{\times,+}$ the translate $(\gamma \smallmatrd{\eta}{0}{0}{\eta},x\eta^{2})$ has the same associated factor
		\[\kU^{-1}(c\eta\mathfrak z+d\eta)\wU(x\eta^2)=\kU^{-1}(\eta)\wU(\eta^{2})\kU^{-1}(c\mathfrak z+d)\wU(x)=\kU^{-1}(c\mathfrak z+d)\wU(x).\]
		For (4), we similarly note that for any $\eta\in (1+N\OO_F)^{\times,+}$, setting $\gamma=\smallmatrd{\eta}{0}{0}{\eta}$ results in the factor $\kU^{-1}(\eta)\wU(\eta^{2})=1.$
		By continuity, the same is true for $\gamma$ in the topological closure $Z_\infty$ of $(1+N\OO_F)^{\times,+}$ in $\OO_{p}^{\times,+}$. As $\mathrm P\Gamma_{\!0}(p^n)=\varprojlim_m \PG_{\!0}(p^n,p^m)=\varprojlim_m \mathrm \Gammabar_{\!0}(p^n,p^m)/Z_m= \mathrm \Gamma_{\!0}(p^n)/Z_\infty$ , this shows that for any $\gamma\in \Gamma_{\!0}(p^n)$, the factor $\kU^{-1}(c\mathfrak z+d)\wU(\det \gamma)$  only depends on the image of $\gamma$ in $\mathrm P\Gamma_{\!0}(p^n)$. This shows that the condition in (4) is well-defined.
	\end{proof}
	
	\begin{lem}\label{lem G* sheaves} Let $\k: \U \to \W^*$ be a smooth weight. 
		The natural morphism of torsors $\mathcal X_{\U,\Gamma^{\ast}(p^\infty)}(\epsilon)_a\to \mathcal X_{\U,\Gamma(p^\infty)}(\epsilon)_a$ over $\mathcal X_{\U,\Gamma_{\!\scaleto{0}{3pt}}^{\ast}(p^n)}(\epsilon)_a$ induces a natural isomorphism
		\[
		\omega^{\kappa,+}_{G^{\ast},(G^*,\Gamma^*,n)} \cong \omega^{\kappa,+}_{G^{\ast},(G^*,\Gamma,n)}.
		\]
		In particular, the definition of forms for $G^*$ is independent of the choice of $\Gamma$ or $\Gamma^*$.
	\end{lem}
	\begin{proof}
		By Lem.~\ref{l:compatibility of Hodge--Tate period map for G and G*}, for any $\gamma \in \Gamma_{\!0}^{\ast}(p^n)$, the function $\kappa^{-1}(c\mathfrak z+d)$ on $\mathcal X_{\U,\Gamma(p^\infty)}(\epsilon)_a$ pulls back to $\kappa^{-1}(c\mathfrak z+d)$ on $\mathcal X_{\U,\Gamma^{\ast}(p^\infty)}(\epsilon)_a$.
		Since the map $\mathcal X_{\U,\Gamma^{\ast}(p^\infty)}(\epsilon)_a\to \mathcal X_{\U,\Gamma(p^\infty)}(\epsilon)_a$ is equivariant under $\Gamma_{\!0}^{\ast}(p^n)\to \Gamma_{\!0}(p^n)$, it follows that the associated map $\OO^+_{\mathcal X_{\U,\Gamma(p^\infty)}(\epsilon)_a}\to \OO^+_{\mathcal X_{\U,\Gamma^{\ast}(p^\infty)}(\epsilon)_a}$ restricts to
		\begin{equation}\label{eq:G* one map}
		\omega^{\kappa,+}_{G^{\ast},(G^*,\Gamma,n)}\to \omega^{\kappa,+}_{G^{\ast},(G^*,\Gamma^*,n)}.
		\end{equation}
		To construct an inverse, let $f\in \omega^{\kappa,+}_{G^{\ast},(G^*,\Gamma^*,n)}$. Note that $f$ is invariant under the action of $\ZZ_p^\times$, since for any $\epsilon \in \ZZ_p^\times$, we have $\gamma_{\epsilon}=\smallmatrd{\epsilon}{0}{0}{1}$ acting via $\gamma_{\epsilon}^{\ast}f=\kappa^{-1}(1)f=f.$
		Consider now
		\[\mathcal X_{\U,\Gamma^{\ast}(p^\infty)}(\epsilon)_a  \longleftarrow\mathcal X_{\U,\Gamma^{\ast}(p^\infty)}(\epsilon)_a\times\OO_p^\times \longrightarrow \mathcal X_{\U,\Gamma(p^\infty)}(\epsilon)_a,\]
		where the right hand morphism is the $\ZZ_p^\times$-torsor from Lem.~\ref{l:Gamma*(p^infty) x O_p^x->Gamma(p^infty) is torsor} and the left hand morphism is simply the projection. Since $f$ is $\Z_p^\times$-invariant, the pullback $f'$ of $f$ under the left map is invariant under the antidiagonal $\ZZ_p^\times$-action. Since the morphism on the right is a $\ZZ_p^\times$-torsor of affinoid perfectoid spaces, this implies that $f'$ descends uniquely down the right to a function $f''$ on $\mathcal X_{\U,\Gamma(p^\infty)}(\epsilon)_a$. We claim that $f''\in \omega^{\kappa}_{G^{\ast},(G^*,\Gamma^*,n)},$ that is, it transforms as required under $\Gamma_{\!0}(p^n)$. This group is generated by $\smallmatrd{\OO_p^\times}{0}{0}{1}$ and $\Gamma_{\!0}^{\ast}(p^n)$, so it suffices to check that
		\begin{equation}\label{eq:Op invariant}
		\gamma_{\epsilon}^{\ast}f''=\kappa^{-1}(0\mathfrak z+1)f'' =f''\text{ for all } \e \in \OO_p^\times, 
		\end{equation}
		\begin{equation}\label{eq:Gamma invariant}
		\gamma^{\ast}f''=\kappa^{-1}(c\mathfrak z+d)f'' \text{ for all }\gamma \in \Gamma_{\!0}^{\ast}(p^n).
		\end{equation}
		As $f''$ came from pullback from the left, the function $f'$ is invariant for the $\OO_p^\times$-action, and thus the same is true for $f''$, where $\epsilon \in \OO_p^\times$ acts via $\gamma_{\epsilon}=\smallmatrd{\epsilon}{0}{0}{1}$. This gives \eqref{eq:Op invariant}. 
		Also, since $f$ was a modular form and the above maps are all equivariant for the natural $\Gamma^{\ast}_0(p^n)$-action, we also have \eqref{eq:Gamma invariant}.
		Thus $f''\in \omega^{\kappa}_{G^{\ast},(G^*,\Gamma^*,n)},$ and we see directly that $f \mapsto f''$ gives an inverse to \eqref{eq:G* one map}.
	\end{proof}
	
	We now use \'etale descent along $\pi:\XX_{\U}(\epsilon)\to \XX_{G,\U}(\epsilon)$ to relate the sheaves for $G^*$ and $G$. Explicitly, this can be done by endowing $\pi_{\ast}\omega^{\kappa,+}_{G^{\ast},(G^*,\Gamma,n)}$ with a $\Delta(N)$-action given by a twist of the $\O_F^{\times,+}$-action on $\mathcal X_{\U,\Gamma(p^\infty)}(\epsilon)_a\to \mathcal X_{G,\U,\Gamma(p^\infty)}(\epsilon)_a$, as follows:
	
	\begin{defn}\label{d:polarisation-action-on-mf}
		Let $\k:\U \to \W$ be a smooth bounded weight. We define a $\wU$-twisted polarisation left-action of $\OO_F^{\times,+}$ on $\pi_{\ast}q_*\mathcal O^+_{\mathcal X_{\U,\Gamma(p^\infty)}(\e)_a}$ by letting $\varepsilon \in \OO_F^{\times,+}$ act as
		\[\varepsilon\cdot_{\wU}f := \wU(\varepsilon) \cdot (\varepsilon^{-1})^{\ast}f \]
		where $\wU$ is as in \ref{nota: weight conv} and the action on the right side is the polarisation action on $\mathcal X_{\U,\Gamma(p^\infty)}(\epsilon)_a$. 
	\end{defn} 
	\begin{lem}\label{lem: G forms from AIP}
		The $\wU$-twisted action of $\OO_F^{\times,+}$ restricts to $\pi_{\ast}\omega^{\kappa,+}_{G^{\ast},(G^*,\Gamma,n)}$, where it factors through an action of $\Delta(\nn)$. Furthermore, this action coincides with the action defined by \cite{AIP3}.
	\end{lem}
	\begin{proof}
		It is clear from the moduli description that the polarisation action of $\OO_F^{\times,+}$ and the level structure action of $\Gamma_0(p^n)$ commute on $\mathcal X_{\U,\Gamma(p^\infty)}(\epsilon)_a$. By Lemma~\ref{l:compatibility of Hodge--Tate period map for G and G*}, the action moreover leaves $\pi_{\HT}$ and thus $\mathfrak z$ invariant. We therefore have for any $\varepsilon \in \O_F^{\times,+}$ and any $f\in \pi_{\ast}\omega^{\kappa,+}_{G^{\ast},(G^*,\Gamma,n)}$
		\[\gamma^{\ast}(\varepsilon\cdot_{\wU}f) =w_{\kappa}(\varepsilon)\varepsilon^{-1\ast}\gamma^{\ast}f=w_{\kappa}(\varepsilon)\varepsilon^{-1\ast}(\kappa^{-1}(c\mathfrak z+d)f)=\kappa^{-1}(c\mathfrak z+d)\varepsilon\cdot_{\wU}f .\]
		This shows that the action restricts. Next, if $\eta \in \OO_F^{\times,+}$, then $\eta^{2}$ acts on $f\in \pi_{\ast}\omega^{\kappa,+}_{G^{\ast},(G^*,\Gamma,n)}$ via
		\[ 
		\eta^2\cdot_{\wU}f=\wU(\eta^{2}) \cdot (\eta^{-2})^{\ast}f \stackrel{\text{Lem.~\ref{l:comparing Gamma-action to polarisation action}}}{=\joinrel=} \wU(\eta^{2}) \cdot \smallmatrd{\eta}{0}{0}{\eta}^{\ast}f= \wU(\eta^{2}) \kappa^{-1}(\eta)f\stackrel{\text{\eqref{eq:rel between kappa and w on global units}}}{=}f,
		\]
		and in particular the subgroup $(1+\n\OO_F)^{\times2}$ acts trivially; thus the action factors through $\Delta(\nn)$ as desired. The last statement follows from Lem.~\ref{lem G* sheaves} and Thm.~\ref{thm:comparison hilbert} since Defn.~\ref{d:polarisation-action-on-mf} and the definition in \cite[Section 4.1]{AIP} match up: here we use that the polarisation action commutes with $\mathfrak s$ since it leaves the wild level $\alpha$ and the Hodge--Tate morphism unchanged.
	\end{proof}
	
	\begin{lem}\label{lem:sheaf descent}
		We have an equality of subsheaves of $\OO^+_{\XX_{\U,\Gamma(p^\infty)}(\epsilon)_a}$ on $\XX_{G,\U,\Gamma_{\!\scaleto{0}{3pt}}(p)}(\epsilon)_a$
		\[ \omega^{\kappa,+}_{G,(G^{\ast},\Gamma,n)} =  (\pi_{\ast}\omega^{\kappa,+}_{G^*,(G^{\ast},\Gamma,n)} )^{\Delta(\nn)}. \]
	\end{lem}
	\begin{proof}
		Let $f\in \omega^{\kappa,+}_{G,(G^{\ast},\Gamma,n)}$. By Lem.~\ref{lem:well-defined}, the action of $E(p^n)$ that defines the modular forms property for $f$ is induced by an action of $\Gamma_{\!0}(p^n)\times \OO_F^{\times,+}$; from this, it is clear that 
		\[ 
		\gamma^{\ast}f=(\gamma,1)^{\ast}f= \kappa^{-1}(c\mathfrak z+d)f \text{ for all }\gamma=\smallmatrd{a}{b}{c}{d}\in \Gamma_{\!0}(p^n),
		\]
		and thus $f\in \pi_{\ast}\omega^{\kappa,+}_{G^*,(G^{\ast},\Gamma,n)}$. To see that it is also $\Delta(\nn)$-equivariant, it suffices to note that 
		\[ x^{\ast}f=(1,x)^{\ast}f= \wU(x)f \text{ for all }x\in \OO_F^{\times,+},\]
		which implies that for the $\wU$-twisted action of $x$ we have $x\cdot_{\wU} f=\wU(x)(x^{-1})^{\ast}f=f$. Thus $f\in  (\pi_{\ast}\omega^{\kappa,+}_{G^*,(G^{\ast},\Gamma,n)})^{\Delta(\nn)}$ as desired. 
		The converse follows by reversing the above calculations.
	\end{proof}
	
	To compare $\omega^{\kappa,+}_{G,(G^*,\Gamma,n)}$ to $\omega^{\kappa,+}_{G,(G,\Gamma,n)}$, pulling back functions along $\XX_{\U,\Gamma(p^\infty)}\to \XX_{G,\U,\Gamma(p^\infty)}$ is not enough: We also need an additional twist that changes the factor of $\kU^{-1}(c\mathfrak z+d)\wU(x)$ in the definition of the former into the factor of $\kappa^{-1}(c\mathfrak z+d)\wU(\det \gamma)$ used in the latter. To describe this, let $\k: \U \to \W \to \W^*$ be a smooth weight and recall from the discussion before Lem.~\ref{l:equivariance of Weil pairing morphism from X_G,Gamma(p^infty)} the Weil pairing morphism $\mathbf e_{\beta}:\XX_{\U,\Gamma(p^\infty)}\to \O_p^\times$. We use this to define a composite morphism
	\[\wU(\mathbf e_{\beta}):\XX_{\U,\Gamma(p^\infty)}\xrightarrow{\mathbf e_{\beta}} \OO_p^\times \xrightarrow{\wU}\hat{\G}_m.\] Restricting to the subspace $\XX_{\U,\Gamma(p^\infty)}(\epsilon)_a$, the universal property of $\hat{G}_m$ associates to $\wU(\mathbf e_{\beta})$ a function in $\O^+(\XX_{\U,\Gamma(p^\infty)}(\epsilon)_a)^\times$ that we might reasonably also denote by $\wU(\mathbf e_{\beta})$. By Lem.~\ref{l:equivariance of Weil pairing morphism from X_G,Gamma(p^infty)},
	\begin{equation}\label{eq:E-action on w(e)}
	(\gamma,x)^{\ast}\wU(\mathbf e_{\beta})=\wU(x^{-1})\wU(\det \gamma)\wU(\mathbf e_{\beta}) \quad\text{ for any }(\gamma,x)\in E(p^n).
	\end{equation}
	
	\begin{lem}\label{prop G sheaves}
		Let $\pi_{\infty}$ be the natural morphism of torsors
		\begin{center}
			\begin{tikzcd}[row sep = 0.55cm]
			{\XX_{\U,\Gamma(p^\infty)}(\epsilon)_a} \arrow[d,"{E(p^n)}"'] \arrow[r,"\pi_{\infty}"] & {\XX_{G,\U,\Gamma(p^\infty)}(\epsilon)_a} \arrow[d,"{\mathrm P\Gamma_{\!\scaleto{0}{3pt}}(p^n)}"] \\
			\XX_{G,\U,\Gamma_{\!\scaleto{0}{3pt}}(p^n)}(\epsilon)_a \arrow[r,equal] & \XX_{G,\U,\Gamma_{\!\scaleto{0}{3pt}}(p^n)}(\epsilon)_a.
			\end{tikzcd}
		\end{center}
		Then sending $f\mapsto \wU^{-1}(\mathbf e_{\beta})\cdot \pi_{\infty}^{\ast}f$ defines an isomorphism of $\O^+$-modules on $\XX_{G,\U,\Gamma_{\!\scaleto{0}{3pt}}(p^n)}(\epsilon)_a$
		\[\omega^{\kappa,+}_{G,(G,\Gamma,n)}\xrightarrow{\sim}\omega^{\kappa,+}_{G,(G^{\ast},\Gamma,n)}.\]
	\end{lem}
	\begin{proof}
		Let $f\in \omega^{\kappa,+}_{G,(G,\Gamma,n)}$. For any $(\gamma,x)\in E(p^n)$, we have
		\begin{align*}
		(\gamma,x)^{\ast}(\wU^{-1}(\mathbf e_{\beta})\pi_{\infty}^{\ast}f)&=(\gamma,x)^{\ast}\wU^{-1}(\mathbf e_{\beta})\cdot (\gamma,x)^{\ast}\pi_{\infty}^{\ast}f\cr
		&=\wU(x)\wU^{-1}(\det \gamma)\wU^{-1}(\mathbf e_{\beta})\pi_{\infty}^{\ast}\gamma^{\ast}f\cr
		&=\wU(x)\wU^{-1}(\det \gamma)\wU^{-1}(\mathbf e_{\beta})\pi_{\infty}^{\ast}(\kappa^{-1}(c\mathfrak z+d)\wU(\det \gamma)f)\\
		&=\kappa^{-1}(c\mathfrak z+d)\wU(x)\wU^{-1}(\mathbf e_{\beta})\pi_{\infty}^{\ast}f.
		\end{align*}
		In the second step we have used equation~\eqref{eq:E-action on w(e)} and the fact that $\pi_{\infty}$ is equivariant with respect to the projection $E(p^n)\to \mathrm{P}\Gamma_0(p^n)$, $(\gamma,x)\mapsto \gamma$, in the third step we use $f\in \omega^{\kappa,+}_{G,(G,\Gamma,n)}$, and in the last step we use that by Lem.~\ref{l:compatibility of Hodge--Tate period map for G and G*}, $\pi_{\infty}^{\ast}(\kappa^{-1}(c\mathfrak z+d))=\kappa^{-1}(c\mathfrak z+d)$. This identity shows that indeed $\wU^{-1}(\mathbf e_{\beta})\pi_{\infty}^{\ast}f\in \omega^{\kappa}_{G,(G^{\ast},\Gamma,n)}$, so the map in the statement of the proposition is well-defined.
		
		To see that the map is an isomorphism, take now $f\in \omega^{\kappa,+}_{G,(G^{\ast},\Gamma,n)}$ and consider the function $g=\wU(\mathbf e_{\beta})f$ on $\XX_{\U,\Gamma(p^\infty)}(\epsilon)_a$. We claim that $g$ descends uniquely to a function on $\XX_{G,\U,\Gamma(p^\infty)}(\epsilon)_a$. To see this, since $\pi_{\infty}:\XX_{\U,\Gamma(p^\infty)}(\epsilon)_a\to \XX_{G,\U,\Gamma(p^\infty)}(\epsilon)_a$ is a perfectoid $\Delta(p^\infty \nn)$-torsor, it suffices to show that $g$ is $\Delta(p^\infty \nn)$-invariant. It suffices by continuity to show that it is invariant for the dense subgroup $\OO_F^{\times,+} \hookrightarrow \Delta(p^\infty \nn)$.  For this we calculate that for any $x\in \OO_F^{\times,+}$ we have
		\[ x^{\ast}g = (1,x)^{\ast}g = (1,x)^{\ast}(\wU(\mathbf e_{\beta}))(1,x)^{\ast}f=\wU(x)^{-1}\wU(\mathbf e_{\beta})\kU^{-1}(1)\wU(x)f=\wU(\mathbf e_{\beta})f=g, \]
		where we have again used equation~\eqref{eq:E-action on w(e)}. This shows that $g$ indeed arises as pullback of a unique function $h=h(f)$ on $\XX_{G,\U,\Gamma(p^\infty)}(\epsilon)_a$ with $\pi_{\infty}^{\ast}h=g$. We claim that $h\in \omega^{\kappa,+}_{G,(G,\Gamma,n)}$. For this we have to check by definition that $\gamma^{\ast}h$ equals $\kU^{-1}(c\mathfrak z+d)\wU(\det \gamma)h$.  The morphism of sheaves $\pi_{\infty}^{\ast}:\OO^+_{\XX_{\U,\Gamma(p^\infty)}(\epsilon)_a}\to \OO^+_{\XX_{\U,\Gamma_{\!\scaleto{0}{3pt}}(p^\infty)}(\epsilon)_a}$ is the pullback of a pro-\'etale map, hence injective; thus, as it is also $\Gamma_{\!0}(p^n)$-equivariant, it suffices to check this property on $\pi_{\infty}^{\ast}h=g$ with respect to the action of $\Gamma_{\!0}(p^n)\to E(p^n)$. But by definition of $g$ as $\wU(\mathbf e_{\beta})f$, we have for any $\gamma\in\Gamma_{\!0}(p^n)$ that
		\[\gamma^{\ast}g=(\gamma,1)^{\ast}(\wU(\mathbf e_{\beta}))(\gamma,1)^{\ast}f=\wU(\det(\gamma))\wU(\mathbf e_{\beta})\kU^{-1}(c\mathfrak z+d)f=\kU^{-1}(c\mathfrak z+d)\wU(\det(\gamma))g.\]
		This shows that $g$ transforms under the action of $\Gamma_{\!0}(p^n)$ as desired, and thus so does $h$. This shows $h\in \omega^{\kappa,+}_{G,(G,\Gamma,n)}$ as desired. Moreover, by construction we have $f=\wU^{-1}(\mathbf e_{\beta})\pi_{\infty}^{\ast}h$.
		
		This shows that the map in the proposition is injective and surjective, thus an isomorphism.
	\end{proof}
	\begin{defn}\label{d:mf-G-via-X_G}
		Let $\k:\U \to \W$ be a bounded smooth weight and $n \in \ZZ_{\geq 0} \cup \{\infty\}$. For $n\neq 0$, let $\omega_{G,\gothc,n}^{\kappa,+}$ be the sheaf $\omega^{\kappa,+}_{G,(G,\Gamma,n)}$ from  Defn.~\ref{defn: 4 sheaves for the price of 2} (4), noting that we have reintroduced our fixed $\gothc$ to the notation. As before, we define the rational version $\omega^{\kappa}_{G,\gothc,n}$ by replacing $\OO^+_{\cX}$ with $\OO_{\cX}$, or equivalently by inverting $p$. We use the Atkin--Lehner isomorphism to define `$n=0$' versions 
		\[
		\omega^\kappa_{G,\gothc} = \omega^\kappa_{G,\gothc,0} := \AL_*\omega^\kappa_{G,\gothc,1}, \hspace{12pt} \omega^{\kappa,+}_{G,\gothc} = \omega^{\kappa,+}_{G,\gothc,0} := \AL_* \omega^{\kappa,+}_{G,\gothc,1}
		\]
		which are sheaves on $\cX_{G,\gothc,\cU}(\e)$. 
		For any $n \in \ZZ_{\geq 0} \cup \{\infty\}$, the space of \emph{overconvergent arithmetic Hilbert modular forms of tame level $\mu_N$, $p$-level $\Gamma_{\!0}(p^n)$, polarisation ideal $\gothc$, radius of overconvergence $\e$ and weight $\kU$}, and its integral subspace, are then
		\begin{align*}
		M_{\k}^G(\Gamma_{\!0}(p^n),\mu_N,\e,\gothc):=\hH^0(\mathcal X_{G,\gothc,\cU}(\e),\omega^{\kappa}_{G,\gothc,n}),\\
		M_{\k}^{G,+}(\Gamma_{\!0}(p^n),\mu_N,\e,\gothc):=\hH^0(\mathcal X_{G,\gothc,\cU}(\e),\omega^{\kappa,+}_{G,\gothc,n}).
		\end{align*}

	\end{defn}
	\begin{remark}
		One can define cusp forms for $G$ in much the same way as for $G^*$ (Rem.~\ref{rem:cusp forms}).
	\end{remark}
	As in \cite[Lem.~4.5]{AIP}, the polarisation action by $p$-adic units $x \in F^{\times,+}$ defines isomorphisms  
	\[
	P_x:M_{\k}^G(\Gamma_{\!0}(p^n),\mu_N,\e,\gothc) \isorightarrow  M_{\k}^G(\Gamma_{\!0}(p^n),\mu_N,\e,x\gothc)
	\]
	We will see below that the action of the Hecke operators permutes the spaces $M_{\k}^G(\mu_N,\e,\gothc)$ by scaling the polarisation ideal. To obtain a Hecke stable space, one therefore  also defines:
	
	\begin{defn}\label{d:full-space-of-OCHMF-for-G}
		Let $\mathfrak{R}_F^{(p)}$ be the group of fractional ideals of $F$ coprime to $p$ and $\mathfrak{T}_F^{+,(p)}$ be the positive elements which are $p$-adic units, then $\mathfrak{R}_F^{(p)}/\mathfrak{T}_F^{+,(p)}$ is the narrow class group. We let 
		\[
		M_{\k}^G(\Gamma_{\!0}(p^n),\mu_N,\e):=\bigg ( \bigoplus_{\gothc \in \mathfrak{R}_F^{(p)} }  M_{\k}^G(\Gamma_{\!0}(p^n),\mu_N,\e,\gothc) \bigg)\bigg/  \big \langle P_x(f)-f \text{ for all }x \in \mathfrak{T}_F^{+,(p)}\big \rangle. 
		\] This is the space of arithmetic overconvergent Hilbert modular forms of weight $\k$, tame level $\mu_{\nn}$ and radius of overconvergence $\e$. One can define integral versions of these spaces by using $\omega^{\kappa,+}_{G,\gothc,n}$.
	\end{defn}

	Putting everything together, we obtain the desired comparison to the sheaf of Hilbert modular forms for $G$ as defined by Andreatta--Iovita--Pilloni. For this we first recall the definition:
	\begin{definition}\label{d:AIP-sheaf-G}
		Let $\w_{G,\gothc,\mathrm{AIP,n}}^{\k,+}:=(\pi_{\ast}\w_{G^*,\gothc,\mathrm{AIP,n}}^{\k,+})^{\Delta(N)}$. This is the integral analytic incarnation of the sheaf of Hilbert modular forms for $G$ on $\XX_{G,\gothc,\U}(\epsilon)$ denoted by $\omega^{\k_G^{\mathrm{un}}}$ in \cite[\S 8.2]{AIP3}.
	\end{definition}
	\begin{thm}\label{thm:comparison-hilbert-for-G}
		Let $\k: \U \to \W$ be a smooth bounded weight, $n \in \ZZ_{\geq 0} \cup \{\infty\}$ and $0\leq \epsilon\leq \epsilon_{\kappa}$. There is a natural isomorphism \[\w_{G,\gothc,n}^{\k,+} \cong \omega^{\k,+}_{G,\gothc,\AIP,n}\]
		of $\O^+$-modules on $\XX_{G,\gothc,\U}(\epsilon)$.
		In particular, $\w_{G,\gothc,n}^{\k,+}$ is an invertible $\O^+$-module. By inverting $p$, it induces a Hecke-equivariant isomorphism of line bundles $\w_{G,\gothc,n}^{\k}= \omega^{\k}_{G,\gothc,\AIP,n}$.
		
	\end{thm}
	
	\begin{proof}
		To ease notation, we suppress the dependence on $\gothc$. By Thm.~\ref{thm:comparison hilbert}, we have $\omega_{G^*}^{\k,+}=\omega_{G^*,\AIP}^{\k,+}$. By combining the lemmas of \S9, we conclude that
		\begin{align*}
		\w_{G,n}^{\k,+}&\stackrel{\eqref{d:mf-G-via-X_G}}{=}
		\w_{G,(G,\Gamma,n)}^{\k,+}\stackrel{\eqref{prop G sheaves}}{=} 
		\w_{G,(G^*,\Gamma,n)}^{\k,+}\stackrel{\eqref{lem:sheaf descent}}{=}
		(\pi_{\ast}\w_{G^*,(G^*,\Gamma,n)}^{\k,+})^{\Delta(N)}\\&\stackrel{\eqref{lem G* sheaves}}{=}
		(\pi_{\ast}\w_{G^*,(G^*,\Gamma^*,n)}^{\k,+})^{\Delta(N)}\stackrel{\eqref{thm:comparison hilbert} \&\eqref{lem: G forms from AIP}}{=}
		(\pi_{\ast}\w_{G^{\ast},\AIP,n}^{\k,+})^{\Delta(N)}\stackrel{\eqref{d:AIP-sheaf-G}}{=}\omega_{G,\AIP,n}^{\k,+}
		\end{align*}
		as desired.
		We postpone the proof of Hecke-equivariance to \S\ref{sec:hecke-operators}.
	\end{proof}
	\section{Hecke operators}\label{sec:hecke-operators}
	Throughout this section, we fix a smooth bounded weight $\kappa:\U\to \W \to \W^*$ and $0\leq \epsilon\leq \epsilon_{\kappa}$. To ease notation, in this section we suppress the subscript $\U$ from our Hilbert modular varieties.
	
	\subsection{The tame Hecke operators}
	Let $\mathfrak a\subseteq \O_F$ be any ideal coprime to $\mathfrak n$ and $p$. Let us denote by ${\mathcal X_{\gothc,\Gamma_{\!\scaleto{0}{3pt}}(\mathfrak a)}(\epsilon)} $ the Hilbert modular variety of tame level $\mu_N \cap \Gamma_{\!0}(\mathfrak a)$, representing tuples $(A,\iota,\lambda,\mu_N,D)$ where  $D\subseteq A[\mathfrak a]$  is a closed $\O_F$-submodule scheme that is \'etale locally isomorphic to $\O_F/\mathfrak a$, and $\lam : A\otimes \gothc \isorightarrow A^\vee$. Let us denote by $\pi_1: {\mathcal X_{\gothc,\Gamma_{\!\scaleto{0}{3pt}}(\mathfrak a)}(\epsilon)} \to {\mathcal X_{\gothc}(\epsilon)}$ the forgetful morphism. By Lem.~\ref{l:KL-(1.9.1)}, there is a second map  defined by sending $A\mapsto A/D$:
	\[\pi_2: {\mathcal X_{\gothc,\Gamma_{\!\scaleto{0}{3pt}}(\mathfrak a)}(\epsilon)_{a}} \to {\mathcal X_{\gothc\mathfrak a}(\epsilon)},\quad (A,\iota,\lambda,\mu_N,D)\to (A':=A/D,\iota',\lambda',\mu'_N).\]
	Now let $n\in \Z_{\geq 1}\cup \{\infty\}$. Then it is clear that the morphisms $\pi_1$ and $\pi_2$ extend uniquely when we add level structures at $p$: more precisely, we obtain a natural commutative diagram
	\begin{equation}\label{eq:tame-Hecke-operator}
	\begin{tikzcd}[row sep = 0.55cm]
	{\mathcal X_{\gothc \mathfrak a,\Gamma(p^\infty)}(\epsilon)_{a}} \arrow[d] & {\mathcal X_{\gothc,\Gamma(p^\infty)\cap \Gamma_{\!\scaleto{0}{3pt}}(\mathfrak a)}(\epsilon)_{a}} \arrow[l,"\pi_{2,\infty}"'] \arrow[r,"\pi_{1,\infty}"] \arrow[d] & {\mathcal X_{\gothc,\Gamma(p^\infty)}(\epsilon)_{a}} \arrow[d] \\
	{\mathcal X_{\gothc \mathfrak a,\Gamma_{\!\scaleto{0}{3pt}}(p^n)}(\epsilon)_{a}}              & {\mathcal X_{\gothc,\Gamma_{\!\scaleto{0}{3pt}}(p^{n})\cap \Gamma_{\!\scaleto{0}{3pt}}(\mathfrak a)}(\epsilon)_{a}} \arrow[l,"\pi_2"'] \arrow[r,"\pi_1"]                            & {\mathcal X_{\gothc,\Gamma_{\!\scaleto{0}{3pt}}(p^{n})}(\epsilon)_{a}},       
	\end{tikzcd}
	\end{equation}
	where $\pi_{2,\infty}$ sends $\alpha:\O_p^2\to T_pA^\vee$ to $\alpha':\O_p^2\xrightarrow{\a} T_pA^{\vee}\xrightarrow{(\varphi^\vee)^{-1}} T_pA'^\vee$, where the isomorphism $\varphi^\vee:T_pA'^\vee\to T_pA^\vee$ is the one induced from the prime-to-$p$ isogeny $\varphi^\vee:A'^\vee=(A/D)^\vee\to A^\vee$.
	\begin{lem}\label{l:hecke-corresp-T}
		There is a canonical isomorphism $\pi_1^{\ast}\omega_n^{\kappa,+}=\pi_2^{\ast}\omega_n^{\kappa,+}$ of sheaves on ${\mathcal X_{\gothc,\Gamma_{\!\scaleto{0}{3pt}}(p^{n})\cap \Gamma_{\!0}(\mathfrak a)}(\epsilon)_{a}}$. 
	\end{lem}
	\begin{proof}
		By diagram~\eqref{eq:tame-Hecke-operator}, we have
		\[\pi_1^{\ast}\omega_n^{\kappa,+}=\{ f\in \mathcal O^+_{\XX_{\gothc,\Gamma(p^\infty)\cap\Gamma_{\!\scaleto{0}{3pt}}(\mathfrak a)}(\epsilon)_a}\mid \gamma^{\ast}f = \pi_{1,\infty}^\ast\kappa^{-1}(c \mathfrak z+d)f \text{ for all }\gamma \in \Gamma_{\!0}(p^n) \}.\]
		
		The same applies to $\pi_2^\ast\omega_n^{\kappa,+}$, so we are left to see that $\pi_{1,\infty}^\ast\kappa^{-1}(c \mathfrak z+d)=\pi_{2,\infty}^\ast\kappa^{-1}(c \mathfrak z+d)$. To see this, observe that $\varphi:A\to A/D$ induces an isomorphism of Hodge--Tate sequences
		\begin{equation}\label{dg:functoriality-of-HT}
		\begin{tikzcd}[row sep = 0.55cm]
		\O_p^2\arrow[d,equal]\arrow[r, "\alpha"]&T_p(A/D)^\vee \arrow[r,  "\HT"] \arrow[d,  "\varphi^{\vee}"] & \omega_{A/D} \arrow[d,  "\varphi^{\ast}"]\\
		\O_p^2\arrow[r, "\alpha'"]&T_pA^\vee \arrow[r,  "\HT"] & \omega_A
		\end{tikzcd}
		\end{equation}
		by definition of $\alpha$. This shows that $\pi_{\HT}\circ \pi_{1,\infty}=\pi_{\HT}\circ \pi_{2,\infty}$, giving the desired equality.
	\end{proof}
	\begin{definition}\label{d:T_a}
		The $T_{\mathfrak a}$-operator is defined as the composition
		\begin{align*}
		&M^{G^*,+}_{\k}(\Gamma_{\!0}(p^n),\mu_N,\epsilon,\gothc\mathfrak a)=\Gamma(\mathcal X_{\gothc \mathfrak a,\Gamma_{\!\scaleto{0}{3pt}}(p^{n})}(\epsilon)_{a},\omega^{\kappa,+}_n)\to \Gamma(\mathcal X_{\gothc,\Gamma_{\!\scaleto{0}{3pt}}(p^{n})\cap \Gamma_{\!\scaleto{0}{3pt}}(\mathfrak a)}(\epsilon)_{a},\pi_2^{\ast}\omega^{\kappa,+}_n)\\
		&\isorightarrow \Gamma(\mathcal X_{\gothc,\Gamma_{\!\scaleto{0}{3pt}}(p^{n})\cap \Gamma_{\!\scaleto{0}{3pt}}(\mathfrak a)}(\epsilon)_{a},\pi_1^{\ast}\omega^{\kappa,+}_n)\xrightarrow{\frac{1}{q_{\mathfrak a}}\Tr_{\pi_1}}\Gamma(\mathcal X_{\gothc,\Gamma_{\!\scaleto{0}{3pt}}(p^{n})}(\epsilon)_{a},\omega^{\kappa,+}_n)=M^{G^*,+}_{\k}(\Gamma_{0}(p^n),\mu_N,\epsilon,\gothc).
		\end{align*}
		where $\Tr_{\pi_1}$ is the trace of the  finite locally free map $\pi_1$, and where $q_{\mathfrak a}:=|\O_F/\mathfrak a|$.
	\end{definition}
	\subsection{The $U_{\mathfrak p}$-operators}
	Let $n \in \ZZ_{\geq 1}$ and $\mathfrak p$ be a prime ideal of $\OO_F$ above $p$ of ramification index $e$. Set $l:=ne+1$.
	For the definition of the $U_{\fp}$-operator, we then use the moduli space $\mathcal X_{\gothc,\Gamma_{\!\scaleto{0}{3pt}}(p^n)\cap \Gamma_{\!\scaleto{0}{3pt}}(\mathfrak p^{l})}(\epsilon)_{a}\to \mathcal X_{\gothc}(\epsilon)$ which relatively represents the data of an anticanonical $\O_F$-submodule scheme $C\subseteq A[p^n]$ \'etale locally isomorphic to $\O_F/p^n\O_F$ together with an $\O_F$-submodule scheme $D\subseteq A[\mathfrak p^{l}]$ \'etale locally isomorphic to $\O_F/\fp^{l}$ such that $C[\fp^{en}]=D[\fp^{en}]$. In particular, $D$ is then anticanonical. There is a forgetful map $\pi_1:\mathcal X_{\gothc,\Gamma_{\!\scaleto{0}{3pt}}(p^n)\cap \Gamma_{\!\scaleto{0}{3pt}}(\mathfrak p^{l})}(\epsilon)_{a}\to \mathcal X_{\gothc,\Gamma_{\!\scaleto{0}{3pt}}(p^n)}(\epsilon)_a$ which is finite flat of degree $q_{\gothp}:=|\OF/\mathfrak p|$. There is also a second map
	\begin{alignat*}{3}
	\pi_2:\mathcal X_{\gothc,\Gamma_{\!\scaleto{0}{3pt}}(p^n)\cap \Gamma_{\!\scaleto{0}{3pt}}(\mathfrak p^{l})}(\epsilon)_{a}&\to&& \mathcal X_{\gothc \fp,\Gamma_{\!\scaleto{0}{3pt}}(p^n)}(\epsilon)_a\\ (A,\iota,\lambda,\mu_N,C,D)&\mapsto&& (A':=A/D[\mathfrak p],\iota',\lambda',\mu'_N,C':=(C+D)/D[\mathfrak p])
	\end{alignat*}
	where $(A',\iota',\lambda',\mu'_N)$ is like in Lem.~\ref{l:KL-(1.9.1)} and where $C+D\subseteq A[p^{n+1}]$ is the submodule scheme generated by $C$ and $D$. 
	Then $C'$ is \'etale locally isomorphic to $\O_F/p^n\O_F$. 
	We note that this map is not surjective, and the image is already contained in an open subspace that can be described using the partial Hasse invariant at $\fp$ (for example, if $F=\Q$, then it lands in $\cX_{p\gothc,\Gamma_{\!\scaleto{0}{3pt}}(p^n)}(p^{-1}\e)_a$).
	
	Let us now fix any uniformiser $\varpi \in \O_p$ such that $\varpi \O_p=\mathfrak p\O_p$ and let $u_{\fp}:=\smallmatrd{\varpi}{0}{0}{1}\in G(\Q_p)$. 
	Then letting  $u_{\fp}$ act in terms of the $G(\Q_p)$-action, we obtain a commutative diagram
	\begin{equation}\label{eq:wild-Hecke-operator}
	\begin{tikzcd}[row sep = 0.54cm]
	{\mathcal X_{\gothc\fp, \Gamma(p^\infty)}(\epsilon)_{a}} \arrow[d,"q"] & {\mathcal X_{\gothc,\Gamma(p^\infty)}(\epsilon)_{a}} \arrow[l,"u_{\mathfrak p}"'] \arrow[r,equal] \arrow[d] & {\mathcal X_{\gothc,\Gamma(p^\infty)}(\epsilon)_{a}} \arrow[d,"q"] \\
	{\mathcal X_{\gothc \fp,\Gamma_{\!\scaleto{0}{3pt}}(p^n)}(\epsilon)_{a}}              & {\mathcal X_{\gothc,\Gamma_{\!\scaleto{0}{3pt}}(p^n)\cap \Gamma_{\!\scaleto{0}{3pt}}(\mathfrak p^{l})}(\epsilon)_{a}} \arrow[l,"\pi_2"'] \arrow[r,"\pi_1"]                            & {\mathcal X_{\gothc,\Gamma_{\!\scaleto{0}{3pt}}(p^n)}(\epsilon)_{a}}.         
	\end{tikzcd}
	\end{equation}
	
	\begin{lem}\label{l:hecke-corresp-U}
		The action of $u_{\fp}$ defines a map $\pi_2^{\ast}\omega_n^{\kappa,+}\rightarrow\pi_1^{\ast}\omega_n^{\kappa,+}$ of invertible sheaves on ${\mathcal X_{\gothc,\Gamma_{\!\scaleto{0}{3pt}}(p^{n})\cap \Gamma_{\!0}(\mathfrak p^{l})}(\epsilon)_{a}}$. 
		It is independent of the choice of the uniformiser $\varpi$.
	\end{lem}
	\begin{proof}
		By diagram~\eqref{eq:wild-Hecke-operator}, we have
		\[\pi_1^{\ast}\omega_n^{\kappa,+}=\{ f\in \mathcal O^+_{\XX_{\gothc,\Gamma(p^\infty)}(\epsilon)_a}\mid \gamma^{\ast}f = \kappa^{-1}(c \mathfrak z+d)f \text{ for all }\gamma \in \Gamma_{\!0}(p^n)\cap \Gamma_{\!0}(\fp^{l}) \}.\]
		We claim that $u_{\mathfrak p}^{\ast}\omega_n^{\kappa,+}$ admits the same description. To see this, we first recall that the action of $u_{\fp}$ is equivariant with respect to the morphism $\Gamma_{\!0}(p^n)\cap \Gamma_{\!0}(\fp^{l})\to \Gamma_{\!0}(p^n)$ given by conjugation by $u_{\mathfrak p}$, namely $j:\smallmatrd{a}{b}{c}{d}\mapsto \smallmatrd{a}{\varpi b}{\varpi^{-1}c}{d}$. Second, we see from $G(\Q_p)$-equivariance of $\pi_{\HT}$ that  $u_\fp^{\ast}\mathfrak z = \varpi \mathfrak z$.
		Consequently, $\kappa(\gamma,\mathfrak z):=\kappa(c\mathfrak z+d)$ is sent by $u_{\mathfrak p}$ to $(\gamma,\mathfrak z)\mapsto \kappa(j(\gamma),\varpi\mathfrak z)=\kappa(\varpi^{-1}\varpi c \mathfrak z+d)=\kappa(c\mathfrak z+d)$. This together with the fact that shows that $\XX_{\gothc,\Gamma(p^\infty)}(\epsilon)_a\to X_{\gothc,\Gamma_{\!\scaleto{0}{3pt}}(p^{n})\cap \Gamma_{\!\scaleto{0}{3pt}}(\mathfrak p^{l})}(\epsilon)_{a}$ is a pro-\'etale $\Gamma_{\!0}(p^{n})\cap \Gamma_{\!0}(\mathfrak p^{l})$-torsor shows that $u_{\mathfrak p}^{\ast}\omega_n^{\kappa,+}$ has the desired form.
	\end{proof}
	\begin{definition}\label{d:U_fp}
		The $U_{\fp}$-operator is defined as the composition
		\begin{align*}
		M^{G^*}_{\k}&(\Gamma_{\!0}(p^n),\mu_N,\epsilon,\gothc\fp)=\Gamma(\mathcal X_{\gothc \fp,\Gamma_{\!\scaleto{0}{3pt}}(p^{n})}(\epsilon)_{a},\omega^{\kappa}_n)\to \Gamma(\mathcal X_{\gothc,\Gamma_{\!\scaleto{0}{3pt}}(p^{n})\cap \Gamma_{\!\scaleto{0}{3pt}}(\mathfrak p^{l})}(\epsilon)_{a},\pi_2^{\ast}\omega^{\kappa}_n)\\
		&\longrightarrow \Gamma(\mathcal X_{\gothc,\Gamma_{\!\scaleto{0}{3pt}}(p^{n})\cap \Gamma_{\!\scaleto{0}{3pt}}(\mathfrak p^{l})}(\epsilon)_{a},\pi_1^{\ast}\omega^{\kappa}_n)\xrightarrow{\frac{1}{q_{\mathfrak p}}\Tr_{\pi_1}}\Gamma(\mathcal X_{\gothc,\Gamma_{\!\scaleto{0}{3pt}}(p^{n})}(\epsilon)_{a},\omega^{\kappa}_n)=M^{G^*}_{\k}(\Gamma_{\!0}(p^n),\mu_N,\epsilon,\gothc).
		\end{align*}
	\end{definition}
	\subsection{The Hecke action on arithmetic Hilbert modular forms}
	For $\gothc$ a polarisation ideal, we let $[\gothc]$ denote its class in the narrow class group. One can then define Hecke operators on the spaces of overconvergent Hilbert modular forms for $G$ of the form
	\[T_{\mathfrak a}:M_{\kappa}^G(\Gamma_{\!0}(p^n),\mu_N,\epsilon,[\mathfrak c\mathfrak a])\to M_{\kappa}^G(\Gamma_{\!0}(p^n),\mu_N,\epsilon,[\mathfrak c])\]
	\[U_{\mathfrak p}:M_{\kappa}^G(\Gamma_{\!0}(p^n),\mu_N,\epsilon,[\mathfrak c\mathfrak p])\to M_{\kappa}^G(\Gamma_{\!0}(p^n),\mu_N,\epsilon,[\mathfrak c])\]
	by taking $\Delta$-invariants of the operators defined in the last section. Alternatively, one can define these operators more directly and without any reference to $G^{\ast}$ based on Defns.~\ref{defn: 4 sheaves for the price of 2} and Defn.~\ref{d:mf-G-via-X_G}; copying the definitions for $G^{\ast}$ and replacing $\XX$ by $\XX_{G}$ throughout. The proofs 
	go through without change. The natural morphisms of Hecke correspondences over the map $\XX\to \XX_G$ shows that the operators thus defined coincide with the ones obtained via $G^{\ast}$.
	
	It is clear from either definition that the Hecke operators commute with the polarisation action from Defn.~\ref{d:full-space-of-OCHMF-for-G}. Consequently, they induce a Hecke action on $M^{G}_{\kappa}(\Gamma_{\!0}(p^n),\mu_N,\epsilon)$.
	
	\begin{rmrk}
		Via the Koecher principle the Hecke operators also extend to the boundary. Moreover, the subspaces of cusp forms will be preserved by the action of Hecke operators.
	\end{rmrk}
	
	\begin{rmrk}
		As defined, the Hecke operators for $G^*$ are canonical. If one fixes a set of representatives $(\gothc_i)$ of the narrow class group and considers the Hecke operators as mapping between the spaces for these fixed $\gothc_i$, then the Hecke operators for $G^*$ become non-canonical, depending on the choice of representatives (cf. \cite[Section 4.3]{AIP}). For $G$, the operators remain canonical as picking different representatives does not affect our polarisation class. In particular, we get commuting Hecke operators on $M_k^G(\Gamma_{\!0}(p^n),\mu_N,\e)$.
		
		Lastly, one can check directly that for $G^*$ or $G$ the Hecke operator $U_p=\prod_{\gothp|p} U_{\gothp}^{e_i}$ is a compact operator. Alternatively, this follows from Prop.~\ref{prop: hecke equiv} together with \cite[Lem.~3.27]{AIP}.
	\end{rmrk}
	\begin{rmrk}
		It is clear from the definition that the tame Hecke operators preserve the integral spaces of overconvergent forms, while for $U_{\mathfrak p}$ this is in general only true after renormalisation.
	\end{rmrk}
	\subsection{Hecke-equivariance of the comparison}
	We can now finish the proof of Thms.~\ref{thm:comparison hilbert} and \ref{thm:comparison-hilbert-for-G} by proving that the comparison isomorphisms are Hecke equivariant.
	\begin{prop}\label{prop: hecke equiv}
		The isomorphisms $\omega^{\kappa,+}_{G^*,n}\isorightarrow \omega^{\kappa,+}_{G^*,\AIP,n}$ and $\omega^{\kappa,+}_{G,n}\isorightarrow \omega^{\kappa,+}_{G,\AIP,n}$ are Hecke equivariant on global sections.
	\end{prop}
	\begin{proof}
		We consider the case of $G^{\ast}$, the case of $G$ follows from this. As in the proof of Thm.~\ref{thm:comparison-hilbert-for-G}, we can assume that $\kappa$ has image in $\W_k^{\ast}$ for some $k>0$.
		It is clear from comparing Definitions~\ref{d:T_a}~and~\ref{d:U_fp} to the definition in \cite[\S8.5]{AIP3} that it suffices to prove that the isomorphism $\pi_{2}^{\ast}\omega_{G^{\ast},n}^{\kappa}=\pi_1^{\ast}\omega_{G^{\ast},n}^{\kappa}$ from Lem.~\ref{l:hecke-corresp-T}~and~\ref{l:hecke-corresp-U} is identified with the isomorphisms $\pi_{2}^{\ast}\omega_{G^{\ast},n,\AIP}^{\kappa}=\pi_1^{\ast}\omega_{G^{\ast},n,\AIP}^{\kappa}$ from \cite[Lem.~8.5]{AIP3} under the comparison isomorphism. For this it suffices to see that the comparison map $\mathfrak s$ from \S\ref{s:comparison-morphism-Hilbert} induces a morphism of Hecke correspondences
		\begin{equation}\label{eq:morph-of-Hecke-correspondence}
		\begin{tikzcd}[row sep = 0.54cm]
		{\mathcal X_{\gothc\fp, \Gamma(p^\infty)}(\epsilon)_{a}} \arrow[d,"\mathfrak s"] & {\mathcal X_{\gothc,\Gamma(p^\infty)}(\epsilon)_{a}} \arrow[l,"\pi_{2,\infty}"'] \arrow[r,"\pi_{1,\infty}"] \arrow[d,"\mathfrak s"] & {\mathcal X_{\gothc,\Gamma(p^\infty)}(\epsilon)_{a}} \arrow[d,"\mathfrak s"] \\
		\mathcal F_m(\epsilon)             &  \mathcal F_{m,\Gamma_{\!\scaleto{0}{3pt}}(\mathfrak a)}(\epsilon)  \arrow[l,"\pi_2"']   \arrow[r,"\pi_1"] & \mathcal F_{m}(\epsilon)
		\end{tikzcd}
		\end{equation}
		where $ \mathcal F_{m,\Gamma_{\!\scaleto{0}{3pt}}(\mathfrak a)}(\epsilon)$ is the pullback along   $\pi_1:{\mathcal X_{\gothc, \Gamma_{\!\scaleto{0}{3pt}}(\mathfrak a)}(\epsilon)}\to {\mathcal X_{\gothc}(\epsilon)}$ of the Andreatta--Iovita--Pilloni torsor, and where $\pi_2$ is induced by the map $\omega_{A/D}\to \omega_{A}$. Commutativity of the right square is clear. In terms of moduli, commutativity of the left square is now precisely that of diagram~\eqref{dg:functoriality-of-HT}.
		
		For the $U_{\fp}$-operator, the top left map in diagram is replaced by the action of $u_{\mathfrak p}=\smallmatrd{\varpi}{0}{0}{1}$. Let $A$ be the universal abelian variety over $\mathcal X_{\gothc,\Gamma_{\!\scaleto{0}{3pt}}(p^n)\cap\Gamma_{\!\scaleto{0}{3pt}}(\fp^{l})}(\epsilon)_a$ with its anticanonical subgroup $D\subseteq A[\fp^{l}]$, then the map $\pi_{2}^{\ast}\omega_{G^{\ast},n,\AIP}^{\kappa}=\pi_1^{\ast}\omega_{G^{\ast},n,\AIP}^{\kappa}$ is induced via the adjoint $\omega_{G^{\ast},n,\AIP}^{\kappa}\to \pi_{2,\ast}\pi_1^{\ast}\omega_{G^{\ast},n,\AIP}^{\kappa}$ obtained by restriction from the map $\varphi^{\ast}:\omega_{A/D[\fp]}\to \omega_{A}$ associated to the isogeny $A\to A/D[\fp]$.
		
		On the other hand, the morphism $\omega_{G^{\ast},n}^{\kappa}\to \pi_{2,\ast}\pi_1^{\ast}\omega_{G^{\ast},n}^{\kappa}$ in Lem.~\ref{l:hecke-corresp-U} is given by restriction of the action of $u_{\mathfrak p}$. To prove that the comparison is equivariant for the $U_{\mathfrak p}$-operator, it thus suffices to prove that these two morphisms commute with the comparison morphism $\mathfrak s$.

		Using Lem.~\ref{l:action-of-G(Q_p)} and the identity $u_{\fp}^\vee(1,0)=(1,0)$, we see that the following diagram commutes
		\begin{center}
			\begin{tikzcd}[row sep = 0.5cm]
			{(1,0)} & \O_p^2 \arrow[d, "u_{\mathfrak p}^{\vee}", equal] \arrow[r, "\alpha"] & T_pA^\vee \arrow[r, "\HT"] & \omega_{A} \\
			{(1,0)} \arrow[u, maps to] & \O_p^2 \arrow[r, "\alpha'"] & T_p(A/D[\mathfrak p])^\vee \arrow[r, "\HT"] \arrow[u, "\varphi^{\vee}"'] & \omega_{A/D[\mathfrak p]} \arrow[u, "\varphi^{\ast}"'].
			\end{tikzcd}
		\end{center}
		This shows that also $\mathfrak s\circ u_{\fp}=\varphi^{\ast}\circ \mathfrak s$, and thus $U_{\mathfrak p}$ commutes with the comparison isomorphism.
	\end{proof}
	
	\bibliographystyle{alpha}

	\bibliography{biblio(3)}

\end{document}